\documentclass[12pt,A4,reqno]{amsart}
\usepackage{amsfonts}
\usepackage{mathrsfs}
\usepackage[T1]{fontenc}
\usepackage{calligra}
\usepackage{amssymb}
\usepackage{amsmath,amscd}
\usepackage{color}
\usepackage{epsf}
\usepackage{graphicx}
\usepackage{xypic}

\theoremstyle{plain}
\newtheorem{thm}{Theorem}[section]
\newtheorem{lem}[thm]{Lemma}

\newtheorem{cor}[thm]{Corollary}

\theoremstyle{definition}
\newtheorem{defi}[thm]{Definition}
\newtheorem{rem}[thm]{Remark}
\newtheorem{exam}{Example}

\newcommand{\R}{\mathbb R}
\newcommand{\Z}{\mathbb Z}
\newcommand{\F}{\mathbb F}
\newcommand{\nn}{\vskip 0.2cm}
\newcommand{\n}{\vskip 0.1cm}

\begin{document}

\title [\ ] {Cubes and Generalized Real Bott Manifolds}

\author{Li Yu}
\address{Department of Mathematics and IMS, Nanjing University, Nanjing, 210093, P.R.China
  \newline
     \textit{and}
 \newline
    \qquad  Department of Mathematics, Osaka City University, Sugimoto,
     Sumiyoshi-Ku, Osaka, 558-8585, Japan}

 \email{yuli@nju.edu.cn}



 \thanks{
  This work is partially supported by
 the Japanese Society for
the Promotion of Sciences (JSPS grant no. P10018) and
 Natural Science Foundation of China (grant no.11001120).}


 \begin{abstract}
    We define a notion of facets-pairing structure and its seal space on
    a nice manifold with corners. We will study
   facets-pairing structures on any cube in detail and investigate
  when the seal space of a facets-pairing structure on a cube is a closed manifold.
  In particular, for any $n\times n$ binary matrix $A$ with zero diagonal,
  there is a canonical facets-pairing structure $\mathcal{F}_A$ on the $n$-dimensional
  cube. We will show that all the closed manifolds that we can obtain from the seal spaces of
  such $\mathcal{F}_A$'s are neither more nor less than all the generalized real Bott
  manifolds --- a special class of real toric manifolds introduced by Choi,
  Masuda and Suh.
  \end{abstract}

\maketitle

  \section{Introduction}

   The motivation of our study comes from the following
    examples.\vskip .2cm

    \begin{exam}
      If we glue the boundary of a square in
       the pictures in Figure~\ref{p:three-Panel-Structures} according to the colors and arrows on
        the edges, we will get
      torus $T^2$, Klein bottle $K^2$, real projective plane $\R P^2$ and sphere $S^2$
       respectively.
      Notice that $T^2$ and $K^2$ admit flat Riemannian metric.
   \end{exam}

       \begin{figure}
         \includegraphics[width=0.55\textwidth]{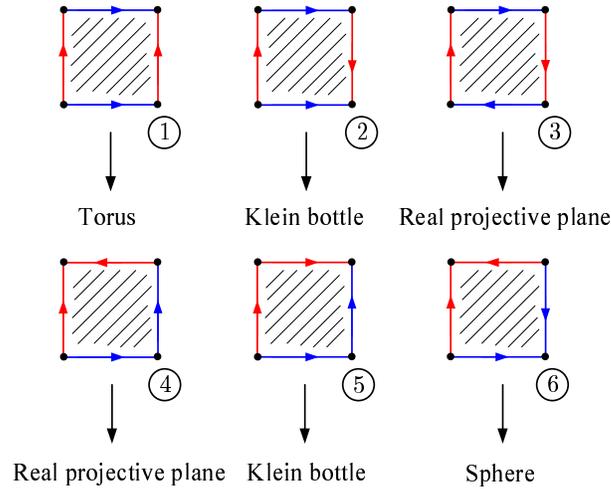}
          \caption{Facets-pairing structures on a square  }\label{p:three-Panel-Structures}
      \end{figure}

      From these examples, we may ask the following questions.\nn

 \noindent \textbf{Question 1:}
    How many different ways are there to pair all the
    facets (i.e. codimension-one faces) of an $n$-dimensional cube and how to classify
    them?\nn

\n

 \noindent \textbf{Question 2:}
   What closed manifolds can we obtain by gluing
  the facets of
  an $n$-dimensional cube according to different ways of pairing its facets
  and how to classify
  these closed manifolds up to homeomorphism? \vskip .2cm

  In this paper, the way of
   pairing the facets of a cube will be formulated into a
  general notion called \textit{facets-pairing structure} on any nice manifold with
  corners. But we will mainly study facets-pairing structures on
  cubes.
  In particular, given an $n\times n$ binary matrix $A$ with zero diagonal entries,
  we can define a canonical
  facets-pairing structure $\mathcal{F}_A$ on the $n$-dimensional
  cube. We will see that there is a nice correspondence between the geometric properties of
  $\mathcal{F}_A$ and algebraic properties of $A$, which allows us
  to answer the Question 1 and Question 2 for these
  facets-pairing structures.
   \nn

   The paper is divided into two parts. The first part is from section 2 to section 5.
   In section 2, we will introduce the concept
  of a \textit{facets-pairing structure} and its \textit{seal space}
  (see Definition~\ref{Def:Facets-Pairing-Struc} and Definition~\ref{Defi:Seal-Space})
    on a nice manifold with corners.
  Then we introduce two descriptive notions --- ``\textit{strong}'' and
  ``\textit{perfect}'' ---
  to describe a facets-pairing structure in different aspects
   (see Definition~\ref{Def:Strong-FP-Struc} and Definition~\ref{Def:Perfect-FP-Struc}).
    These two notions play a central role in our study throughout this
    paper.\nn

  In section 3, we will begin studying the facets-pairing structures
  on cubes. But mainly we will study those facets-pairing structures
  on cubes with some mild restrictions, called \textit{regular} facets-pairing structures
  (see Definition~\ref{Def:Regular-FP-Struc}).
  We will see in section 3 that any regular facets-pairing structure
   on an $n$-dimensional cube
    can be represented by $2n+1$ signed permutations on
   $\{ 1,-1,\cdots, n , -n\}$.  So the
     classification of regular facets-pairing structures on a cube
     is equivalent to a purely algebraic problem
  (see Theorem~\ref{thm:Main-2-Classify}).
   In section 4, we will study
  the relationship between the strongness and perfectness for regular facets-pairing
  structures on a cube.  The main result
   is: any perfect regular facets-pairing structure on a cube
  must be strong (see Theorem~\ref{thm:Perfect-Strong}).
  In section 5, we will show that the seal space of any perfect
   regular facets-pairing structure on a cube is
  always a closed manifold which admits a flat Riemannian metric.
  So we can construct many interesting
   closed flat Riemannian manifolds from cubes in this way. Hopefully,
   this construction will help us to understand the
   topology of these flat Riemannian manifolds. \nn

   The second part of this paper is from section 6 to section 9.
   In this part, we will use the theory developed in the first part to
   study a special class of regular
    facets-pairing structures $\mathcal{F}_A$
   on a cube, where $A$ is a binary square matrix with zero diagonal.
   In section 6,
    we will study the relationship between
    such $\mathcal{F}_A$'s and real Bott manifolds. We find that real Bott
     manifolds are exactly the seal spaces of all those
     $\mathcal{F}_A$'s which are perfect (see Theorem~\ref{thm:Main-4}).
   In addition, we obtain a simple algebraic test on the matrix $A$
   which tells us when the facets-pairing structure $\mathcal{F}_A$ is strong (see
   Theorem~\ref{thm:Main-5}).
    In section 7, we show that the seal space of an arbitrary
    $\mathcal{F}_A$ can be constructed via another method
    called \textit{glue-back construction}.
    This relates our study of facets-pairing structures to toric topology.\nn

    In section 8, we will study the singularity that might
    occur in any glue-back construction. The discussion will help us to
    determine whether the seal space of a given $\mathcal{F}_A$ is a closed manifold
    directly from the matrix $A$ (see Theorem~\ref{thm:Main-8}).
   In section 9, we will see how to view any generalized real
  Bott manifold as the seal space of some $\mathcal{F}_A$.
  A somewhat unexpected result is that the generalized real Bott manifolds
  are exactly all the closed manifolds that we can obtain from the seal spaces of
  these $\mathcal{F}_A$'s
  (see Theorem~\ref{thm:Main-6} and Theorem~\ref{thm:Main-7} and the summary at the end).
  So this gives another reason why generalized real Bott manifolds
  are naturally the ``extension'' of real Bott manifolds. In
  addition, this new viewpoint on generalized real Bott manifolds
  should be useful for us to study the topology of
  generalized real Bott manifolds in the future.\\

 \section{Facets-pairing Structure}
     A manifold with corners $W^n$ is called \textit{nice} if any codimension-$l$ face
     of $W^n$ meets exactly $l$ different
    \textit{facets} (i.e. codimension-one faces) of $W^n$. In this paper, we will use
   $\mathcal{S}_F(W^n)$ to denote the set of all facets of $W^n$. \nn

   \begin{rem}
        A nice manifold with corners may have no $0$-dimensional
      faces (vertices) and its faces are not necessarily contractible.
     See~\cite{Da83} or~\cite{BP02} for a detailed introduction
     to manifolds with corners and related
     concepts.
     \nn
   \end{rem}

   To make the words ``pairing all the facets of an $n$-dimensional cube''
   in the Question 1 have more strict mathematical meaning, we
  formulate the following concept on any nice manifold with
  corners.\nn

 \begin{defi}[Facets-Pairing Structure]
      \label{Def:Facets-Pairing-Struc}
          Suppose $W^n$ is a nice manifold with corners (may not be connected)
          and suppose all the facets of $W^n$ satisfy the following
        two conditions:\n
           \begin{enumerate}

            \item[(I)] each facet $F\subset \partial W^n$ is uniquely paired with a
            facet $F^*\subset \partial W^n$ (it is possible that $F^*=F$)
         and there are face-preserving
            homeomorphisms $\tau_F : F \rightarrow F^*$ and
            $\tau_{F^*} : F^* \rightarrow F$ such that $\tau_{F^*} =
            \tau^{-1}_F$ (here $F$ and $F^*$ themselves are considered as
            manifolds with corners).
             If $F^* \neq  F$, we call $\widehat{F} = \{ F, F^* \}$
             a \textit{facet pair} and call $F^*$ the \textit{twin facet} of
             $F$. If $F^* = F$, the
              $\tau_F : F \rightarrow F$ is necessarily
             an involution on $F$ (i.e. $\tau_F\circ \tau_F = id_F$). Then
             we define $\widehat{F}=\{ F \}$ and call such an $F$ a \textit{self-involutive} facet.
            \n

            \item[(II)] For any $x \in F_1\cap F_2$,
            if $\tau_{F_1}(x) \in F_1^* \cap F_3$, $\tau_{F_2}(x) \in F^*_2 \cap F_4$,
              then $\tau_{F_3}\tau_{F_1}(x) =
                 \tau_{F_4}\tau_{F_2}(x) \in F^*_3 \cap F^*_4$ (see
                 Figure~\ref{p:Local_Commute}). Here
                 we also allow $F_3=F^*_2$ or $F_4 = F^*_1$.\n
            \end{enumerate}

            Then we call $\mathcal{F}=\{ \widehat{F}, \tau_F \}_{F\subset \partial W^n}$
            a \textit{facets-pairing structure} on $W^n$, and call
            $\{ \tau_F : F \rightarrow F^* \}_{F\subset \partial W^n}$
      the \textit{structure maps of $\mathcal{F}$}.

        \begin{figure}
         \includegraphics[width=0.22\textwidth]{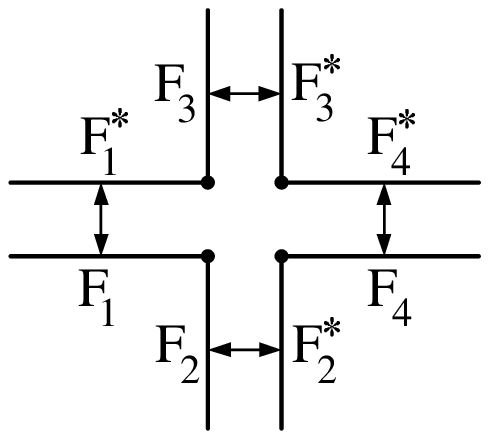}
          \caption{ }\label{p:Local_Commute}
      \end{figure}

  \end{defi}
    \vskip .2cm

       \begin{rem}
           The condition(II) on
           the structure maps $\{ \tau_F : F \rightarrow F^* \}_{F\subset \partial
           W^n}$ in the above definition is a bit special.
            It will exclude many well-known examples of gluing the boundary of
            a manifold with corners from our study. For example, if we pair any facet of
            a dodecahedron with its opposite facet
          by the minimal clockwise twist and glue up the paired facets
          accordingly, we will get the \textit{Poincar\'e homology sphere}.
          But this way of pairing the facets does not meet the condition(II), so it is
           not a facets-pairing structure on the dodecahedron.
       \end{rem}
       \n

        For any proper face $f$ of $W^n$,
        suppose $F_1$ is a facet of $W^n$ with
         $f\subset F_1$, then $\tau_{F_1}(f)$ is also a face of $W^n$
         because $\tau_{F_1}: F_1 \rightarrow F^*_1$ is a face-preserving map.
          Let $F_2$ be a facet so that
         $\tau_{F_1}(f) \subset F_2$. Then we get another face
         $\tau_{F_2}(\tau_{F_1}(f))$ of $W^n$ and so on. In general,
         an expression $\tau_{F_k} \circ \cdots \circ \tau_{F_1}(f)$ is called \textit{valid}
         if $f\subset F_1$ and $\tau_{F_j} \circ \cdots \circ \tau_{F_1}(f) \subset F_{j+1}$ for
         each $1\leq j < k$.
          Moreover, by an abuse of notation, when $k=0$, we define
          $\tau_{F_k} \circ \cdots \circ \tau_{F_1}(f) := f$.\nn

      \begin{defi} [Face Family] \label{Defi:Face-Family}
        Suppose $\mathcal{F}=\{ \widehat{F}, \tau_F \}_{F\subset \partial W^n}$
         is a facets-pairing structure
          on $W^n$.
            For any proper face $f$ of $W^n$,
            let $\widehat{f}$ be the set of all faces
             of the valid form $\tau_{F_k} \circ \cdots \circ \tau_{F_1}(f)$
              for some $k\geq 0$.
             We call $\widehat{f}$ the \textit{face family} containing $f$ in $\mathcal{F}$.
              In particular, the face family containing a facet $F$ is just
               $\widehat{F}$. Obviously, each face
              of $W^n$ is contained in a unique face family of $\mathcal{F}$.
              In addition, for a point $x$ in the relative interior of $f$,  any
              point of the valid form $\tau_{F_k} \circ \cdots \circ \tau_{F_1}(x) \in W^n$ is
              called a \textit{family point} of $x$.
       \end{defi}
        \nn

    \begin{defi}[Seal Space] \label{Defi:Seal-Space}
      For any nice manifold with corners $W^n$ equipped with
      a facets-pairing structure $\mathcal{F}$,
       let $Q^n_{\mathcal{F}}$ denote the quotient space of $W^n$
       with respect to the gluing relation
      $\{ x \sim \tau_F(x) ;  \ \text{for any facet}\ F\subset
        W^n\ \text{and}\ \forall\, x\in F  \}$.  In other words, $Q^n_{\mathcal{F}}$ is obtained by gluing all the
         family points of any point in $W^n$ together. Let
         $\zeta_{\mathcal{F}} : W^n \rightarrow Q^n_{\mathcal{F}}$
         be the corresponding quotient map.
        $Q^n_{\mathcal{F}}$ is called the \textit{seal space of
        $\mathcal{F}$} and $\zeta_{\mathcal{F}}$ is called the \textit{seal map} of $\mathcal{F}$.
     \end{defi}
       \n

    \begin{exam}
      In Figure~\ref{p:three-Panel-Structures},
      the picture 4 is not a facets-pairing structure because it does not meet the condition(II)
      in Definition~\ref{Def:Facets-Pairing-Struc}.
      The other five pictures give us five different facets-pairing structures on
    the square whose seal spaces are all closed manifolds.
   \end{exam}
   \n

   \begin{exam}[Trivial Facets-Pairing Structure]
    \label{Exam:Trivial}
       For a nice manifold with corners $W^n$, if we define $F^*=F$ and
        $\tau_F = id_F$ for each facet $F$ of $W^n$, what we get is
        obviously a facets-pairing structure on $W^n$. We call it
         the \textit{trivial facets-pairing structure} on $W^n$.
      The corresponding seal space is just $W^n$ itself.
    \end{exam}
    \n

   \begin{exam} \label{Exam:Strange}
      We define a facets-pairing structure $\mathcal{F}$ on a $3$-dimensional
      cube $W^3$ centered at the origin $O$ in Figure~\ref{p:Cube_Invol-2} by:
         \begin{itemize}
         \item[(i)]
         $\tau_{F_1} : F_1 \rightarrow F^*_1$
         sends any point $(x_1,x_2,x_3)\in F_1$ to $(-x_1,x_2,x_3) \in F^*_1$; \n

       \item[(ii)]  $\tau_{F_2} : F_2 \rightarrow F^*_2$ sends any
        point $(x_1,x_2,x_3)\in F_2$ to $(x_1,-x_2,-x_3) \in F^*_2$. \n

       \item[(iii)] $\tau_{F_3} : F_3 \rightarrow F_3^*$  sends any point
         $(x_1,x_2,x_3)\in F_3$ to $(-x_1,-x_2,-x_3) \in F^*_3$.
     \end{itemize}

      \begin{figure}
         \includegraphics[width=0.3\textwidth]{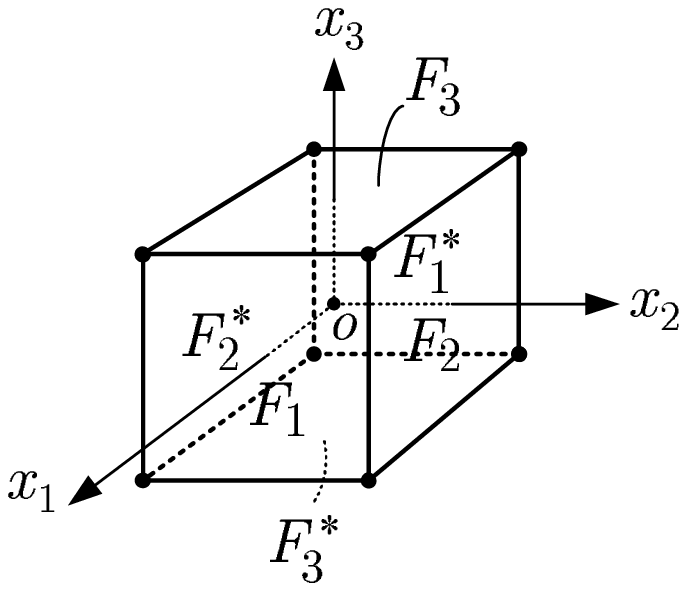}
          \caption{  }\label{p:Cube_Invol-2}
      \end{figure}

      \nn
    There are two $0$-dimensional face families in $\mathcal{F}$, each of which consists
    of four vertices of the cube. And there are four
      $1$-dimensional face families in $\mathcal{F}$, two of them consist of four edges
      each and the other two consist of two edges each.
   In addition, notice that $\tau_{F_2} (F_2\cap F_3) =
   \tau_{F_2} (F_2\cap F_3) = F^*_2 \cap F^*_3$. But
    for a point $p =(x_1,x_2,x_3) \in F_2\cap F_3$,
   $\tau_{F_2}(p) = (x_1, -x_2, -x_3)$ while $\tau_{F_3}(p) = (-x_1,-x_2,-x_3)$.
   So the seal map
     $\zeta_{\mathcal{F}}: W^3 \rightarrow Q^3_{\mathcal{F}}$ will
    glue two different points $\tau_{F_2}(p)$ and $\tau_{F_3}(p)$ in the interior of
         $F^*_2 \cap F^*_3$ into one point in $Q^3_\mathcal{F}$.
  \end{exam}

 To distinguish the facets-pairing structures with the kind of ``bad'' seal map
  as in Example~\ref{Exam:Strange} from other
 facets-pairing structures,  we introduce some extra notions as
 following.\nn

 For any proper face $f$ of $W^n$, let $\Xi(f)$
 denote the set of facets of $W^n$ that contain $f$, i.e.
 $\Xi(f) = \{ F \, | \, \text{$F$ is a facet of $W^n$ with  $f\subset F$}\}$.
 And let $\Xi^{\perp}(f)$ be the set of facets of $W^n$ that intersect $f$
 transversely. Since $W^n$ is a nice manifold with corners,
 if a facet $F$ of $W^n$ intersects $f$ transversely, then $F\cap f$ must be a
 codimension-one face of $f$. So we have:
   $$\Xi^{\perp}(f) = \{ F \, | \, \text{$F$ is a facet of $W^n$ with $f\cap
   F$ being a codimension-one face of $f$}\}.$$

 Choose an arbitrary facet $F\in \Xi(f)$ and let $f'= \tau_F (f)$.
 Then we have a map
  \begin{equation} \label{Equ:Psi-F}
    \Psi^{f}_F : \Xi(f) \rightarrow \Xi(f').
  \end{equation}
  where $ \Psi^{f}_F(F) = F^*$ and for $\forall\, F' \in \Xi(f)$, $ \Psi^{f}_F(F') \cap F^* = \tau_F (F'\cap F)$.
   Similarly, we define a map
    \begin{equation} \label{Equ:Psi-F-Perp}
    (\Psi^{f}_F)^{\perp} : \Xi^{\perp}(f) \rightarrow \Xi^{\perp}(f').
  \end{equation}
  where for any $F' \in \Xi^{\perp}(f)$,
  $(\Psi^{f}_F)^{\perp}(F') \cap f' = \tau_F (F'\cap f) $.
  Since $\tau_F : F \rightarrow F^*$ is a face-preserving
  homeomorphism, so $\Psi^f_F$ and $ (\Psi^{f}_F)^{\perp}$ are both
  bijections.\nn

  For example, for any facet $F$ of $W^n$, $\Xi(F) =\{ F \}$. By definition,
  $\tau_F(F) = F^*$, so $\Psi^F_F : \Xi(F)=\{ F \} \rightarrow \Xi(F^*) = \{ F^* \}$
  is unique.  \nn

  Moreover, for any facet $F'\in \Xi(f')$, let $f''=
  \tau_{F'}(f)$. So we have the composite:
 \[
   \Psi^{f'}_{F'} \circ \Psi^{f}_F : \Xi(f) \rightarrow \Xi(f'')\ \,
   \text{and}\ \,
   (\Psi^{f'}_{F'})^{\perp} \circ (\Psi^{f}_F)^{\perp} : \Xi^{\perp}(f)
   \rightarrow \Xi^{\perp}(f'').
  \]

   \begin{defi}[Strong Facets-Pairing Structure] \label{Def:Strong-FP-Struc}
    Suppose $\mathcal{F}=\{ \widehat{F}, \tau_F \}_{F\subset \partial W^n}$ is
     a facets-pairing structure on a nice manifold with corners $W^n$.
     If for any face $f$ of $W^n$,  whenever $\tau_{F_k} \circ \cdots \circ \tau_{F_1} (f) =
       \tau_{F'_{r}} \circ \cdots \circ \tau_{F'_1} (f) \overset{\vartriangle}{=} \widetilde{f}$,
       we then have:
       \begin{enumerate}
         \item[(a)]  $\tau_{F_k} \circ \cdots \circ \tau_{F_1}(p) =
         \tau_{F'_{r}} \circ \cdots \circ \tau_{F'_1}(p)$ for any point $p \in f$, \n

         \item[(b)] let $f_1= f'_1 :=f $, $f_{i+1} := \tau_{F_i} \circ \cdots \circ \tau_{F_1}
         (f)$ for any $1\leq i \leq k$, and
         $f'_{j+1} :=  \tau_{F'_j} \circ \cdots \circ \tau_{F'_1} (f)$ for any $1\leq j \leq r$.
         So $f_{k+1} = f'_{r+1} = \widetilde{f}$. Then
          \[ \Psi^{f_k}_{F_k} \circ \cdots \circ \Psi^{f_1}_{F_1} =
       \Psi^{f'_r}_{F'_r} \circ \cdots \circ \Psi^{f'_1}_{F'_1} :
             \Xi(f) \rightarrow \Xi(\widetilde{f})  \]
       \end{enumerate}
       we call
      $\mathcal{F}$ a \textit{strong facets-pairing structure} on $W^n$.
    Here if $k=0$, we define $\tau_{F_k} \circ \cdots \circ \tau_{F_1}: f\rightarrow f$
    to be the identity map $id_f$ and $\Psi^{f_k}_{F_k} \circ \cdots \circ \Psi^{f_1}_{F_1}:
     \Xi(f) \rightarrow \Xi(f)$ to be the identity map $id_{\Xi(f)}$.
     Notice that (a) implies
      \[ (\Psi^{f_k}_{F_k})^{\perp} \circ \cdots \circ (\Psi^{f_1}_{F_1})^{\perp} =
         (\Psi^{f'_r}_{F'_r})^{\perp} \circ \cdots \circ (\Psi^{f'_1}_{F'_1})^{\perp} :
             \Xi^{\perp}(f) \rightarrow \Xi^{\perp}(\widetilde{f}).  \]
   \end{defi}
 \nn

   \begin{rem} \label{Rem:Strong-Self-Invol-Facets}
       In a strong facets-pairing structure
       $\mathcal{F}=\{ \widehat{F}, \tau_F \}_{F\subset \partial W^n}$,
      if $F$ is a self-involutive facet, i.e. $\tau_F (F) = F$, then
      $\tau_F$ must be $id_F$. \nn
  \end{rem}

  By the above definition, the trivial facets-pairing structure is strong.
  But the facets-pairing structure defined in Example~\ref{Exam:Strange} is not strong because
   $\tau_{F_2} (F_2\cap F_3) = \tau_{F_3} (F_2\cap F_3)$, but
    $\tau_{F_2}$ does not agree with  $\tau_{F_3}$
   on $F_2\cap F_3$. \nn

   The following lemma is an immediate consequence of Definition~\ref{Def:Strong-FP-Struc}.
   \nn

   \begin{lem}
    If $\mathcal{F}$ is a strong facets-pairing structure on $W^n$,
    then for any face $f$ of $W^n$, the seal map
   $\zeta_{\mathcal{F}} : W^n \rightarrow Q^n_{\mathcal{F}}$
    will map the interior of $f$ injectively into
    the seal space $Q^n_{\mathcal{F}}$.
    \end{lem}
  \begin{proof}
     If two different points $p$ and $p'$ in
    the interior of $f$ are identified under $\zeta_{\mathcal{F}}$, we must have
   $p'= \tau_{F_k} \circ \cdots \circ \tau_{F_1} (p)$ for some
    sequence of facets $F_1,\cdots, F_k$. This implies that
    $\tau_{F_k} \circ \cdots \circ \tau_{F_1}(f)=f$. Since
     $id_f$ also sends $f$ to $f$,
    so by the strongness of $\mathcal{F}$,
     $\tau_{F_k} \circ \cdots \circ \tau_{F_1} |_f = id_f$.
    In particular, $\tau_{F_k} \circ \cdots \circ \tau_{F_1}(p)=p$.
     But this contradicts $p'\neq p$.
    \end{proof}
       \nn

      Notice that the definition of facets-pairing structure does not
      tell us how many different faces there are in a face
       family. In fact, in Example~\ref{Exam:Strange} we see that
     two face families of the same dimension might have different number of faces.

     \vskip .2cm

         \begin{defi}[Perfect Facets-Pairing Structure] \label{Def:Perfect-FP-Struc}
          Suppose $\mathcal{F}$ is
         a facets-pairing structure on a nice manifold with corners
         $W^n$. For a codimension-$l$ face $f$ of $W^n$, if the face family $\widehat{f}$
              consists of exactly $2^l$ different faces,
             we call $\widehat{f}$ a \textit{perfect face family} in $\mathcal{F}$.
              A facets-pairing structure $\mathcal{F}$ is called \textit{perfect}
              if all its face families in all dimensions are perfect. Note that
              a perfect facets-pairing structure has no self-involutive
              facets.\nn

         A strong facets-pairing structure may not
          be perfect (e.g. the facets-pairing structure on the square defined by
         picture 3 and 6 in Figure~\ref{p:three-Panel-Structures}).
         So we have the following strict hierarchy of notions.
     \end{defi}
   \nn

  facets-pairing structures $\, \supsetneq\,$ strong facets-pairing structures
  $\,\supsetneq\, $ strong and perfect facets-pairing structures. \nn
   \n

   \begin{rem}
    In a perfect facets-pairing structure $\mathcal{F}$ on $W^n$,
    it is possible that two faces in the same face family have nonempty
    intersection. For example, the picture 5 in Figure~\ref{p:three-Panel-Structures} defines a
     perfect facets-pairing structure $\mathcal{F}$ on a square where any
     facet of the square is adjacent to its
     twin facet.\\
   \end{rem}

  \section{Facets-pairing structures on a cube} \label{Section3}

   Although the notion of facets-pairing structure makes sense for
   an arbitrary nice manifold with corners $W^n$, if $W^n$ has no
  symmetry, such a structure might just be trivial (see Example~\ref{Exam:Trivial}).
   So in the rest of this paper, we will restrict our attention to solid cubes
    on which we can construct a lot of different facets-pairing structures.
  Our main aim in this paper is to understand those facets-pairing structures on cubes
  whose seal spaces are homeomorphic to closed manifolds.
   To make our discussion convenient, we introduce the following
   notations.\nn

  Let $[\pm n ] := \{ \pm 1 , \cdots, \pm n \}
       = \{ 1, -1, 2, -2, \cdots, n, -n \}$.
      A map $\sigma: [\pm n] \rightarrow [\pm n]$ is called a
      \textit{signed permutation} on $[\pm n ]$ if $\sigma$ is a bijection and $\sigma(-k) = - \sigma
      (k)$ for any $k\in [\pm n]$. The set of all signed permutations on $[\pm n]$ with respect to
      the composition of maps forms a group, denoted by $\mathfrak{S}^{\pm}_n$.  \vskip .2cm

    Let $\mathcal{C}^n$ denote the following $n$-dimensional cube
   in the Euclidean space $\R^n$.
   \[  \mathcal{C}^n := \{ (x_1,\cdots, x_n) \in \R^n \, | \,  -\frac{1}{4} \leq x_i \leq \frac{1}{4},
     \, 1\leq \forall\, i \leq n  \} \]
    For any $1\leq i \leq n$,
  let $\mathbf{F}(i)$ and $\mathbf{F}(-i)$ be the facets
    of $\mathcal{C}^n$ which lie in the hyperplane $\{ x_i = \frac{1}{4}
    \}$ and $\{ x_i = -\frac{1}{4} \}$ respectively.  Moreover, for any
    $j_1, \cdots, j_s \in [\pm n]$ whose absolute values $|j_1|,\cdots, |j_s|$ are pairwise distinct,
     we define
    \[   \quad \mathbf{F}(j_1,\cdots ,j_s) := \mathbf{F}(j_1) \cap \cdots \cap  \mathbf{F}(j_s)
                \subset \mathcal{C}^n. \]
       Then $\mathbf{F}(j_1,\cdots ,j_s)$ is a
        face of $\mathcal{C}^n$ with codimension $s$. Conversely,
       for any proper codimension-$s$ face $f$ of
       $\mathcal{C}^n$, there exists  $j_1, \cdots, j_s \in [\pm n]$ so that
       $\mathbf{F}(j_1,\cdots ,j_s)$ equals $f$.  We call such an $\mathbf{F}(j_1,\cdots ,j_s)$
       a \textit{normal form} of $f$. If two normal forms
       $\mathbf{F}(j_1,\cdots ,j_s)$ and  $\mathbf{F}(j'_1,\cdots ,j'_s)$
       denote the same face, we write $\mathbf{F}(j_1,\cdots ,j_s) =
       \mathbf{F}(j'_1,\cdots ,j'_s)$. Obviously, we have: \n
       \begin{itemize}
         \item  $\mathbf{F}(j_1,\cdots ,j_s) = \mathbf{F}(j'_1,\cdots ,j'_s)$
      if and only if $\{ j_1,\cdots ,j_s \} = \{ j'_1,\cdots ,j'_s \}$.\n

         \item $\Xi(\mathbf{F}(j_1,\cdots ,j_s)) = \{ \mathbf{F}(j_1),\cdots,
         \mathbf{F}(j_s)\}$, and \n

         \item $\Xi^{\perp}(\mathbf{F}(j_1,\cdots ,j_s)) =
         \{ \mathbf{F}(k) \, ;\, k \in [\pm n] \backslash
          \{ \pm j_1,\cdots, \pm j_s \} \}$.
       \end{itemize}

    \nn

  \textbf{Fact:} The symmetry group of $\mathcal{C}^n$
         is isomorphic to the signed permutation group $\mathfrak{S}^{\pm}_n$. This is because
         each symmetry of $\mathcal{C}^n$ is uniquely determined by
         its permutation on the set of $2n$ facets $\{ \mathbf{F}(j) \}_{j\in [\pm n]}$ of
         $\mathcal{C}^n$. \nn

   Suppose $\mathcal{F}$ is a facets-pairing structure
   on $\mathcal{C}^n$.
   For any facet $\mathbf{F}(j)$ of $\mathcal{C}^n$, let the twin facet of $\mathbf{F}(j)$
   in $\mathcal{F}$ be $\mathbf{F}(\omega(j))$ where $\omega(j)\in [\pm n]$.
    Then the facet-pair
   $\widehat{\mathbf{F}}(j) = \mathbf{F}(j) \cup \mathbf{F}(\omega(j))$.
   Since $\omega(\omega(j)) = j$, so
  $\omega$ is an involutive permutation on $[\pm n]$.
   Note that $\omega(j)=j \Longleftrightarrow \mathbf{F}(j)$ is
   self-involutive.\nn

 The structure maps of $\mathcal{F}$ are
  a collection of face-preserving homeomorphisms
    $$\{ \tau_j \overset{\vartriangle}{=} \tau_{\mathbf{F}(j)}
    : \mathbf{F}(j) \rightarrow \mathbf{F}(\omega(j))\}_{j\in [\pm n]}$$
    which
    satisfy the conditions in
    Definition~\ref{Def:Facets-Pairing-Struc}.
    Therefore, $\mathcal{F}$ can be formally represented
     by $\{ \omega, \tau_j \}_{j\in [\pm n]}$. \nn

    Moreover, by the following lemma we can assume that
    each $\tau_j: \mathbf{F}(j) \rightarrow \mathbf{F}(\omega(j))$
     is an isometry with respect to the natural Euclidean metric on
     $\mathbf{F}(j)$ and $\mathbf{F}(\omega(j))$ up to face-preserving isotopy. \vskip .2cm

   \begin{lem} \label{Lem:Isotopy}
     Any face-preserving homeomorphism $\tau: \mathcal{C}^{n} \rightarrow \mathcal{C}^{n}$
     is isotopic to an element of symmetry of $\mathcal{C}^{n}$ via face-preserving isotopy.
     \end{lem}
     \begin{proof}
        We proceed our proof by induction on the dimension $n$.
        When $n=0$, claim is obviously true. Now assume the statement is true in dimension
        less than $n$. Then given a face-preserving homeomorphism
        $\tau: \mathcal{C}^{n} \rightarrow \mathcal{C}^{n}$,
         suppose $\tau$ maps the facet $\mathbf{F}(j)$ to $\mathbf{F}(\sigma(j))$
         where $j, \sigma(j) \in  [\pm n]$. Since $\tau$ maps
         disjoint faces to disjoint faces, so $\sigma (-j) = -
         \sigma(j)$, i.e. $\sigma$ is an element of
         $\mathfrak{S}^{\pm}_n$.
        In addition, for any face $\mathbf{F}(j_1,\cdots ,j_s)$ of $\mathcal{C}^n$,
        $\tau(\mathbf{F}(j_1,\cdots ,j_s))= \mathbf{F}(\sigma(j_1),\cdots ,
        \sigma(j_s))$. So
        the permutation on the face lattice of $\mathcal{C}^n$ induced by $\tau$ is completely
        determined by $\sigma$. \n

       Since the symmetry group of $\mathcal{C}^n$
         is isomorphic to $\mathfrak{S}^{\pm}_n$, so there exists a unique
         symmetry $h_{\tau}$ of $\mathcal{C}^n$ so that $h_{\tau}$
         maps $\mathbf{F}(j)$ to $\mathbf{F}(\sigma(j))$. Then $h_{\tau}$
         induces the same permutation on the face lattice of $\mathcal{C}^n$ as $\tau$.
         \n

        Next, by the induction hypothesis, we can isotope $\tau$ on the boundary of $\mathcal{C}^{n}$ from
        $0$-dimensional faces to $(n-1)$-dimensional faces by face-preserving isotopy
         so that after the
        isotopy, we get a new face-preserving homeomorphism
        $\tau': \mathcal{C}^{n} \rightarrow \mathcal{C}^{n}$ whose
       restriction to any proper face of $\mathcal{C}^n$ is an isometry.
        Moreover, $\tau'$ induces the same permutation on the
         face lattice of $\mathcal{C}^n$ as $\tau$. So $\tau'$ must
         agree with $h_{\tau}$ on the boundary of $\mathcal{C}^n$ since they induce
         the same permutation on the face lattices of any proper face of
          $\mathcal{C}^n$. This means that
          $h^{-1}_{\tau}\circ \tau'$ fixes the boundary of $\mathcal{C}^n$.
          \n

       By the Alexander's lemma below, we can isotope $h^{-1}_{\tau}\circ\tau'$ to the identity
        map of the whole $\mathcal{C}^{n}$ by an isotopy fixing $\partial \mathcal{C}^{n}$.
        So $\tau'$ is isotopic to $h_{\tau}$, so is $\tau$. Moreover, since all the
        isotopies of $\mathcal{C}^n$ we have used are face-preserving, so
       the isotopy from $\tau$ to $h_{\tau}$ is also face-preserving.
     \end{proof}
   \nn

    \begin{lem}[Alexander]
       If $h$ is a homeomorphism of an $n$-dimensional disk $D^n$ onto
       itself that fixes all points of $\partial D^n$, then there is an
       isotopy $H_t$ ($0\leq t \leq 1$) of $D^n$ onto itself such
       that $H_0=h$, $H_1=id_{D^n}$ and each $H_t$ fixes all points in $\partial D^n$.
    \end{lem}
    \nn

   \begin{defi}[Regular Facets-Pairing Structure on a Cube]
   \label{Def:Regular-FP-Struc}
     A facets-pairing structure $\mathcal{F}=\{ \omega, \tau_j \}_{j\in [\pm n]}$
     on $\mathcal{C}^n$ is called \textit{regular} if\n
    \begin{enumerate}
      \item  each $\tau_j: \mathbf{F}(j) \rightarrow \mathbf{F}(\omega(j))$ is an
        isometry with respect to the Euclidean metric on $\mathbf{F}(j)$ and
        $\mathbf{F}(\omega(j))$;\n
      \item $\omega: [\pm n] \rightarrow [\pm n]$ is a signed permutation with
        $\omega\circ \omega = id_{[\pm n]}$.\n
    \end{enumerate}
    The second condition in the above definition simply means that if $\mathbf{F}(j)$ is paired with
    $\mathbf{F}(\omega(j))$, then $\mathbf{F}(-j)$ is paired with $\mathbf{F}(-\omega(j))$.
   \end{defi}
      \nn

   In the rest of this section,
   we always assume that $\mathcal{F}=\{ \omega, \tau_j \}_{j\in [\pm n]}$
    is a regular facets-pairing structure on $\mathcal{C}^n$. By our
    assumption, each $\tau_j$ determines a map
    $\sigma_j :  [\pm n] \backslash \{ \pm j \} \rightarrow
      [\pm n] \backslash \{ \pm \omega(j) \} $ where
      \begin{equation} \label{Equ:sigma_j}
       \tau_j(\mathbf{F}(j,k)) = \mathbf{F}(\omega(j), \sigma_j(k)), \
    \forall\, k\in [\pm n] \backslash \{ \pm j\}
       \end{equation}
       Obviously, $\sigma_j$ is a bijection and $\sigma_j(-k) = -\sigma_j(k)$ for all $k$.
    Moreover, each $\sigma_j$ canonically determines a permutation
    $\widetilde{\sigma}_j : [\pm n] \rightarrow [\pm n]$ by:
      \begin{equation} \label{Equ:Sigma-Tilde}
            \widetilde{\sigma}_j (k)  := \left\{
                                       \begin{array}{ll}
                                         \sigma_j(k), & \hbox{ $ k \neq \pm j$;} \\
                                              \omega(k) , & \hbox{$\ k= \pm j$.}
                                            \end{array}
                                          \right.
       \end{equation}
   Since $\omega$ is a signed permutation on $[\pm n]$, so
   is $\widetilde{\sigma}_j $. And by the definition,
     \begin{equation} \label{Equ:omega-sigma}
        \widetilde{\sigma}_j(\pm j) = \pm \omega(j),\ \,
        \forall\, j \in [\pm n].
     \end{equation}
   By the fact above, $\widetilde{\sigma}_j $
     determines a unique symmetry of the cube $\mathcal{C}^n$,
    denoted by $\widetilde{\tau}_j : \mathcal{C}^n \rightarrow \mathcal{C}^n$ where
    $\widetilde{\tau}_j ( \mathbf{F}(k) ) =\mathbf{F}(\widetilde{\sigma}_j(k))$ for any
     $k\in [\pm n]$.
    Then~\eqref{Equ:sigma_j} becomes:
      $$\widetilde{\tau}_j(\mathbf{F}(j,k))
      =\mathbf{F}(\widetilde{\sigma}_j(j),\widetilde{\sigma}_j(k))).$$
      Obviously, $\tau_j =\widetilde{\tau}_j|_{\mathbf{F}(j)}$.
      So $\tau_j$ and $\widetilde{\tau}_j$ determine each other.
      By an abuse of terms, we also call
       $\{ \widetilde{\tau}_j : \mathcal{C}^n \rightarrow \mathcal{C}^n \}_{j\in [\pm n]}$
       the \textit{structure maps} of $\mathcal{F}$.
    \vskip .2cm

    \begin{defi} \label{Def:FPS-Equiv}
     Two facets-pairing structures
      $\{ \omega, \tau_j \}_{j\in [\pm n]}$ and
        $\{ \omega' , \tau'_j \}_{j\in [\pm n]}$
        on $\mathcal{C}^n$ are called
    \textit{equivalent} if there exists an symmetry
    $h: \mathcal{C}^n \rightarrow \mathcal{C}^n$ such that
    $\tau_j  = h^{-1} \circ \tau'_{j'} \circ h$ where
    $\mathbf{F}(j')= h(\mathbf{F}(j))$ for each $j\in [\pm n]$.\nn

     Obviously, if two facets-pairing structure $\mathcal{F}$ and $\mathcal{F}'$
      on $\mathcal{C}^n$ are equivalent, one of $\mathcal{F}$ and $\mathcal{F}'$
       is perfect (or strong) will imply the other is also perfect (or
       strong).\nn
    \end{defi}

   \noindent \textbf{Problem(A):} how to classify all the regular facets-pairing
    structures on the $n$-dimensional cube $\mathcal{C}^n$ up to equivalence? \vskip .2cm

    By the above discussion, the information of
    a regular facets-pairing structure $\mathcal{F}$ on $\mathcal{C}^n$
    is completely encoded in a set of maps $\{ \omega, \sigma_j \}_{j\in [\pm n]}$ or
    $\{ \omega, \widetilde{\sigma}_j \}_{j\in [\pm n]}$. So to investigate this problem,
    we first interpret the condition (I) and (II)
    in the definition of facets-pairing structure into some
    conditions on $\{ \omega,  \sigma_j \}_{j\in [\pm n]}$ or
    $\{ \omega, \widetilde{\sigma}_j \}_{j\in [\pm n]}$.
    In fact,  the condition (I) in Definition~\ref{Def:Facets-Pairing-Struc}
    is equivalent to:
  \begin{equation} \label{Equ:Cube-Commute-0}
    \tau_{\omega(j)} \circ \tau_j = id_{\mathbf{F}(j)}\ \text{for}\ \forall\, j \in [\pm n].
   \end{equation}
  The condition(II) in Definition~\ref{Def:Facets-Pairing-Struc} is equivalent to:
   \begin{equation} \label{Equ:Cube-Commute-1}
    \tau_{\sigma_j(k)}\left( \tau_j( p )\right) =
             \tau_{\sigma_k(j)} \left( \tau_k(p) \right),
     \ \forall\, p \in \mathbf{F}(j,k) \ \text{and}\ |j|\neq |k|.
   \end{equation}
    This implies that $\tau_{\sigma_j(k)}\left( \tau_j( \mathbf{F}(j,k) )\right) =
             \tau_{\sigma_k(j)} \left( \tau_k(\mathbf{F}(j,k)) \right)$.
             So we have:
     \begin{equation*}
        \tau_{\sigma_j(k)}\left(  \mathbf{F}(\omega(j),\sigma_j(k)) \right) =
             \tau_{\sigma_k(j)} \left( \mathbf{F}(\sigma_k(j),\omega(k)) \right)
      \end{equation*}
        \begin{equation*}
   \Longrightarrow\  \mathbf{F}(\sigma_{\sigma_j(k)} (\omega(j)), \omega(\sigma_j(k)))
              = \mathbf{F}(\omega(\sigma_k(j)), \sigma_{\sigma_k(j)}(\omega(k)))
      \end{equation*}
Notice that $\sigma_j(k)\neq \sigma_k(j)$. Otherwise,
   $\sigma_{\sigma_k(j)} (\omega(j)) = \sigma_{\sigma_j(k)} (\omega(j)) =
      \sigma_{\sigma_k(j)}(\omega(k)) $ then $\omega(j) = \omega(k)$, hence
      $j=k$ which contradicts $ |j|\neq |k|$. So we must have
      \begin{equation} \label{Equ:Cube-Commute-2}
         \sigma_{\sigma_j(k)} (\omega(j)) = \omega(\sigma_k(j));\quad
       \sigma_{\sigma_k(j)}(\omega(k)) =\omega(\sigma_j(k))\ \ \text{for}\ \forall\, |j|\neq |k|.
       \end{equation}
  \nn

     In addition,  since $\tau_{\sigma_j(k)} \circ \tau_j$ and
           $ \tau_{\sigma_k(j)} \circ \tau_k$ are both
isometries, the Equation~\eqref{Equ:Cube-Commute-1} is equivalent to
 saying that $\tau_{\sigma_j(k)} \circ \tau_j$ and
           $ \tau_{\sigma_k(j)} \circ \tau_k$ induce the same
           map from the face lattice of $\mathbf{F}(j, k)$ to the face lattice of the image face, i.e.
          for any face $ \mathbf{F}(j,k,l)\subset \mathbf{F}(j, k)$,
          we should have:
        \begin{equation}
          \tau_{\sigma_j(k)} \circ \tau_j (\mathbf{F}(j,k, l))  =
            \tau_{\sigma_k(j)} \circ \tau_k (\mathbf{F}(j,k, l))
          \end{equation}
      \begin{equation*}
       \Longrightarrow\   \tau_{\sigma_j(k)} (\mathbf{F}(\omega(j),\sigma_j(k), \sigma_j(l)))  =
            \tau_{\sigma_k(j)} (\mathbf{F}(\sigma_k(j),\omega(k), \sigma_k(l)))
          \end{equation*}
\begin{equation*}
     \mathbf{F}(\sigma_{\sigma_j(k)} (\omega(j)), \omega(\sigma_j(k)), \sigma_{\sigma_j(k)}(\sigma_j(l)))
              = \mathbf{F}( \omega(\sigma_k(j)), \sigma_{\sigma_k(j)}(\omega(k)), \sigma_{\sigma_k(j)}(\sigma_k(l)))
      \end{equation*}
       \begin{equation} \label{Equ:Cube-Commute-3}
          \overset{\eqref{Equ:Cube-Commute-2}}{\Longrightarrow}\ \
             \sigma_{\sigma_j(k)}(\sigma_j(l)) =
             \sigma_{\sigma_k(j)}(\sigma_k(l)),\ \forall\, |l| \neq  |k| \neq |j|.
       \end{equation}
     \nn

    \noindent So the condition(II) on $\mathcal{F}$ is equivalent to the
   condition~\eqref{Equ:Cube-Commute-2} and~\eqref{Equ:Cube-Commute-3}
   on $\{ \omega, \sigma_j \}_{j \in [\pm n]}$.
  Moreover, we can interpret the~\eqref{Equ:Cube-Commute-0},~\eqref{Equ:Cube-Commute-2} and~\eqref{Equ:Cube-Commute-3}
  into conditions on $\{ \widetilde{\sigma}_j \}_{j \in [\pm n]}$ as:
  \begin{equation} \label{Equ:Cube-Commute-6}
  \widetilde{\sigma}_{\omega(j)} \circ
  \widetilde{\sigma}_{j} = id_{[\pm n]},\ \ \forall\, j \in [\pm n].
   \end{equation}
      \begin{equation} \label{Equ:Cube-Commute-4}
               \widetilde{\sigma}_{\widetilde{\sigma}_j(k)} (\omega(j))=
   \omega(\widetilde{\sigma}_k(j)),\ \forall\, j,k\in [\pm n]
       \end{equation}
     \begin{equation} \label{Equ:Cube-Commute-5}
             \widetilde{\sigma}_{\widetilde{\sigma}_j(k)}(\widetilde{\sigma}_j(l)) =
             \widetilde{\sigma}_{\widetilde{\sigma}_k(j)}(\widetilde{\sigma}_k(l)),\
  \forall\, j,k,l \in [\pm n]
       \end{equation}
  \n

   Notice that if we set $l=j$ in~\eqref{Equ:Cube-Commute-5}, we get~\eqref{Equ:Cube-Commute-4}
  from~\eqref{Equ:Cube-Commute-6}.
  So~\eqref{Equ:Cube-Commute-4} is essentially contained in
  ~\eqref{Equ:Cube-Commute-6} and~\eqref{Equ:Cube-Commute-5}.

  \nn

   By the above analysis, we can build a one-to-one correspondence between
     the set of all regular facets-pairing
 structures on $\mathcal{C}^n$ and the set of all tuples of
   elements $(\omega; T_1, T_{-1},\cdots, T_n, T_{-n})$ in $\mathfrak{S}^{\pm}_n$ which
   satisfy:\n
    \begin{enumerate}
       \item[(a)] $\omega\circ \omega = id_{[\pm n]}$ for $\forall\, j \in [\pm n]$.\n
      \item[(b)]  $T_j(j)= \omega(j)$ and $T_{\omega(j)}\circ T_j = id_{[\pm n]}$,\
        $\forall\, j\in [\pm n]$, \vskip .1cm
      \item[(c)]   $T_{T_j(k)}\circ T_j = T_{T_k(j)}\circ T_k $,
      \  $\forall\, j,k \in [\pm n]$.
    \end{enumerate}

    The correspondence is just
    mapping a regular facets-pairing structure $\mathcal{F}$ on $\mathcal{C}^n$
    to the tuple of elements
    $(\omega; \widetilde{\sigma}_1, \widetilde{\sigma}_{-1},\cdots,
    \widetilde{\sigma}_n, \widetilde{\sigma}_{-n})$ in $\mathfrak{S}^{\pm}_n$
    defined in~\eqref{Equ:Sigma-Tilde}.\nn

      Two tuples $(\omega ; T_1, T_{-1},\cdots, T_n, T_{-n})$ and
    $(\omega' ; T'_1, T'_{-1},\cdots, T'_n, T'_{-n})$ of elements in $\mathfrak{S}^{\pm}_n$
        are called \textit{shuffled-conjugate} if there exists some
  $S\in \mathfrak{S}^{\pm}_n$ such that:
           $$ \omega= S^{-1} \omega' S; \ \  T_j= S^{-1} T'_{S(j)} S, \ \ \forall\, j\in [\pm n]. $$

  Then by the above definitions, it is easy to see the following.\nn

  \begin{thm} \label{thm:Main-2-Classify}
     Two regular
      facets-pairing structures $\mathcal{F}$ and $\mathcal{F}'$ on $\mathcal{C}^n$
      are equivalent
   if and only if their corresponding tuples of elements
  $(\omega; \widetilde{\sigma}_1, \widetilde{\sigma}_{-1},\cdots, \widetilde{\sigma}_n, \widetilde{\sigma}_{-n})$
      and
 $(\omega';\widetilde{\sigma}'_1, \widetilde{\sigma}'_{-1},\cdots, \widetilde{\sigma}'_n, \widetilde{\sigma}'_{-n})$
    in $\mathfrak{S}^{\pm}_n$ are shuffled-conjugate.
    \end{thm}
   \n

  So we can formally answer Problem(A) by saying that: \nn

   classifying all the regular
  facets-pairing structures on $\mathcal{C}^n$ up to equivalence is
   the same
    as classifying all tuples of elements
     $(\omega; T_1, T_{-1},\cdots, T_n, T_{-n})$ in
    $\mathfrak{S}^{\pm}_n$ which satisfy (a)---(c)
    up to shuffled-conjugacy. \\

  \section{ Strong vs. Perfect} \label{Section4}

  In this section, we will study the relationship between the strongness and
   the perfectness for regular facets-pairing structures on $\mathcal{C}^n$.
    Suppose $\mathcal{F}=\{ \omega, \tau_j \}_{j\in [\pm n]}$ is a regular facets-pairing structure on
  $\mathcal{C}^n$. For each $j\in [\pm n]$, let
   $\widetilde{\sigma}_j \in \mathfrak{S}^{\pm}_n$ be the
    signed permutation defined by~\eqref{Equ:Sigma-Tilde} and let
    $\widetilde{\tau}_j : \mathcal{C}^n \rightarrow \mathcal{C}^n $ be the
    corresponding symmetry of $\mathcal{C}^n$. \nn

    Since the strongness and
  perfectness are descriptive notions on the face families in a facets-pairing structure, so let us
     examine the structure of the face families in $\mathcal{F}$ first.
  To avoid ambiguity, for any face $f$ of $\mathcal{C}^n$, we call each face belonging to the face family
   $\widehat{f}$ a \textit{component} of $\widehat{f}$. \nn

     Suppose $f=\mathbf{F}(j_1,\cdots, j_s)$. In the following, we will also
      $\widehat{\mathbf{F}}(j_1,\cdots ,j_s)$ to denote the face family $\widehat{f}$.
       For any $1\leq i_1 \leq s$, we can get a component of $\widehat{f}$ by:
      \begin{align*}
      \tau_{\mathbf{F}(j_{i_1})} (f) = \widetilde{\tau}_{j_{i_1}} (f) &=
      \mathbf{F}(\widetilde{\sigma}_{j_{i_1}}(j_1), \cdots, \widetilde{\sigma}_{j_{i_1}}(j_s)) \\
       &= \mathbf{F}(\widetilde{\sigma}_{j_{i_1}}(j_1), \cdots, \widetilde{\sigma}_{j_{i_1}}(j_{i_1-1}),
       \omega(j_{i_1}), \widetilde{\sigma}_{j_{i_1}}(j_{i_1+1}), \cdots, \widetilde{\sigma}_{j_{i_1}}(j_s)).
       \end{align*}
   Then if we choose an arbitrary element in
  $\{ \widetilde{\sigma}_{j_{i_1}}(j_1), \cdots, \widetilde{\sigma}_{j_{i_1}}(j_{s})\}$,
   say $\widetilde{\sigma}_{j_{i_1}}(j_{i_2})$, we can get another component
   $\tau_{\mathbf{F}(\widetilde{\sigma}_{j_{i_1}}(j_{i_2}))}
    \circ \tau_{\mathbf{F}(j_{i_1})} (f)
    =\widetilde{\tau}_{\widetilde{\sigma}_{j_{i_1}}(j_{i_2})} ( \widetilde{\tau}_{j_{i_1}} (f) )$
   in $\widehat{f}$. \nn

    More generally, let
   $i_1,\cdots, i_m \in \{1,\cdots, s\}$ and we define
     $$ k_1 = j_{i_1}\in \{ j_1,\cdots, j_s \},\  \text{and}\qquad $$
  \begin{equation} \label{Equ:Derived-Seq}
     k_p = \widetilde{\sigma}_{k_{p-1}} \circ \cdots \widetilde{\sigma}_{k_1}(j_{i_p}),\
       2 \leq  p \leq m.
\end{equation}
    Obviously, the sequence $(k_1, \cdots, k_m)$
   is completely determined by the sequence $(i_1, \cdots, i_m)$ and
   the $(j_1,\cdots, j_s)$.
  So we call $(k_1, \cdots, k_m)$ the \textit{derived sequence} of
  $(i_1, \cdots, i_m)$ from $(j_1,\cdots, j_s)$.\nn

  By our definition,
  it is easy to see that for any valid form $\tau_{\mathbf{F}(k_m)}\circ\cdots\circ
   \tau_{\mathbf{F}(k_1)}(f)$, there exists a sequence
   $(i_1,\cdots, i_m)$ with $1\leq i_1,\cdots, i_m \leq s$ so that
    $(k_1,\cdots,k_m)$ is the derived sequence of $(i_1,\cdots, i_m)$ from
  $(j_1,\cdots ,j_s)$. We call
  $ \tau_{\mathbf{F}(k_m)}\circ\cdots\circ \tau_{\mathbf{F}(k_1)}(f)$ the
    \textit{face generated by $(i_1, \cdots, i_m)$ from the normal
    form $\mathbf{F}(j_1,\cdots, j_s)$ of $f$}. For convenience, we introduce the following
    notation.
   \begin{align} \label{Equ:Notation-2}
    \qquad \mathbf{F}_{k_1,\cdots,k_m}(j_1,\cdots, j_s)
      & := \widetilde{\tau}_{k_m}\circ\cdots\circ\widetilde{\tau}_{k_1}
      (\mathbf{F}(j_1,\cdots, j_s))  \notag \\
      &\: =
    \mathbf{F}(\widetilde{\sigma}_{k_m}\circ\cdots\circ\widetilde{\sigma}_{k_1}(j_1),\cdots,
      \widetilde{\sigma}_{k_m}\circ\cdots\circ\widetilde{\sigma}_{k_1}(j_s)).
   \end{align}
  Especially when $m=0$, we call $(i_1, \cdots, i_m)$ the empty sequence,
  denoted by $(\varnothing)$, and the derived sequence of $(\varnothing)$ from
  $(j_1,\cdots, j_s)$ is still $(\varnothing)$.
  Then according to Definition~\ref{Defi:Face-Family},
  $ \mathbf{F}_{\varnothing}(j_1,\cdots, j_s) = \mathbf{F}(j_1,\cdots, j_s)$.

  \n

  \begin{rem}
  Generally speaking, if $\mathbf{F}(j'_1,\cdots ,j'_s) = \mathbf{F}(j_1,\cdots ,j_s)$
  but $(j'_1,\cdots ,j'_s)$ and $(j_1,\cdots ,j_s)$ are not the same sequence,
   the face generated by
  $(i_1, \cdots, i_m)$
     from $\mathbf{F}(j_1,\cdots, j_s)$ and from $\mathbf{F}(j'_1,\cdots ,j'_s)$ are
     very likely to be different.
     So the order of the $j_1,\cdots ,j_s$ in the notation
     $\mathbf{F}_{k_1,\cdots,k_m}(j_1,\cdots, j_s)$ is essential.
  \end{rem}
    \n

    To emphasize the order of the $j_1,\cdots, j_s$ in a normal form
    $\mathbf{F}(j_1,\cdots ,j_s)$, we introduce a new
    notation as following.
    \nn

  \begin{defi} [Strongly Equal]
     \

    Two normal forms
      $\mathbf{F}(j_1,\cdots ,j_s)$ and $\mathbf{F}(j'_1,\cdots ,j'_s)$ are called
     \textit{strongly equal} if
     $(j_1,\cdots ,j_s) = (j'_1,\cdots ,j'_s)$ as a sequence and we
     write $\mathbf{F}(j_1,\cdots ,j_s) \equiv \mathbf{F}(j'_1,\cdots ,j'_s)$. \nn
  \end{defi}

   In the rest of the paper, we always understand $\mathbf{F}_{k_1,\cdots,k_m}(j_1,\cdots, j_s)$ as
     $$\mathbf{F}_{k_1,\cdots,k_m}(j_1,\cdots, j_s)
     \equiv \mathbf{F}(\widetilde{\sigma}_{k_m}\circ\cdots\circ\widetilde{\sigma}_{k_1}(j_1),\cdots,
      \widetilde{\sigma}_{k_m}\circ\cdots\circ\widetilde{\sigma}_{k_1}(j_s)).$$
     If $(k_1,\cdots,k_m)$ is the derived sequence of $(i_1,\cdots, i_m)$ from
  $(j_1,\cdots ,j_s)$, we call $\mathbf{F}_{k_1,\cdots,k_m}(j_1,\cdots, j_s)$
   the \textit{normal form generated by $(i_1,\cdots, i_m)$ from} $\mathbf{F}(j_1,\cdots ,j_s)$.
   We will see that this convention is very convenient for our
     proofs.
  \nn

  \begin{lem} \label{Lem:First-Simple}
    For a
     sequence $(i_1, \cdots, i_m)$ with $1\leq i_1,\cdots, i_m \leq s$,
     let $(k_1,\cdots, k_m)$ be the derived
     sequence of $(i_1, \cdots, i_m)$ from $(j_1,\cdots, j_s)$. For
      any $1\leq p \leq m$, let
     $\mathbf{F}(j'_1,\cdots, j'_s) \equiv \mathbf{F}_{k_1,\cdots,k_p}(j_1,\cdots, j_s)$.
     Then $(k_{p+1},\cdots,k_m)$ is the derived sequence of $(i_{p+1},\cdots, i_m)$ from
     $(j'_1,\cdots, j'_s)$ and we have
     \[  \mathbf{F}_{k_1,\cdots,k_m}(j_1,\cdots, j_s) \
     \equiv \mathbf{F}_{k_{p+1},\cdots,k_m}(j'_1,\cdots, j'_s) \]
  \end{lem}
  \begin{proof}
     It follows easily from the definition of both sides of the equation.
   \end{proof}

    \nn

  \begin{lem} \label{Lem:Second-Simple}
    For a
     sequence $(i_1, \cdots, i_m)$ with $1\leq i_1,\cdots, i_m \leq s$,
     let $(k_1,\cdots, k_m)$ be the derived
     sequence of $(i_1, \cdots, i_m)$ from $(j_1,\cdots, j_s)$.
     Suppose
     $\mathbf{F}(j'_1,\cdots, j'_s) \equiv \mathbf{F}_{k_1,\cdots,k_m}(j_1,\cdots, j_s)$.
     Then $(k_m,\cdots,k_1)$ is the derived sequence of
     $(i_m,\cdots, i_1)$ from $(j'_1,\cdots, j'_s)$ and
      $ \mathbf{F}(j_1,\cdots, j_s)  \equiv \mathbf{F}_{k_m,\cdots,k_1}(j'_1,\cdots,
     j'_s)$.
  \end{lem}
  \begin{proof}
    Since
     $ \mathbf{F}_{k_1,\cdots,k_m}(j_1,\cdots, j_s) \equiv \mathbf{F}(\widetilde{\sigma}_{k_m}\circ\cdots\circ\widetilde{\sigma}_{k_1}(j_1),\cdots,
      \widetilde{\sigma}_{k_m}\circ\cdots\circ\widetilde{\sigma}_{k_1}(j_s))$,
      so $j'_i
      =\widetilde{\sigma}_{k_m}\circ\cdots\circ\widetilde{\sigma}_{k_1}(j_i)$
      for each $1\leq i \leq s$. Notice that
    $j'_{i_m} = \widetilde{\sigma}_{k_m}(k_m)=\omega(k_m)$.
    So since $\widetilde{\sigma}_{\omega(k_m)}\circ \widetilde{\sigma}_{k_m} = id_{[\pm n]}$,
    the face generated by $(i_m)$ from $\mathbf{F}(j'_1,\cdots, j'_s)$ is
     \begin{align*}
     \widetilde{\tau}_{j'_{i_m}}(\mathbf{F}(j'_1,\cdots, j'_s) ) & \equiv
     \widetilde{\tau}_{\omega(k_m)}(\mathbf{F}(\widetilde{\sigma}_{k_m}\circ\cdots\circ\widetilde{\sigma}_{k_1}(j_1),\cdots,
      \widetilde{\sigma}_{k_m}\circ\cdots\circ\widetilde{\sigma}_{k_1}(j_s))) \\
      &\equiv \mathbf{F}(\widetilde{\sigma}_{k_{m-1}}\circ\cdots\circ\widetilde{\sigma}_{k_1}(j_1),\cdots,
      \widetilde{\sigma}_{k_{m-1}}\circ\cdots\circ\widetilde{\sigma}_{k_1}(j_s)).
     \end{align*}
   Note that the order of the indices in $\mathbf{F}(j'_1,\cdots, j'_s) $ and
    $ \mathbf{F}_{k_1,\cdots,k_m}(j_1,\cdots, j_s)$ are crucial to make the formula work here.\n

    Furthermore, since
   $\widetilde{\sigma}_{k_{m-1}}\circ\cdots\circ\widetilde{\sigma}_{k_1}(j_{i_{m-1}}) = \omega(k_{m-1})$,
   we can repeat the above argument and show that the face generated
   by $(i_m, i_{m-1})$ from $\mathbf{F}(j'_1,\cdots, j'_s)$ is strongly equal to
    $\mathbf{F}(\widetilde{\sigma}_{k_{m-2}}\circ\cdots\circ\widetilde{\sigma}_{k_1}(j_1),\cdots,
      \widetilde{\sigma}_{k_{m-2}}\circ\cdots\circ\widetilde{\sigma}_{k_1}(j_s))$.
  Then after $m$ steps,
  the face generated by $(i_m, \cdots, i_1)$ from $\mathbf{F}(j'_1,\cdots, j'_s)$
   is exactly strongly equal to $\mathbf{F}(j_1,\cdots, j_s)$. By our
   construction, it is clear that $(k_m,\cdots,k_1)$ is the derived sequence of
     $(i_m,\cdots, i_1)$ from $(j'_1,\cdots, j'_s)$.
    \end{proof}
     \nn

  \begin{exam} \label{Exam:Face-Family}
  All the possible faces in the face family
  containing $\mathbf{F}(j,k)$ are: $\mathbf{F}(j,k)$,
  $\mathbf{F}(\omega(j),\widetilde{\sigma}_j(k))$, $\mathbf{F}(\widetilde{\sigma}_k(j),\omega(k))$
  and
 $\mathbf{F}(\omega(\widetilde{\sigma}_k(j)), \omega(\widetilde{\sigma}_j(k)))$ whose
  relations are shown in the following diagram.
  \[  \xymatrix{
  \mathbf{F}(\widetilde{\sigma}_k(j),\omega(k)) \ \ar[r]^{ \widetilde{\tau}_{\widetilde{\sigma}_k(j)}\,\quad \ }
   &  \ \mathbf{F}(\omega(\widetilde{\sigma}_k(j)), \omega(\widetilde{\sigma}_j(k)))   \\
  \mathbf{F}(j,k) \ar[u]_{\widetilde{\tau}_{k}} \ar[r]_{\widetilde{\tau}_j\quad\,  \ } & \,
  \mathbf{F}(\omega(j),\widetilde{\sigma}_j(k))
   \ar[u]^{\widetilde{\tau}_{\widetilde{\sigma}_j(k)}}}
    \]

 Of course, we can use longer sequences to generate faces
  with more complicated forms from $\mathbf{F}(j,k)$.
  But we will not get any new face. For example,
 \[ \widetilde{\tau}_{\omega(\widetilde{\sigma}_k(j))} \left(
  \mathbf{F}(\omega(\widetilde{\sigma}_k(j)), \omega(\widetilde{\sigma}_j(k))) \right)
   \overset{\eqref{Equ:omega-sigma}}{=}
   \mathbf{F}(\widetilde{\sigma}_k(j), \widetilde{\sigma}_{\omega(\widetilde{\sigma}_k(j))}
   (\omega(\widetilde{\sigma}_j(k)))). \]
  Since $ \widetilde{\sigma}_{\omega(\widetilde{\sigma}_k(j))}
   (\omega(\widetilde{\sigma}_j(k))) \overset{\eqref{Equ:Cube-Commute-4}}{=}
    \widetilde{\sigma}_{ \widetilde{\sigma}_{\widetilde{\sigma}_j(k)} (\omega(j))}
   (\omega(\widetilde{\sigma}_j(k)))
   \overset{\eqref{Equ:Cube-Commute-4}}{=}
   \omega((\widetilde{\sigma}_{\omega(j)} (\widetilde{\sigma}_j(k)) ))
    \overset{\eqref{Equ:Cube-Commute-6}}{=}  \omega(k)
  $, so $\mathbf{F}(\widetilde{\sigma}_k(j), \widetilde{\sigma}_{\omega(\widetilde{\sigma}_k(j))}
   (\omega(\widetilde{\sigma}_j(k)))) =
   \mathbf{F}(\widetilde{\sigma}_k(j),\omega(k))$ is not a new face.\nn

  In addition, if we assume
  $\mathbf{F}(\omega(j),\widetilde{\sigma}_j(k)) = \mathbf{F}(\widetilde{\sigma}_k(j),\omega(k))$,
  then we must have $\widetilde{\sigma}_j(k)= \omega(k)$ and $\widetilde{\sigma}_k(j) = \omega(j)$.
  Then $\mathbf{F}(\omega(\widetilde{\sigma}_k(j)), \omega(\widetilde{\sigma}_j(k)))
   = \mathbf{F}(j,k)$,
  which implies that the face family $ \widehat{\mathbf{F}}(j,k)$ has only two
  components. \nn

  In general, we will see in Theorem~\ref{thm:Component-1}
  that the number of components in any face
  family in a strong regular facets-pairing structure must be a
  power of $2$.
 \end{exam}

  \n

  \begin{lem} \label{Lem:Order-Indep}
   Suppose $(k_1,\cdots,k_m)$ and $(k'_1,\cdots, k'_m)$ are the derived
  sequences of $(i_1,\cdots, i_m)$ and $(i'_1,\cdots, i'_m)$ from
  $(j_1,\cdots ,j_s)$ respectively. If $(i'_1, \cdots, i'_m)$ is just a
  permutation of the entries in $(i_1, \cdots, i_m)$, then we have\n
   \begin{enumerate}
     \item[(a)] $ \mathbf{F}_{k_1,\cdots,k_m}(j_1,\cdots, j_s)
            \equiv \mathbf{F}_{k'_1,\cdots,k'_m}(j_1,\cdots, j_s)$.\n
     \item[(b)] $\widetilde{\sigma}_{k_m} \circ\cdots\circ\widetilde{\sigma}_{k_1} =
        \widetilde{\sigma}_{k'_m} \circ\cdots\circ
        \widetilde{\sigma}_{k'_1}$ as a permutation on $[\pm n]$.
   \end{enumerate}
   \end{lem}
\begin{proof}
   First, we assume the sequence $(i'_1, \cdots, i'_m)$ is just a transposition of two
   neighboring entries in $(i_1, \cdots, i_m)$. By
   Lemma~\ref{Lem:First-Simple}, it suffice to prove the case when
   $(i'_1, i'_2, i'_3, \cdots, i'_m) = (i_2, i_1, i_3, \cdots, i_m)$. By
   definition, we have
   \begin{align*}
     \mathbf{F}_{k_1, k_2}(j_1,\cdots, j_s)
    & \equiv \left( \widetilde{\sigma}_{\widetilde{\sigma}_{j_1}(j_2)}(\widetilde{\sigma}_{j_1}(j_1)),
    \widetilde{\sigma}_{\widetilde{\sigma}_{j_1}(j_2)}(\widetilde{\sigma}_{j_1}(j_2)),
     \widetilde{\sigma}_{\widetilde{\sigma}_{j_1}(j_2)}(\widetilde{\sigma}_{j_1}(j_3)),
     \cdots,
     \widetilde{\sigma}_{\widetilde{\sigma}_{j_1}(j_2)}(\widetilde{\sigma}_{j_1}(j_s))
     \right) \\
   \mathbf{F}_{k'_1, k'_2}(j_1,\cdots, j_s)
   &= \left(  \widetilde{\sigma}_{\widetilde{\sigma}_{j_2}(j_1)}(\widetilde{\sigma}_{j_2}(j_1)),
      \widetilde{\sigma}_{\widetilde{\sigma}_{j_2}(j_1)}(\widetilde{\sigma}_{j_2}(j_2)),
     \widetilde{\sigma}_{\widetilde{\sigma}_{j_2}(j_1)}(\widetilde{\sigma}_{j_2}(j_3)),
     \cdots,
     \widetilde{\sigma}_{\widetilde{\sigma}_{j_2}(j_1)}(\widetilde{\sigma}_{j_2}(j_s))  \right)
     \end{align*}

   By~\eqref{Equ:Cube-Commute-5},
      $\widetilde{\sigma}_{\widetilde{\sigma}_{j_1}(j_2)} \circ \widetilde{\sigma}_{j_1}
       (l) =  \widetilde{\sigma}_{\widetilde{\sigma}_{j_2}(j_1)}\circ
      \widetilde{\sigma}_{j_2} (l)$ for any $l\in [\pm n]$, so
      $$
      \mathbf{F}_{k_1, k_2}(j_1,\cdots, j_s) \equiv \mathbf{F}_{k'_1, k'_2}(j_1,\cdots,
      j_s).$$
      Then Lemma~\ref{Lem:First-Simple} implies that  $\mathbf{F}_{k_1,\cdots,k_m}(j_1,\cdots, j_s)$
      and $\mathbf{F}_{k'_1, \cdots, k'_m}(j_1,\cdots, j_s)$
      are generated by the same sequence $(i_3,\cdots, i_m)$ from the same normal form.
       So we get (a) and (b) for this case. \n

      For the general cases, since any permutation of the entries in
      $(i_1, \cdots, i_m)$ is a composite of transpositions of
      neighboring entries, so the lemma follows from the above special case.
        \end{proof}
   \n

 \begin{lem} \label{Lem:Component-0}
   Suppose $1\leq i_1,\cdots, i_m \leq s$  and $(k_1,\cdots,k_m)$ is the derived
  sequence of $(i_1,\cdots, i_m)$ from $(j_1,\cdots ,j_s)$.
  Then there exists a subsequence $(i_{l_1},\cdots, i_{l_r})$
  of $(i_1,\cdots, i_m)$ with $i_{l_1},\cdots, i_{l_r}$  pairwise distinct so that
  \[ \mathbf{F}_{k_1,\cdots, k_m}(j_1, \cdots, j_s) \equiv \mathbf{F}_{k'_1,\cdots, k'_r}(j_1, \cdots, j_s) \]
 where $(k'_1,\cdots, k'_r)$ is the derived
  sequence of $(i_{l_1},\cdots, i_{l_r})$ from $(j_1,\cdots ,j_s)$.
 \end{lem}
 \begin{proof}
 If $i_1,\cdots, i_m$ are already pairwise distinct, there is nothing to
 prove. If there exist
  $i_l = i_{l'}$ for some $1\leq l < l' \leq m$,  we
   permutate the entries of $(i_1,\cdots, i_m)$ to get a new sequence below:
   \begin{equation} \label{Equ:Seq-1}
     (i_l,i_{l'}, i_1, \cdots, \widehat{i}_{l}, \cdots,
     \widehat{i}_{l'}, \cdots, i_m)
   \end{equation}
   Let $(d_1,\cdots, d_m)$ be the derived sequence of
   the sequence~\eqref{Equ:Seq-1} from $(j_1, \cdots, j_s)$.
   By Lemma~\ref{Lem:Order-Indep}, we have
    $\mathbf{F}_{k_1,\cdots, k_m}(j_1, \cdots, j_s)\equiv
   \mathbf{F}_{d_1,\cdots, d_m}(j_1, \cdots, j_s)$.\n

   Notice that $d_1=j_{i_l}$ and $d_2=\omega(j_{i_l})$ since $i_l = i_{l'}$.
   So by~\eqref{Equ:Cube-Commute-6}, we get
    $$\mathbf{F}_{d_1,d_2}(j_1, \cdots, j_s)\equiv
   \mathbf{F}(j_1, \cdots, j_s).$$
   By Lemma~\ref{Lem:First-Simple},
   $\mathbf{F}_{k_1,\cdots, k_m}(j_1, \cdots, j_s)\equiv
   \mathbf{F}_{d_1,\cdots, d_m}(j_1, \cdots, j_s)\equiv
   \mathbf{F}_{d_3,\cdots, d_m}(j_1, \cdots, j_s)$,
   where $(d_3,\cdots, d_m)$ is the derived sequence of
   $(i_1, \cdots, \widehat{i}_{l}, \cdots,
     \widehat{i}_{l'}, \cdots, i_m)$ from $(j_1, \cdots, j_s)$.
   \n

   The above argument implies that
  if $i_1,\cdots, i_m$ are not pairwise distinct, we can
  delete those entries which appear even number of times in it and reduce
  $(i_1,\cdots, i_m)$ to a shorter sequence
  which still generates the same normal form from
  $\mathbf{F}(j_1, \cdots, j_s)$ as $(i_1,\cdots, i_m)$ does. So the lemma is proved.
\end{proof}
   \n

  \begin{defi}[Irreducible Sequence]
   If all the entries in a sequence $(i_1,\cdots, i_m)$ are pairwise
  distinct, we call $(i_1,\cdots, i_m)$ an \textit{irreducible sequence}.
  Otherwise, we call it \textit{reducible}.\n

   If a sequence $(i_1,\cdots, i_m)$ is
  reducible, we can delete those entries which appear even number of times in it and reduce
  $(i_1,\cdots, i_m)$ to an irreducible sequence.
  \end{defi}

  \n

  By Lemma~\ref{Lem:Order-Indep}, any permutation of entries in an irreducible
  sequence $(i_1,\cdots, i_r)$ will not change the normal form generated from
  $\mathbf{F}(j_1, \cdots, j_s)$. So
   any subset $\Sigma=\{ i_1,\cdots, i_r \} \subset \{ 1, \cdots, s \}$
   determines a unique normal form of a unique component in
   the face family $\widehat{\mathbf{F}}(j_1, \cdots, j_s)$,
    denoted by $\mathbf{F}(j_1, \cdots, j_s)^{\Sigma}$.
   We call $\mathbf{F}(j_1, \cdots, j_s)^{\Sigma}$ the \textit{normal form
    generated by $\Sigma$ from} $\mathbf{F}(j_1, \cdots, j_s)$.
   Note that
     the order of $j_1, \cdots, j_s$ in the notation
     $\mathbf{F}(j_1, \cdots, j_s)^{\Sigma}$ is still essential.\nn

   Then by Lemma~\ref{Lem:Component-0},
   any component in $\widehat{\mathbf{F}}(j_1, \cdots, j_s)$ can be generated by
   some subset of $\{ 1, \cdots, s \}$ (including the empty set) from $\mathbf{F}(j_1, \cdots, j_s)$.
  Since there are a total of $2^s$ subsets of $\{ 1, \cdots, s \}$,
   we have the following corollary. \nn

    \begin{cor} \label{Cor:Number-Components}
     For any regular facets-pairing structure $\mathcal{F}$ on
     $\mathcal{C}^n$, there are at most $2^s$ components in any
     $(n-s)$-dimensional face family of $\mathcal{F}$.
   So the number of components in a face family
  of $\mathcal{F}$ reaches the maximum exactly when the face family is
  perfect. \nn
  \end{cor}

  \begin{lem} \label{Lem:Second-Simple-2}
    For any $\Sigma_1,\Sigma_2 \subset \{1,\cdots, s\}$, let
   $\mathbf{F} (j'_1, \cdots, j'_s)\equiv\mathbf{F}(j_1, \cdots,j_s)^{\Sigma_1}$.
   Then $\mathbf{F}(j'_1, \cdots, j'_s)^{\Sigma_2} \equiv
   \mathbf{F} (j_1, \cdots, j_s)^{\Sigma_1\ominus \Sigma_2}$
   where $\Sigma_1\ominus \Sigma_2 =
   (\Sigma_1\backslash \Sigma_2) \cup (\Sigma_2\backslash \Sigma_1)$ is the symmetric
   difference of $\Sigma_1$ and $\Sigma_2$.
  \end{lem}
  \begin{proof}
   Without loss of generality, we can assume $\Sigma_1=\{ i_1,\cdots, i_p, \cdots, i_r \}$
   and $\Sigma_2=\{ i_p, \cdots, i_r, \cdots, i_m \}$ where $1\leq p \leq r \leq m \leq s$. So
    $\Sigma_1\cap \Sigma_2 = \{ i_p,\cdots, i_r\}$. By the preceding
    lemma,
    $\mathbf{F}(j'_1, \cdots, j'_s)^{\Sigma_2}$ is strongly equal to
    the normal form generated by a sequence
   $(i_1,\cdots, i_p,\cdots, i_{r-1}, i_r, i_r, i_{r-1},\cdots, i_p, i_{r+1} \cdots, i_m)$
  from $\mathbf{F}(j_1,\cdots, j_s)$. Then using the same argument as in
  Lemma~\ref{Lem:Second-Simple}, we can show that the
   middle part $(i_p,\cdots, i_{r-1}, i_r, i_r, i_{r-1},\cdots, i_p)$
   in the above sequence can be reduced to $(\varnothing)$. So
  $\mathbf{F}(j'_1, \cdots, j'_s)^{\Sigma_2}$ is strongly equal to
  the normal form generated by a shorter sequence
   $(i_1,\cdots, i_{p-1}, i_{r+1} \cdots, i_m)$
  from $\mathbf{F}(j_1,\cdots, j_s)$. Notice that
  $ \{ i_1,\cdots, i_{p-1}, i_{r+1} \cdots, i_m \}$ is exactly $\Sigma_1\ominus
  \Sigma_2$, so we are done.
  \end{proof}
   \nn

  In Lemma~\ref{Lem:Second-Simple-2},
   if we denote $\mathbf{F}(j'_1, \cdots, j'_s)^{\Sigma_2}$ by
    $\left(\mathbf{F}(j_1, \cdots, j_s)^{\Sigma_1} \right)^{\Sigma_2}$,  we can rewrite
   Lemma~\ref{Lem:Second-Simple-2} as: for any $\Sigma_1,\Sigma_2 \subset \{1,\cdots, s\}$
   \begin{equation} \label{Equ:Second-Simple-2}
    \left(\mathbf{F}(j_1, \cdots, j_s)^{\Sigma_1} \right)^{\Sigma_2}
   \equiv \mathbf{F} (j_1, \cdots, j_s)^{\Sigma_1\ominus \Sigma_2} .
   \end{equation}
   Similarly, we can rewrite
    Lemma~\ref{Lem:Second-Simple} as: for any $\Sigma \subset \{1,\cdots, s\}$,
      \begin{equation} \label{Equ:Second-Simple}
       \left(\mathbf{F}(j_1, \cdots,j_s)^{\Sigma} \right)^{\Sigma} \equiv \mathbf{F}(j_1, \cdots,j_s).
       \end{equation}
    \nn

   By the above discussion,
    if there exist two different subsets of $\{ 1, \cdots, s \}$
     which generate the same component in $\widehat{\mathbf{F}}(j_1, \cdots, j_s)$,
      the number of components
  in $\widehat{\mathbf{F}}(j_1, \cdots, j_s)$ will be strictly less than $2^s$,
  i.e. the face family $\widehat{\mathbf{F}}(j_1, \cdots, j_s)$ is not perfect.
  So $\widehat{\mathbf{F}}(j_1, \cdots, j_s)$ is perfect if and only
  if different subsets of $\{ 1, \cdots, s \}$ always generate
   different components in $\widehat{\mathbf{F}}(j_1, \cdots, j_s)$.
  \nn

  Next, let us see what is the meaning of ``strongness'' for
  a regular facets-pairing structure $\mathcal{F}$ on $\mathcal{C}^n$.
  For any face $f=\mathbf{F}(j_1,\cdots ,j_s)$ of $\mathcal{C}^n$, let
   $f'= \tau_{\mathbf{F}(j_i)} (f)$, $1\leq i \leq s$. By our notations,
   $f'=\widetilde{\tau}_{j_i} (\mathbf{F}(j_1,\cdots ,j_s)) =
      \mathbf{F}(\widetilde{\sigma}_{j_i}(j_1), \cdots, \widetilde{\sigma}_{j_i}(j_s))$.
   So the map $\Psi^{f}_{\mathbf{F}(j_i)} : \Xi(f) \rightarrow \Xi(f')$
     is given by
     $\widetilde{\sigma}_{j_i}: \{ j_1,\cdots, j_s \} \rightarrow
     \{ \widetilde{\sigma}_{j_i}(j_1), \cdots, \widetilde{\sigma}_{j_i}(j_s) \}$.
     Similarly, the map
     $(\Psi^{f}_{\mathbf{F}(j_i)})^{\perp} : \Xi^{\perp}(f) \rightarrow \Xi^{\perp}(f')$
      is given by:
      $$\widetilde{\sigma}_{j_i}: [\pm n] \backslash \{ \pm j_1,\cdots, \pm j_s \} \rightarrow
       [\pm n] \backslash \{ \pm \widetilde{\sigma}_{j_i}(j_1),
        \cdots, \pm \widetilde{\sigma}_{j_i}(j_s) \}.$$
     Note that $(\Psi^{f}_{\mathbf{F}(j_i)})^{\perp}$ describes the map
       between the face lattices of $f$ and $f'$ induced by $\tau_{\mathbf{F}(j_i)}$.
       Since we assume each $\tau_{j}=\tau_{\mathbf{F}(j)}$ is an isometry with respect
       to the induced Euclidean metric on the facets of $\mathcal{C}^n$,
       so the map $\tau_{\mathbf{F}(j_i)}|_f : f \rightarrow f'$ is
       completely determined by
       $(\Psi^{f}_{\mathbf{F}(j_i)})^{\perp}$, hence by
       $\widetilde{\sigma}_{j_i}$.\nn

   In addition, by our discussion above,
   for any valid form $\tau_{\mathbf{F}(k_m)}\circ\cdots\circ
   \tau_{\mathbf{F}(k_1)}(f)$, there exists a sequence
   $(i_1,\cdots, i_m)$ with $1\leq i_1,\cdots, i_m \leq s$ so that
    $(k_1,\cdots,k_m)$ is the derived sequence of $(i_1,\cdots, i_m)$ from
  $(j_1,\cdots ,j_s)$. By our notation,
   $$\tau_{\mathbf{F}(k_m)}\circ\cdots\circ
   \tau_{\mathbf{F}(k_1)}(f) = \mathbf{F}_{k_1,\cdots, k_m}(j_1,\cdots,
   j_s).$$
   Suppose $(k'_1,\cdots, k'_r)$ is the derived sequence of another
   sequence $(i'_1,\cdots, i'_r)$ from $(j_1,\cdots, j_s)$.
   Then we claim that: $\mathcal{F}$ is strong if and only
   if whenever
   $\mathbf{F}_{k_1,\cdots, k_m}(j_1,\cdots, j_s)
   = \mathbf{F}_{k'_1,\cdots, k'_r}(j_1,\cdots, j_s)$, we must have
   \n
    \begin{enumerate}
      \item[(i)] $\mathbf{F}_{k_1,\cdots, k_m}(j_1,\cdots, j_s)
      \equiv \mathbf{F}_{k'_1,\cdots, k'_r}(j_1,\cdots, j_s)$, and
      \n

      \item[(ii)] $\widetilde{\sigma}_{k_m} \circ\cdots\circ\widetilde{\sigma}_{k_1} =
        \widetilde{\sigma}_{k'_r} \circ\cdots\circ
        \widetilde{\sigma}_{k'_1}$ as a permutation on $[\pm n]$.
    \end{enumerate}
    \n

    By our discussion above, the necessity is easy to see.
    To prove the sufficiency, we only need to show the
    condition (i) and (ii) can guarantee that  $\tau_{\mathbf{F}(k_m)}\circ\cdots\circ
   \tau_{\mathbf{F}(k_1)}$ and $\tau_{\mathbf{F}(k'_r)}\circ\cdots\circ
   \tau_{\mathbf{F}(k'_1)}$ agree on $f$. Indeed, let
   $\widetilde{f}=\mathbf{F}_{k_1,\cdots, k_m}(j_1,\cdots, j_s)$.
    Then from (i) and (ii), we can conclude that
   $\tau_{\mathbf{F}(k_m)}\circ\cdots\circ
   \tau_{\mathbf{F}(k_1)}$ and $\tau_{\mathbf{F}(k'_r)}\circ\cdots\circ
   \tau_{\mathbf{F}(k'_1)}$ induce the same map from the face lattice of
   $f$ to that of $\widetilde{f}$. Since they are both isometries from $f$ to $\widetilde{f}$,
   so they must agree on any point of $f$. \nn

    Using the above interpretations of ``strongness'' and
    ``perfectness'' for regular facets-pairing structures on $\mathcal{C}^n$,
    we can prove the following theorem easily.\nn

   \begin{thm} \label{thm:Perfect-Strong}
     If a regular facets-pairing
    structure $\mathcal{F}$
     on $\mathcal{C}^n$ is perfect, then $\mathcal{F}$ must be strong.
   \end{thm}
   \begin{proof}
   Suppose $(k_1,\cdots, k_m)$ and $(k'_1,\cdots, k'_r)$
   are the derived sequences of two different sequences $(i_1,\cdots, i_m)$ and $(i'_1,\cdots, i'_{r})$ from
   $(j_1,\cdots, j_s)$ respectively with
    $\mathbf{F}_{k_1,\cdots, k_m}(j_1,\cdots, j_s)
   = \mathbf{F}_{k'_1,\cdots, k'_r}(j_1,\cdots, j_s)$. By
   Lemma~\eqref{Lem:Component-0}, we can assume that
   $(i_1,\cdots, i_m)$ and $(i'_1,\cdots, i'_{r})$ are both
   irreducible.\n

    Since we assume $\mathcal{F}$ is perfect,
     the components in the face family $\widehat{\mathbf{F}}(j_1,\cdots,j_s)$
     are one-to-one correspondent with all the subsets of $\{ 1,\cdots, s \}$.
     So we must have
      $\{ i_1,\cdots, i_m \} = \{ i'_1,\cdots,  i'_r \}$.
      This implies $m=r$ and
     $(i_1,\cdots, i_m)$ is just a permutation of the entries in $(i'_1,\cdots, i'_r)$.
     Then by Lemma~\ref{Lem:Order-Indep} and
      the above interpretation of the strongness of $\mathcal{F}$,
       we are done.
   \end{proof}
     \nn

  \noindent \textbf{Question 3:}
    when a strong regular facets-pairing structure $\mathcal{F}$
   on $\mathcal{C}^n$ is perfect?\nn

    To answer this question,
     we need to examine the properties of the face families in
   a strong regular facets-pairing structure $\mathcal{F}$ on
    $\mathcal{C}^n$ more carefully. For convenience,
    for any face $f$ of
    $\mathcal{C}^n$, let $| \widehat{f} |$ be the number of components in
    the face family $\widehat{f}$ of $\mathcal{F}$. \nn

     \begin{lem}\label{Lem:Component-1}
  In a strong regular facets-pairing
    structure $\mathcal{F}$ on $\mathcal{C}^n$, for any
    two faces $e\subset f$ in $\mathcal{C}^n$ with
    $\dim(f)-\dim(e)=1$,  then either
     $|\widehat{e}| = |\widehat{f}|$ or  $|\widehat{e}| = 2 |\widehat{f}|$.
  \end{lem}
  \begin{proof}
      Suppose $f=\mathbf{F}(j_1,\cdots, j_s)$ and $e= \mathbf{F}(j_1,\cdots, j_s, j_{s+1})$.
      By Lemma~\ref{Lem:Component-0},
       each component of $\widehat{e}$ must be one of the following two types:
      \begin{align*}
         \text{type-1:} \ \ &\mathbf{F}(j_1,\cdots, j_s, j_{s+1})^{\Sigma} ,\
          \forall\,\Sigma \subset \{ 1,\cdots, s\}. \\
         \text{type-2:} \ \ &\mathbf{F}(j_1,\cdots, j_s, j_{s+1})^{\Sigma\cup \{ s+1\}},
           \ \forall\, \Sigma \subset \{ 1,\cdots, s\}.
       \end{align*}
        There is a
   natural map between the set of type-1 and type-2 components of $\widehat{e}$ as following.
   For any set $\Sigma \subset \{ 1,\cdots, s \}$, we define
   \begin{equation} \label{Equ:Correspondence}
     \mathbf{F}(j_1,\cdots, j_s, j_{s+1})^{\Sigma} \ \longrightarrow\
      \mathbf{F}(j_1,\cdots, j_s, j_{s+1})^{\Sigma\cup \{ s+1\}}.
   \end{equation}
    But we need to show this map is well-defined,
    i.e. if $\Sigma'$ is another subset of $\{ 1,\cdots, s\}$ with
    $ \mathbf{F}(j_1,\cdots, j_s,j_{s+1})^{\Sigma'} = \mathbf{F}(j_1,\cdots, j_s,j_{s+1})^{\Sigma}$,
    then we must have
  \begin{equation} \label{Equ:Well-Defined}
     \mathbf{F}(j_1,\cdots, j_s, j_{s+1})^{\Sigma'\cup \{ s+1\}}
    = \mathbf{F}(j_1,\cdots, j_s, j_{s+1})^{\Sigma\cup \{ s+1\}}.
    \end{equation}

    Indeed,
    $\mathbf{F}(j_1,\cdots, j_s, j_{s+1})^{\Sigma\cup \{ s+1\}}$ and
    $\mathbf{F}(j_1,\cdots, j_s, j_{s+1})^{\Sigma'\cup \{ s+1\}}$
    are generated by $\{ s+1\}$ from $\mathbf{F}(j_1,\cdots,
    j_s,j_{s+1})^{\Sigma}$ and $ \mathbf{F}(j_1,\cdots,
    j_s,j_{s+1})^{\Sigma'}$ respectively (see Lemma~\ref{Lem:First-Simple}).
   Since $\mathcal{F}$ is strong,
   $\mathbf{F}(j_1,\cdots, j_s,j_{s+1})^{\Sigma'} = \mathbf{F}(j_1,\cdots, j_s,j_{s+1})^{\Sigma}$
   will force  $ \mathbf{F}(j_1,\cdots, j_s,j_{s+1})^{\Sigma'}
   \equiv \mathbf{F}(j_1,\cdots, j_s,j_{s+1})^{\Sigma}$.
   So~\eqref{Equ:Well-Defined} follows from Lemma~\ref{Lem:First-Simple}.
   So the map in~\eqref{Equ:Correspondence} is well-defined and it is
   obviously surjective. To show it is injective, we
   notice that  $\mathbf{F}(j_1,\cdots, j_s, j_{s+1})^{\Sigma}$ is
    generated by $\{ s+1\}$ from
    $\mathbf{F}(j_1,\cdots, j_s,j_{s+1})^{\Sigma \cup \{ s+1\}}$
     (see Lemma~\ref{Lem:Second-Simple-2} and~\eqref{Equ:Second-Simple-2}).
    Then by a similar argument as above, we can show:
    if $ \mathbf{F}(j_1,\cdots, j_s,j_{s+1})^{\Sigma'\cup \{ s+1\}} =
    \mathbf{F}(j_1,\cdots, j_s,j_{s+1})^{\Sigma \cup \{ s+1\}}$,
    we must have $ \mathbf{F}(j_1,\cdots, j_s,j_{s+1})^{\Sigma'} =
    \mathbf{F}(j_1,\cdots, j_s,j_{s+1})^{\Sigma}$. So
    the map defined in~\eqref{Equ:Correspondence} is injective, hence a bijection.
   So we can conclude that there are equal number of type-1 and type-2 components in
   $\widehat{e}$.    \n

   Similarly, we can define a natural
    map from the set of components in $\widehat{f}$ to the set of type-1 components in
    $\widehat{e}$ by: for any set $\Sigma \subset \{ 1,\cdots, s \}$
    \begin{equation} \label{Equ:Correspondence-3}
     \mathbf{F}(j_1,\cdots, j_s)^{\Sigma}    \ \longrightarrow\ \mathbf{F}(j_1,\cdots, j_s, j_{s+1})^{\Sigma}
     \end{equation}
      We can show that
       the map in~\eqref{Equ:Correspondence-3}
        is a well-defined bijection by a similar argument as we do for the
        map in~\eqref{Equ:Correspondence}.
       Therefore, the number of type-1 components (hence type-2 components)
       in $\widehat{e}$ exactly equals $|\widehat{f}|$.
       \n

       If the set of
        type-1 components and the set of type-2 components of $\widehat{e}$ are disjoint,
        then $|\widehat{e}| = 2 |\widehat{f}|$. Otherwise, there
        exist $\Sigma_1, \Sigma_2 \subset \{ 1,\cdots, s \}$ so that
         $
         \mathbf{F}(j_1,\cdots, j_s, j_{s+1})^{\Sigma_1}
          = \mathbf{F}(j_1,\cdots, j_s, j_{s+1})^{\Sigma_2\cup \{ s+1\}}
          $. Then since $\mathcal{F}$ is strong, we must have:
        \begin{equation} \label{Equ:proof-1}
          \mathbf{F}(j_1,\cdots, j_s, j_{s+1})^{\Sigma_1}
          \equiv \mathbf{F}(j_1,\cdots, j_s, j_{s+1})^{\Sigma_2\cup \{s+1\}}.
         \end{equation} \n

          In this case, we claim that the set of type-1 components and
             type-2 components in $\widehat{e}$ are the same set.
      So $|\widehat{e}| = |\widehat{f}|$. Indeed by Lemma~\ref{Lem:Second-Simple-2},
        the~\eqref{Equ:proof-1} implies that
        \begin{align*}
          \mathbf{F}(j_1,\cdots, j_s, j_{s+1})  \equiv
        \left(\mathbf{F}(j_1,\cdots, j_s, j_{s+1})^{\Sigma_1} \right)^{\Sigma_1}
          & \equiv \mathbf{F}(j_1,\cdots, j_s, j_{s+1})^{(\Sigma_2\cup \{s+1\}) \ominus
          \Sigma_1} \\
          &\equiv \mathbf{F}(j_1,\cdots, j_s, j_{s+1})^{(\Sigma_1 \ominus \Sigma_2 )\cup \{s+1\}}
         \end{align*}

         Then for any set $\Sigma \subset \{ 1,\cdots, s \}$, we have
          $$ \mathbf{F}(j_1,\cdots, j_s, j_{s+1})^{\Sigma}
          \equiv \mathbf{F}(j_1,\cdots, j_s, j_{s+1})^{((\Sigma_1 \ominus \Sigma_2 ) \ominus \Sigma)
            \cup \{s+1\}}. $$
            So any type-1 component in $\widehat{e}$ is also type-2.
        Similarly, we have:
         \begin{align*}
           \mathbf{F}(j_1,\cdots, j_s, j_{s+1})^{\Sigma \cup \{ s+1 \}}
             & \equiv \mathbf{F}(j_1,\cdots, j_s, j_{s+1})^{((\Sigma_1 \ominus \Sigma_2 )
            \cup \{s+1\})\ominus (\Sigma \cup \{ s+1 \})}\\
            &\equiv \mathbf{F}(j_1,\cdots, j_s, j_{s+1})^{(\Sigma_1 \ominus \Sigma_2 ) \ominus \Sigma}
         \end{align*}
        So any type-2 component in $\widehat{e}$ is also type-1.
        So the claim is proved.
      \end{proof}
      \nn

     Notice that for any facet $\mathbf{F}(j)$ of
      $\mathcal{C}^n$, $|\widehat{\mathbf{F}}(j)| = 1$ or $2$.
      So the above lemma implies the following. \nn

  \begin{thm} \label{thm:Component-1}
   If $\mathcal{F}$ is a strong regular facets-pairing structure on
   $\mathcal{C}^n$, then the number of components in any face family
   in $\mathcal{F}$ is a power of $2$.
 \end{thm}
 \nn

  \begin{rem}
   If we do not assume $\mathcal{F}$ is strong in the statement of Theorem~\ref{thm:Component-1},
   it is not clear whether the conclusion in Theorem~\ref{thm:Component-1} still holds or not.
   So it is interesting to find an example of
   regular facets-pairing structure
   on $\mathcal{C}^n$ which has a face family $\widehat{f}$ with
   $|\widehat{f}|$ not a power of $2$.
 \end{rem}
  \n

  In addition, Lemma~\ref{Lem:Component-1} implies the
  following.

  \begin{cor}\label{Cor:Number-Components-2}
   If $\mathcal{F}$ is a strong regular facets-pairing structure on
   $\mathcal{C}^n$, for any faces $e\subset f$ in $\mathcal{C}^n$,  we always have
     $|\widehat{e}| \leq 2^{\dim(f)-\dim(e)} |\widehat{f}|$.
  \end{cor}
  \n

  By the discussion above, we can give an answer to the Question 3 as following.\nn

    \begin{thm} \label{thm:Main-3}
       A strong regular facets-pairing
    structure $\mathcal{F}$
     on $\mathcal{C}^n$ is perfect if and only if
     all the $2^n$ vertices of $\mathcal{C}^n$ form
     a unique $0$-dimensional face family.
    \end{thm}
  \begin{proof}
      Suppose all the
       $2^n$ vertices of $\mathcal{C}^n$ form a unique $0$-dimensional
       face family of $\mathcal{F}$.
     For any face $f$ of $\mathcal{C}^n$, since $f$ contains some vertices of $\mathcal{C}^n$,
     then Corollary~\ref{Cor:Number-Components-2} implies that
      $  2^n \leq 2^{\dim(f)}\cdot |\widehat{f}| $.
     So $|\widehat{f}| \geq 2^{n-\dim(f)} $.
     On the other hand, Corollary~\ref{Cor:Number-Components} tells us that
     $|\widehat{f}| \leq 2^{n-\dim(f)}$.
      Hence $|\widehat{f}| = 2^{n-\dim(f)}$, which means that $\widehat{f}$
      is a perfect face family. So $\mathcal{F}$ is perfect.
  \end{proof}
   \nn

 \begin{rem} \label{Rem:Bun}
   In Theorem~\ref{thm:Main-3}, if we do not assume $\mathcal{F}$ is strong,
   the conclusion may not be true. For example, in Figure~\ref{p:Bun},
   we have a facets-pairing structure $\mathcal{F}$ on $\mathcal{C}^2$ whose facets are all
   self-involutive. The structure map of $\mathcal{F}$ on each
   edge of the square is just reflecting the edge about its
   midpoint. By our definition, $\mathcal{F}$ is not strong
   (see Remark~\ref{Rem:Strong-Self-Invol-Facets}).
   So although the $0$-dimensional face family $\mathcal{F}$
   is perfect (i.e. the four vertices of $\mathcal{C}^2$ form a unique $0$-dimensional face family
   of $\mathcal{F}$),
    we can not deduce that $\mathcal{F}$ is perfect. In addition,
    the seal space of $\mathcal{F}$ is homeomorphic to $S^2$ (see Figure~\ref{p:Bun}), where
    the seal map is just doubling up each edge of the square.
   \\
  \end{rem}
     \begin{figure}
         \includegraphics[width=0.64\textwidth]{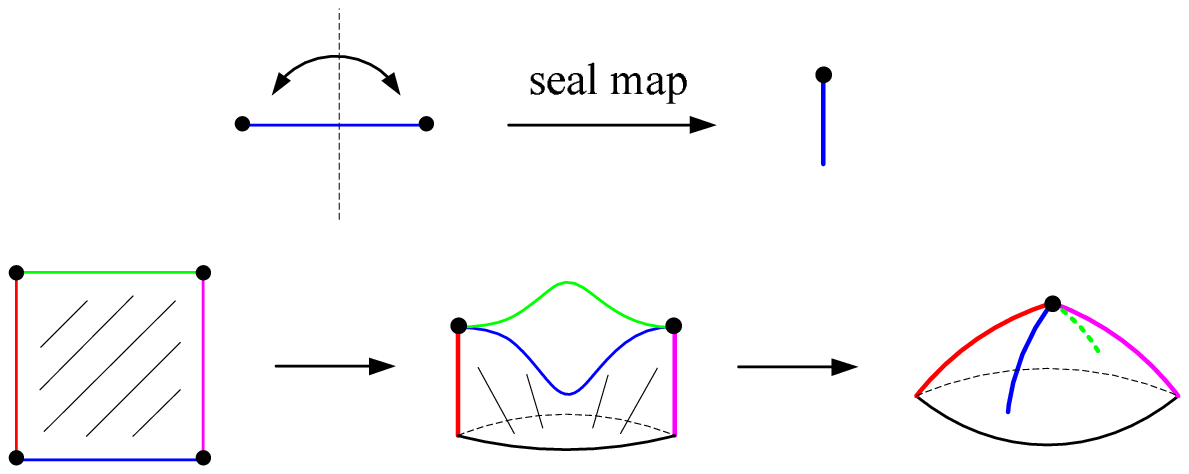}
          \caption{  }\label{p:Bun}
      \end{figure}

  \section{Seal Space of Perfect Regular Facets-Pairing Structure} \label{Section5}
  At the beginning of this paper, we ask what closed
  manifolds can be obtained from the seal spaces of facets-pairing
  structures on a cube (see Question 2). We have seen in Figure~\ref{p:three-Panel-Structures}
  that we can obtain torus, Klein bottle (admitting flat Riemannian
  metric), $\R P^2$ and $S^2$ from $\mathcal{C}^2$ in this way.
  In this section, we will show that
  the seal space of
  any perfect regular facets-pairing structure on a cube is always a closed
   manifold which admits a flat Riemannian metric. \nn

 Suppose $\mathcal{F} = \{ \omega, \tau_j \}_{j\in [\pm n]}$
  is a perfect regular facets-pairing structure on
 $\mathcal{C}^n$.  First, let us analyze the local picture of the
 seal map $\zeta_{\mathcal{F}}$ around
  any point on $\mathcal{C}^n$.
  By our discussion in the previous section,
 for any face $\mathbf{F}(j_1,\cdots, j_s)$ of $\mathcal{C}^n$, the
 face family
  $\widehat{\mathbf{F}}(j_1,\cdots, j_s)
   =\{ \mathbf{F}(j_1,\cdots, j_s)^{\Sigma}\, ; \, \Sigma \subset\{ 1,\cdots, s \}\}$.
  Since $\mathcal{F}$ is perfect,
   the components of $\widehat{\mathbf{F}}(j_1,\cdots, j_s)$ are one-to-one
  correspondent with all the subsets of $\{ 1,\cdots, s\}$.
  So the seal map $\zeta_{\mathcal{F}}$ will glue the $2^s$ different faces
  $\{ \mathbf{F}(j_1,\cdots, j_s)^{\Sigma}\, ; \, \Sigma \subset\{ 1,\cdots, s \} \}$
 into one face in the seal space $Q^n_{\mathcal{F}}$. Before we
 study the shape of $Q^n_{\mathcal{F}}$ around
 $\zeta_{\mathcal{F}}(\mathbf{F}(j_1,\cdots, j_s))$, let us first see
 a standard way of gluing $2^s$ right-angled cones in the Euclidean space $\R^n$.
 \nn

  For any $ 1\leq s\leq n$, let $\mathbb{L}^{n-s}$ be an $(n-s)$-dimensional linear subspace of $\R^n$
 defined by
 $\mathbb{L}^{n-s} = \{ (x_1,\cdots, x_n) \in \R^n \, | \ x_1=0, \cdots ,  x_s =0 \}$.
  So $ \mathbb{L}^{n-s} $ is the intersection of coordinate
  hyperplanes $H_1,\cdots, H_s$ in $\R^n$ where
  \[ H_i = \{ (x_1,\cdots, x_n) \in \R^n \, | \, x_i = 0 \}, \  i=1,\cdots, s. \]
  Note that the hyperplanes $H_1,\cdots, H_s$ divide $\R^n$ into $2^s$ connected
  domains
  $\{  \mathbb{P}_{\Sigma} \, ; \, \forall\ \text{set}\ \Sigma\subset \{ 1,\cdots, s\}
  \}$ where
   \[   \mathbb{P}_{\Sigma} := \{\, (x_1,\cdots, x_n) \in \R^n \ | \
      x_l > 0, \forall \, l\in \Sigma\ \text{and} \ x_{l'} < 0,
      \forall \, l'\in \{1,\cdots,s\} \backslash \Sigma \, \}
      \]
  Choose a small $\delta > 0$, let $H_i \times (-\delta, \delta)$
  be an open neighborhood of $H_i$ in $\R^n$. If we remove
  $\overset{s}{\underset{i=1}{\bigcup}} H_i \times (-\delta, \delta)$ from $\R^n$, we get
  $2^s$ components $\{ B_{\Sigma}\, ; \, \Sigma \subset \{ 1,\cdots, s\} \}$
  where each $B_{\Sigma}$ is a right-angled cone contained in $ \mathbb{P}_{\Sigma}$ (see Figure~\ref{p:Local-Gluing}).
  We denote all the facets of $B_{\Sigma}$ by $\{ B_{\Sigma}(i), 1\leq i \leq s \}$ where
  $B_{\Sigma}(i)$ is the facet of $B_{\Sigma}$ parallel to $H_i$.
  Moreover, for any nonempty set $\{ i_1,\cdots, i_q \} \subset \{1,\cdots, s\}$, we define
  $$ B_{\Sigma}(i_1,\cdots, i_q) := B_{\Sigma}(i_1) \cap \cdots \cap
  B_{\Sigma}(i_q). $$
   Then $B_{\Sigma}(i_1,\cdots, i_q)$ is a codimension-$q$ face of $B_{\Sigma}$
  which is parallel to $H_{i_1},\cdots, H_{i_q}$. Obviously, any
  codimension-$q$ face of $B_{\Sigma}$ can be written in this
  form.\n

      \begin{figure}
         \includegraphics[width=0.31\textwidth]{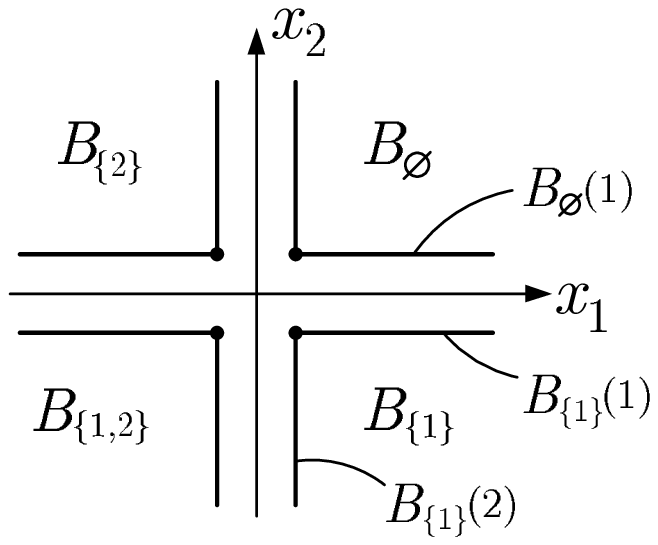}
          \caption{  }\label{p:Local-Gluing}
      \end{figure}

   For any $1\leq l \leq s$, let $\beta_l : \R^n \rightarrow \R^n$ denote the reflection of
   $\R^n$ about the hyperplane $H_l$. It is easy to see that
    $\beta_l$ will map $B_{\Sigma}$ to $B_{\Sigma \ominus \{ l\}}$
    and preserve their manifold with corners structures.
   Moreover, if $ l \in \{ i_1,\cdots, i_q \} \subset \{1,\cdots, s\}$,
    \begin{equation} \label{Equ:Compare-1}
     \beta_l(B_{\Sigma}(i_1,\cdots, i_q)) = B_{\Sigma \ominus \{ l \} }(i_1,\cdots, i_q).
    \end{equation}
    \n

   If we glue the $2^s$ right-angled cones $\{ B_{\Sigma}\, ; \, \Sigma \subset \{ 1,\cdots, s\} \}$
   together along their facets by: for any set $\Sigma \subset \{ 1,\cdots,
   s\}$ and  $\{ i_1,\cdots, i_q \} \subset \{1,\cdots, s\}$,
   glue any point $x\in B_{\Sigma}(i_1,\cdots, i_q)$ with $\beta_{i_1}(x),\cdots,
   \beta_{i_q}(x)$, the quotient space can obviously be identified with $\R^n$
   and, the image of each $B_{\Sigma}$ can be identified with the closure of
   $\mathbb{P}_{\Sigma}$ in $\R^n$.\nn

    \begin{figure}
         \includegraphics[width=0.57\textwidth]{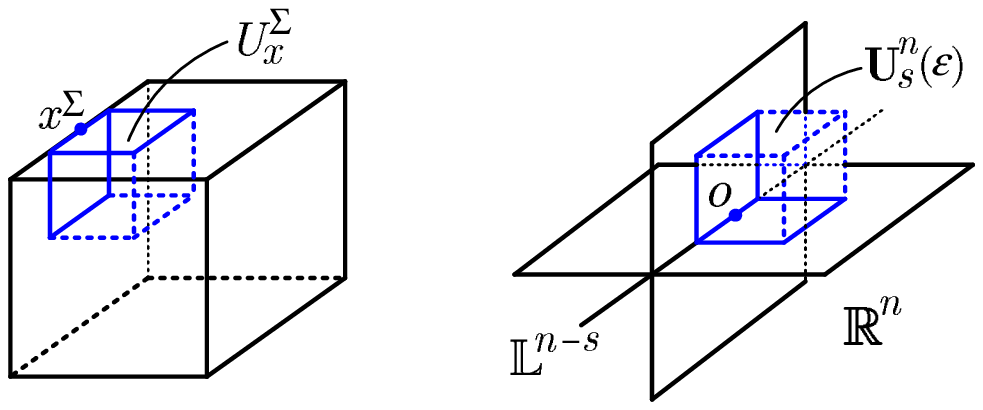}
          \caption{  }\label{p:Cubic-U}
      \end{figure}

    Next, let us see how the seal map $\zeta_{\mathcal{F}}$
    glues all the components of $\widehat{\mathbf{F}}(j_1,\cdots, j_s)$ together.
    By our discussion in section~\ref{Section4},
    for any set $\Sigma =\{ i_1,\cdots, i_r\} \subset\{ 1,\cdots, s \}$,
     $$ \mathbf{F}(j_1,\cdots, j_s)^{\Sigma} =
   \widetilde{\tau}_{k_r}\circ \cdots \circ
     \widetilde{\tau}_{k_1} (\mathbf{F}(j_1,\cdots, j_s))
     $$
     where $(k_1,\cdots, k_r)$ is the derived sequence of $(i_1,\cdots, i_r)$
    from $(j_1,\cdots, j_s)$. For any point $x$ in the
    relative interior of $\mathbf{F}(j_1,\cdots, j_s)$,
     we define its family point
     \[  x^{\Sigma} := \widetilde{\tau}_{k_r}\circ \cdots \circ
     \widetilde{\tau}_{k_1} (x)
      \in \mathbf{F}(j_1,\cdots, j_s)^{\Sigma}\]
     By our discussion in section~\ref{Section4},
    the definition of $x^{\Sigma}$
    is independent on the order of $i_1,\cdots, i_r$, hence determined only by $\Sigma$ and
    $(j_1,\cdots, j_s)$.
    For each $\Sigma$, choose a cubic open neighborhood $U^{\Sigma}_x$ of $x^{\Sigma}$
    in $\mathcal{C}^n$ so that there exists an isometry from $U^{\Sigma}_x$ to
    a standard
    half-open cube $\mathbf{U}^n_s(\varepsilon)$ in $\R^n$,
    which maps $x^{\Sigma}$ to the origin.
  \[ \mathbf{U}^n_s(\varepsilon) := \{ (x_1,\cdots, x_n)\in \R^n\, |\,
     0\leq x_i < \varepsilon, 1\leq i \leq s\ \text{and}\,
     -\frac{\varepsilon}{2} < x_{j} < \frac{\varepsilon}{2}, s+1 \leq j \leq
     n\} \]
     Moreover, we can choose $\varepsilon$ to be small enough so that
     $U^{\Sigma}_x \cap U^{\Sigma'}_x =\varnothing$ for any $\Sigma \neq \Sigma'$.
     Note that
     $U^{\Sigma}_x$ is a nice manifold with corners whose faces lie in
     $U^{\Sigma}_x\cap \partial \mathcal{C}^n$ (see Figure~\ref{p:Cubic-U}).
    The seal map $\zeta_{\mathcal{F}}$ will glue the the $2^s$ half-open cubes $\{ U^{\Sigma}_x\, ; \, \Sigma
    \subset\{1,\cdots, s\} \}$ along their boundaries via the structure map
    $\{  \tau_j \}$ of
    $\mathcal{F}$. So the shape of  $Q^n_{\mathcal{F}}$ around $\zeta_{\mathcal{F}}(x)$
    is completely determined by how these $U^{\Sigma}_x$ are glued. Next, we
    make a comparison between the gluing of
     $\{ U^{\Sigma}_x\, ; \, \Sigma \subset\{1,\cdots, s\} \}$
      by $\{ \tau_j\}$ and the above standard gluing
     of $\{ B_{\Sigma}\, ; \, \Sigma \subset\{1,\cdots, s\} \}$ by $\{ \beta_l \}$ in $\R^n$.\nn

     To make the comparison more explicit, we define a synthetic notation as follows.
     For a fixed sequence $(j_1,\cdots, j_s)$ and any
      set $\Sigma =\{ i_1,\cdots, i_r\} \subset\{ 1,\cdots, s \}$,
      let $(k_1,\cdots, k_r)$ be the derived sequence of $(i_1,\cdots, i_r)$
    from $(j_1,\cdots, j_s)$. Then define
      \[   j^{\Sigma}_i := \widetilde{\sigma}_{k_r}\circ \cdots \circ
         \widetilde{\sigma}_{k_1} (j_i),\ i=1,\cdots, s.
     \]
    By Lemma~\ref{Lem:Order-Indep}, $j^{\Sigma}_i$ depends only on
    the set $\Sigma$ and $(j_1,\cdots, j_s)$. Then we have:
    \[   \mathbf{F}(j_1,\cdots, j_s)^{\Sigma} \equiv \mathbf{F}(j^{\Sigma}_1,\cdots, j^{\Sigma}_s)
    =\mathbf{F}(j^{\Sigma}_1)\cap \cdots \cap \mathbf{F}(j^{\Sigma}_s). \]
    Then all the facets of $U^{\Sigma}_x$ are
     $U^{\Sigma}_x(j^{\Sigma}_i) = U^{\Sigma}_x\cap
     \mathbf{F}(j^{\Sigma}_i)$, $i=1,\cdots, s$. Moreover,
     for any nonempty set $\{ i_1,\cdots, i_q \} \subset \{1,\cdots, s\}$, we define
     $$U^{\Sigma}_x(j^{\Sigma}_{i_1},\cdots, j^{\Sigma}_{i_q}) := U^{\Sigma}_x(j^{\Sigma}_{i_1})
      \cap \cdots \cap U^{\Sigma}_x(j^{\Sigma}_{i_q})=
      U^{\Sigma}_x \cap \mathbf{F}(j^{\Sigma}_{i_1},\cdots, j^{\Sigma}_{i_q}).$$
   Obviously, $U^{\Sigma}_x(j^{\Sigma}_{i_1},\cdots, j^{\Sigma}_{i_q})$
   is a codimension-$q$ face of $U^{\Sigma}_x$ and any codimension-$q$ face of $U^{\Sigma}_x$ can be written in this
    form. In addition, for any $l\in \{ i_1,\cdots, i_q\}$,
    \begin{equation} \label{Equ:Compare-2}
 \text{we claim:}\ \  \tau_{j^{\Sigma}_l}\left( U^{\Sigma}_x(j^{\Sigma}_{i_1},\cdots, j^{\Sigma}_{i_q}) \right) =
    U^{\Sigma\ominus \{l\}}_x(j^{\Sigma\ominus \{l\}}_{i_1},\cdots, j^{\Sigma\ominus \{l\}}_{i_q}).
    \end{equation}
     Indeed, by the above definition of $j^{\Sigma}_l$ and Lemma~\ref{Lem:Second-Simple-2}, we have
    \begin{align*}
  \tau_{j^{\Sigma}_l}
    \left(\mathbf{F}(j^{\Sigma}_1,\cdots, j^{\Sigma}_s) \right)
     & =
    \mathbf{F}(\widetilde{\sigma}_{j^{\Sigma}_l} (j^{\Sigma}_1),\cdots,
     \widetilde{\sigma}_{j^{\Sigma}_l} (j^{\Sigma}_s)) \\
    & \equiv \mathbf{F}(j_1,\cdots, j_s)^{\Sigma\ominus \{l\}}
    \equiv \mathbf{F}(j^{\Sigma\ominus \{l\}}_1,\cdots, j^{\Sigma\ominus \{l\}}_s)
     \end{align*}
      $$ \ \Longrightarrow\ \ \widetilde{\sigma}_{j^{\Sigma}_l} (j^{\Sigma}_i) = j^{\Sigma\ominus
     \{l\}}_i,\ i=1,\cdots, s.$$
    \begin{align*}
    \mathrm{So} \ \ \tau_{j^{\Sigma}_l}\left(
    \mathbf{F}(j^{\Sigma}_{i_1},\cdots, j^{\Sigma}_{i_q}) \right) =
    \mathbf{F}(\widetilde{\sigma}_{j^{\Sigma}_l} (j^{\Sigma}_{i_1}),\cdots,
     \widetilde{\sigma}_{j^{\Sigma}_l} (j^{\Sigma}_{i_q})) =
     \mathbf{F}(j^{\Sigma\ominus \{l\}}_{i_1},\cdots, j^{\Sigma\ominus \{l\}}_{i_q}).
    \end{align*}
    Notice that $\tau_{j^{\Sigma}_l}$ sends $\mathbf{F}(j_1,\cdots, j_s)^{\Sigma}$
    to $\mathbf{F}(j_1,\cdots, j_s)^{\Sigma\ominus \{l\}}$, so
    $\tau_{j^{\Sigma}_l}(x^{\Sigma}) = x^{\Sigma\ominus \{ l\}}$.
    Then since $\tau_{j^{\Sigma}_l}$ is an isometry with respect to the Euclidean metric on the faces,
     it will map
   $U^{\Sigma}_x \cap \mathbf{F}(j^{\Sigma}_{i_1},\cdots, j^{\Sigma}_{i_q})$
   onto $U^{\Sigma\ominus \{l\}}_x \cap
   \mathbf{F}(j^{\Sigma\ominus \{l\}}_{i_1},\cdots, j^{\Sigma\ominus
   \{l\}}_{i_q})$, which confirms our claim. \nn

    Next, for each $\Sigma \subset \{ 1,\cdots, s \}$ we choose a face-preserving homeomorphism
       $$\Phi^{\Sigma} : U^{\Sigma}_x \rightarrow B_{\Sigma}$$
       so that $\Phi^{\Sigma}(U^{\Sigma}_x(j^{\Sigma}_i)) = B_{\Sigma}(i)$ for any
       $1\leq i \leq s$. Then for any nonempty set
       $\{ i_1,\cdots, i_q \} \subset \{1,\cdots, s\}$, we have
       $$\Phi^{\Sigma} \left(U^{\Sigma}_x(j^{\Sigma}_{i_1},\cdots, j^{\Sigma}_{i_q})\right) =
        B_{\Sigma}(i_1,\cdots, i_q).  $$

       By comparing~\eqref{Equ:Compare-1} and~\eqref{Equ:Compare-2} it is easy to see
       that: for any $l\in \{ i_1,\cdots, i_q\}$
    \begin{equation*}
    \beta_l \left(\Phi^{\Sigma} \left(U^{\Sigma}_x(j^{\Sigma}_{i_1},\cdots, j^{\Sigma}_{i_q})\right) \right) =
          \Phi^{\Sigma\ominus \{l\}} \left(\tau_{j^{\Sigma}_l}
            \left(U^{\Sigma}_x(j^{\Sigma}_{i_1},\cdots, j^{\Sigma}_{i_q}) \right)
            \right) =   B_{\Sigma\ominus \{ l\}}(i_1,\cdots, i_q).
    \end{equation*}
    In other words, the set of homeomorphisms
    $\{\Phi^{\Sigma}: U^{\Sigma}_x \rightarrow B_{\Sigma} \}$
     are equivariant with respect to the action of $\{ \tau_j \}$
     and $\{ \beta_l \}$ on the faces of $\{ U^{\Sigma}_x\}$
     and $\{  B_{\Sigma} \}$. So when going down to the quotient space,
      $\{\Phi^{\Sigma}:  U^{\Sigma}_x \rightarrow B_{\Sigma} \}$
      induces a homeomorphism from
     a neighborhood of $\zeta_{\mathcal{F}}(x)$ in $Q^n_{\mathcal{F}}$
     to a neighborhood of the origin in $\R^n$.
    So we conclude that $Q^n_{\mathcal{F}}$ is a closed manifold.
     Moreover, we can show that
     $Q^n_{\mathcal{F}}$ admits a flat Riemannian metric below.\nn

 \begin{thm} \label{thm:flat-metric}
       For any perfect regular facets-pairing structure
     $\mathcal{F}$ on $\mathcal{C}^n$ ($n\geq 2$), the seal space $Q^n_{\mathcal{F}}$ is a
       closed manifold which admits a flat Riemannian metric.
    \end{thm}
    \begin{proof}
     As a quotient space of $\mathcal{C}^n$,
   let $Q^n_{\mathcal{F}}$ be equipped with the quotient
    metric $d_R$ induced from the Euclidean metric $d$ on $\mathcal{C}^n$.
   Obviously, the metric $d$ on $\mathcal{C}^n$ is intrinsic, in
   other words $(\mathcal{C}^n,d)$ is a length space. So $(Q^n_{\mathcal{F}},d_R)$ is also
  a length space (see chapter 3.1 of~\cite{Burago2001}). \n

   In addition, as we have shown above,
    for any face $\mathbf{F}(j_1,\cdots, j_s)$ of $\mathcal{C}^n$ and any point $x$ in
    the relative interior of $\mathbf{F}(j_1,\cdots, j_s)$, the way
    how the structure map $\{ \tau_j\}$ of $\mathcal{F}$ fit
     the cubic neighborhood $\{ U^{\Sigma}_x \}$ of
    all the family points of $x$
    together is exactly the same as gluing
    $2^s$ right-angled cones $\{ B_{\Sigma}\, ; \, \Sigma \subset \{ 1,\cdots, s\} \}$
    by $\{ \beta_l\}$ in $\R^n$.
   Since $\mathcal{F}$ is a regular facets-pairing structure, each
   $ \tau_j $ is isometric with respect to the
   Euclidean metric on the facets of $\mathcal{C}^n$. So
    the neighborhood of $\zeta_{\mathcal{F}}(x)$
  in $Q^n_{\mathcal{F}}$ with the quotient metric is isometric to an open
  neighborhood of the origin in the Euclidean space.
  So the metric $d_R$ on $Q^n_{\mathcal{F}}$ is locally a Riemannian
  flat metric.
  Then there exists a global flat Riemannian metric $g$ on $Q^n_{\mathcal{F}}$ which induces the
  quotient metric $d_R$ (see ~\cite{KobaNomizu} or~\cite{Burago2001}).
  \end{proof}

 By the above theorem, we can
 construct many closed flat Riemannian manifolds from perfect facets-pairing structures of cubes.
 This construction should be useful for us to understand the geometry and topology of these
 manifolds in the future.\nn

  If we only assume the facets-pairing structure $\mathcal{F}$ on $W^n$ is strong,
     the seal space $Q^n_{\mathcal{F}}$ may not be a closed manifold (see
    Example~\ref{Exam:Trivial} and Example~\ref{Exam:Strang-Strong-2}).
   On the other hand,
     there are many strong but non-perfect facets-pairing structures whose seal spaces
     are indeed closed manifolds.
     We will see examples of such kind of facets-paring structures on cubes
     in section~\ref{Section8} and section~\ref{Section9}. \\

 \section{Real Bott Manifolds as the seal spaces of cubes}\label{Section6}

   In this section, we will study a
  special class of real toric manifolds called real Bott manifolds.
  We will relate real Bott manifolds to some special type of regular facets-pairing
  structures on a cube.  Recall that a manifold $M$ is called a \textit{real Bott manifold} if there
  is a sequence of $\R P^1$ bundles as following.
  \begin{equation} \label{Equ:Real-Bott}
       M = M_n \overset{\R P^1}{\longrightarrow} M_{n-1}
   \overset{\R P^1}{\longrightarrow} \cdots \overset{\R P^1}{\longrightarrow}
   M_1 \overset{\R P^1}{\longrightarrow} M_0=\{ \text{a point} \}
   \end{equation}
  where each $\pi_i: M_i \rightarrow M_{i-1}$ the projective bundle of
   the Whitney sum of two real line bundles over $M_{i-1}$.
   Another way to describe real Bott manifolds is using binary square matrices
  (see~\cite{MasKam09-1} and ~\cite{SuMasOum10}). A binary matrix
  is a matrix with all its entries in $\Z_2 = \Z\slash 2\Z $.
  For any binary square matrix $A$, let $A^i_j \in \Z_2$  denote the $(i,j)$ entry of
     $A$ and let $A^i$ and $A_j$ be the $i$-th row and $j$-column vector of
     $A$. \nn

   A binary square matrix $A$ is called a \textit{Bott matrix}
     if it is conjugate to a strictly upper triangular
    binary matrix via a permutation matrix. Obviously, $A$ is a Bott matrix will imply that
    all the diagonal entries of $A$ are zero.
    We use $\mathfrak{B}(n)$ to denote the set of all $n\times n$ Bott matrices.
  \nn

    Generally, for any $n\times n $ binary matrix $A$ with zero diagonal,
     we can define a set of Euclidean motions $s_1,\cdots, s_n$ on $\R^n$ by:
  \[ s_i(x_1,\cdots,x_n):= ((-1)^{A^i_1}x_1, \cdots, (-1)^{A^i_{i-1}}x_{i-1}, x_i + \frac{1}{2},
     (-1)^{A^i_{i+1}}x_{i+1}, \cdots, (-1)^{A^i_n}x_n ) \]
   Let $\Gamma(A)$ be the discrete subgroup of $\mathrm{Isom}(\R^n)$
   generated by $s_1, \cdots, s_n$ and let
   $M(A) = \R^n \slash \Gamma(A)$. It turns
    out that $M(A)$ is a real Bott manifold when $A$ is a Bott matrix. Conversely,
    any real Bott manifold
    can be obtained in this way.   \nn

  On the other hand, for any binary square matrix $A$ with zero diagonal,
   we can define a
    set of signed permutations of $[\pm n]$ from $A$ by:
     \begin{align}\label{Equ:FPS-A}
       \omega (j) &= -j, \ \  \forall\, j\in [\pm n]  \notag \\
      \widetilde{\sigma}^A_j ( k ) &= \left\{
                                     \begin{array}{ll}
                                         (-1)^{A^{|j|}_{|k|}}\cdot k, & \hbox{$k\in [\pm n],\ k\neq \pm j$;} \\
                                       -k, & \hbox{$k=\pm j$.}
                                     \end{array}
                                   \right.
      \end{align}

    It is easy to see that each
    $\widetilde{\sigma}^A_j$ is a signed permutation on $[\pm n]$ and we have: \n
   \begin{itemize}
     \item $\widetilde{\sigma}^A_{-j}= \widetilde{\sigma}^A_j$ for any $j \in [\pm n]$;\n

     \item  $\widetilde{\sigma}^A_{j_1} \circ \widetilde{\sigma}^A_{j_2} =
           \widetilde{\sigma}^A_{j_2} \circ \widetilde{\sigma}^A_{j_1}$ for any
               $j_1, j_2 \in [\pm n]$. \nn
   \end{itemize}
  Notice that if we define $\widetilde{A} = A + I_n$, where $I_n$ is
  the identity matrix, then
  $$\widetilde{\sigma}^A_j ( k ) = (-1)^{\widetilde{A}^{|j|}_{|k|}}\cdot k ,
   \ \forall\, j,k\in [\pm n].$$
By Theorem~\ref{thm:Main-2-Classify}, we can easily verify that
    this $\omega$ and $\{ \widetilde{\sigma}^A_j \}_{j\in [\pm n]}$ determine a regular
  facets-pairing structure on $\mathcal{C}^n$, denoted by
  $\mathcal{F}_A$.
   Let
  $\widetilde{\tau}^A_j : \mathcal{C}^n \rightarrow \mathcal{C}^n$
  be the corresponding symmetry of $\mathcal{C}^n$, i.e.
  $\widetilde{\tau}^A_j(\mathbf{F}(k)) = \mathbf{F}(\widetilde{\sigma}^A_j(k))$
  for $\forall\, k \in [\pm n]$.
  In addition, the structure map $\tau^A_j : \mathbf{F}(j) \rightarrow \mathbf{F}(\omega(j)) = \mathbf{F}(-j)$
   of $\mathcal{F}_A$ is given by
    $$\tau^A_j(\mathbf{F}(j, k)) = \mathbf{F}(-j,\widetilde{\sigma}^A_j(k)), \
    \forall\, k \neq \pm j .$$
  More specifically, for $\forall\, j\in [\pm n]$ and $\forall\, x=(x_1,\cdots, x_n) \in \mathbf{F}(j)$, we have:
   \begin{align} \label{Equ:tau_j-Cube}
     &\ \ \ \, \tau^A_j (x_1,\cdots,x_n) \notag \\
      &=
      ((-1)^{A^{|j|}_1}x_1, \cdots, (-1)^{A^{|j|}_{|j|-1}}x_{|j|-1}, -x_{|j|},
     (-1)^{A^{|j|}_{|j|+1}}x_{|j|+1}, \cdots, (-1)^{A^{|j|}_n}x_n )
   \end{align}
  Since
  $\tau^A_j$ only makes some sign changes to each coordinate of $x$,
  So intuitively such an $\mathcal{F}_A$ is the easiest type of
   facets-pairing structure on the cube. In the rest of this paper,
   we will mainly study such $\mathcal{F}_A$'s and try to answer the Question 1 and
   Question 2 for these facets-pairing structures. First, we see how
   to classify $\mathcal{F}_A$ up to the equivalence of
   facets-pairing structures on $\mathcal{C}^n$. \n

  \begin{thm}\label{thm:F_A-Classify}
    Suppose $A$ and $B$ are two $n\times n$ binary matrices with
    zero diagonal. Then $\mathcal{F}_A$ and $\mathcal{F}_B$ are equivalent
    facets-pairing structures if and only
    if there exists an $n\times n$ permutation matrix $P$ so that $B = P^{-1}AP$.
  \end{thm}
  \begin{proof}
     By Theorem~\ref{thm:Main-2-Classify}, $\mathcal{F}_B$ is equivalent to
      $\mathcal{F}_A$ if and only if there exists a signed
      permutation $\eta$ on $[\pm n]$ so that
      $\omega= \eta^{-1} \omega \eta$ and
      \begin{equation} \label{Equ:A-B-Equiv-0}
       \widetilde{\sigma}^B_j(k)= \eta^{-1} \widetilde{\sigma}^A_{\eta(j)}
      \eta(k),\ \forall\, j,k \in [\pm n] .
      \end{equation}
      It is clear that $\omega= \eta^{-1} \omega
      \eta$ always holds. Let $\widetilde{A} = A+I_n$, $\widetilde{B} = B+ I_n$.
       Then by the definition of $\{ \widetilde{\sigma}^A_j \}_{j\in [\pm n]}$ and
      $\{ \widetilde{\sigma}^B_j \}_{j\in [\pm n]}$, we have
      \begin{equation} \label{Equ:A-B-Equiv-1}
        \widetilde{\sigma}^B_j(k)= \eta^{-1} \widetilde{\sigma}^A_{\eta(j)}
        \eta(k) \ \Longleftrightarrow \ \widetilde{B}^{|j|}_{|k|} =  \widetilde{A}^{|\eta(j)|}_{|\eta(k)|}
        \in \Z_2 \  \Longleftrightarrow \  B^{|j|}_{|k|} =
        A^{|\eta(j)|}_{|\eta(k)|} \in \Z_2
       \end{equation}
 Let $\widehat{\eta}$ be a transformation on the set $\{ 1,\cdots , n\}$
    defined by:
       $$\widehat{\eta}(i) = |\eta(i)|,\  1\leq i \leq n . $$
    Since $\eta$ is a signed permutation on $[\pm n]$, it is easy
    to see that $\widehat{\eta}$ is a bijection, i.e.
    $\widehat{\eta}$ is a permutation on $\{ 1,\cdots , n\}$.
    Then the right side of~\eqref{Equ:A-B-Equiv-1} is
     equivalent to
       \begin{equation} \label{Equ:A-B-Equiv-2}
       {B^j_k} = { A^{\widehat{\eta}(j)}_{\widehat{\eta}(k)} }
        \in \Z_2, \ 1\leq\forall\, j,k \leq n .
      \end{equation}
    Let $P_{\widehat{\eta}}$ be the permutation matrix
    corresponding to $\widehat{\eta}$. Then by the fact that
    $P^{-1}_{\widehat{\eta}} = (P_{\widehat{\eta}})^{t}$,
     it is easy to see that~\eqref{Equ:A-B-Equiv-2}
   is exactly equivalent to $B = P_{\widehat{\eta}}^{-1} A
   P_{\widehat{\eta}}$.\n

   Conversely, if there exists an $n\times n$ permutation matrix $P$
   so that $B = P^{-1}AP$, let $\widehat{\eta}$ the
   permutation on $\{ 1,\cdots, n\}$ corresponding to $P$. Notice
   that $\widehat{\eta}$ canonically determines a signed
   permutation $\eta$ on $[\pm n]$ by:
    $$\eta(\pm i) := \pm \widehat{\eta}(i), 1\leq \forall\, i \leq n.$$
    Then because $B = P^{-1}AP$, this $\eta$ will satisfy~\eqref{Equ:A-B-Equiv-0}.
    So $\mathcal{F}_A$ and $\mathcal{F}_B$ are equivalent facets-pairing structures
    on $\mathcal{C}^n$.
  \end{proof}
  \n

  \begin{lem}
   When $A$ is a Bott matrix,
   the seal space $Q^n_{\mathcal{F}_A}$ is homeomorphic to the real Bott manifold
    $M(A)$.
 \end{lem}
  \begin{proof}
    Notice that $\mathcal{C}^n$ is a fundamental domain of
    the $\Gamma(A)$ action on $\R^n$. So by the definitions of
    $M(A) = \R^n \slash \Gamma(A)$ and $Q^n_{\mathcal{F}_A}$,
    it is easy to see that they are homeomorphic.
  \end{proof}

  \nn

    In the rest of this section, we will investigate the relationship between the
    geometric properties of $\mathcal{F}_A$ and the algebraic properties of $A$ where $A$
    is a binary square matrix with zero diagonal.
    For convenience, we introduce some auxiliary notations as follows.
    For any $n\times n$ binary matrix $A$ with zero diagonal, let
       \begin{equation} \label{Equ:Auxil-Matrix}
      \widetilde{A} :=  A + I_n, \ \text{where}\  I_n \ \text{is the identity
     matrix}.
     \end{equation}
     For any $1\leq j_1 < \cdots < j_s \leq n$, let
    $\widetilde{A}^{j_1\cdots j_s}$ be the $s\times n$ submatrix of $\widetilde{A}$
   formed by the $j_1$-th, $\cdots, j_s$-th row vectors of
   $\widetilde{A}$, and
     we define an
    $s\times s$ submatrix of $\widetilde{A}$ by:
    \begin{equation} \label{Equ:Auxil-SubMatrix}
        \widetilde{A}^{j_1\cdots j_s}_{j_1 \cdots j_s} : = \begin{pmatrix}
       \widetilde{A}^{j_1}_{j_1}  & \cdots & \widetilde{A}^{j_1}_{j_s}  \\
       \cdots  & \cdots & \cdots \\
       \widetilde{A}^{j_s}_{j_1}  & \cdots & \widetilde{A}^{j_s}_{j_s}
       \end{pmatrix}
     \end{equation}
   Usually,
   $ \widetilde{A}^{j_1\cdots j_s}_{j_1 \cdots j_s} $
   is called a \textit{principal minor matrix} of $\widetilde{A}$ and
   its determinant $\det(\widetilde{A}^{j_1\cdots j_s}_{j_1 \cdots j_s})$ is called a
   \textit{principal minor} of $\widetilde{A}$.
   Note that $ \widetilde{A}^{j_1\cdots j_s}_{j_1 \cdots j_s} $ is
   a submatrix of $\widetilde{A}^{j_1\cdots j_s}$, so
   $\mathrm{rank}_{\Z_2}(\widetilde{A}^{j_1\cdots j_s}_{j_1 \cdots j_s})
   \leq \mathrm{rank}_{\Z_2}(\widetilde{A}^{j_1\cdots j_s})$.
    \nn

   \begin{lem} \label{Lem:Component-2}
  In a facets-pairing structure $\mathcal{F}_A$ on $\mathcal{C}^n$,
    for any face $\mathbf{F}(j_1,\cdots, j_s)$ with $1\leq |j_1| < \cdots < |j_s| \leq n$ and any set
   $\Sigma =\{ i_1,\cdots, i_r \} \subset \{ 1,\cdots, s \}$,
    \begin{equation} \label{Equ:Face-matrix-repre}
   \mathbf{F}(j_1,\cdots, j_s)^{\Sigma} \equiv
   \mathbf{F}((-1)^{\varepsilon_1} j_1,\cdots, (-1)^{\varepsilon_s} j_s)
     \end{equation}
   where $\varepsilon_p =  \widetilde{A}^{|j_{i_1}|}_{|j_p|} +
    \cdots +  \widetilde{A}^{|j_{i_r}|}_{|j_p|} \ (\mathrm{mod}\ 2) ,  \ 1\leq p \leq s$.
    In other words, the vector $(\varepsilon_1,\cdots, \varepsilon_s)$ is the sum of the
    $i_1$-th, $\cdots, i_r$-th row vectors of the matrix
    $\widetilde{A}^{|j_1|\cdots |j_s|}_{|j_1|\cdots |j_s|}$.\nn
   \end{lem}
   \begin{proof}
     The lemma follows easily from the definition of
     $\{ \widetilde{\sigma}^A_j\}_{j\in [\pm n]}$ in~\eqref{Equ:FPS-A} and the property
     $\widetilde{\sigma}^A_{-j}= \widetilde{\sigma}^A_j$ for $\forall\, j \in [\pm n]$.
     So we leave it as an exercise to the reader.
   \end{proof}
     \nn

   \begin{lem} \label{Lem:Component-3}
     For any proper face $f$ of $\mathcal{C}^n$, we can write
     $f = \mathbf{F}(j_1,\cdots, j_s)$ where $1\leq |j_1| < \cdots < |j_s| \leq n$.
     Then in $\mathcal{F}_A$,  $|\widehat{f}| = 2^d$ where
     $d=\mathrm{rank}_{\Z_2}( \widetilde{A}^{|j_1|\cdots  |j_s|}_{|j_1|\cdots |j_s|})$.
   \end{lem}
   \begin{proof}
       By our discussion in section~\ref{Section4}, all the components in $\widehat{f}$
      are
       $$\{ \mathbf{F}(j_1,\cdots, j_s)^{\Sigma} \, ; \, \Sigma \subset \{ 1,\cdots, s \}
      \}.$$
     Notice that when $\Sigma$ runs over all the subsets of $\{ 1,\cdots, s \}$,
    the corresponding vector $(\varepsilon_1,\cdots, \varepsilon_s)$ in
    Lemma~\ref{Lem:Component-2} will take all the values in
   the image space of $\widetilde{A}^{|j_1|\cdots  |j_s|}_{|j_1|\cdots |j_s|}$, where
    we consider $\widetilde{A}^{|j_1|\cdots  |j_s|}_{|j_1|\cdots |j_s|}$ as a linear
   transformation on $(\Z_2)^s$.
   Then it is easy to see that the number of components in the face family
   $\widehat{f}$ equals the number of elements in the image space
   of $\widetilde{A}^{|j_1|\cdots  |j_s|}_{|j_1|\cdots |j_s|}$.
  So the lemma follows.
   \end{proof}
   \nn
   \begin{rem}
     For an arbitrary binary square matrix $A$ with zero diagonal,
     the facets-pairing structure
      $\mathcal{F}_A$ may not be strong (see Theorem~\ref{thm:Main-5}).
      But Lemma~\ref{Lem:Component-3} shows that the number of components in
      any face family in $\mathcal{F}_A$ is always a power of $2$. So this result
       does not follow from Theorem~\ref{thm:Component-1} in which the ``strongness''
       of the facets-pairing structure has to be assumed.
   \end{rem}
   \nn

   The following theorem can be thought of as a geometric
    interpretation of Bott matrices.\nn

  \begin{thm} \label{thm:Main-4}
    For an $n\times n$ binary matrix $A$ with zero diagonal ($n\geq 2$),
  the facets-pairing structure $\mathcal{F}_A$ on $\mathcal{C}^n$
  is perfect if and only if $A$ is a Bott matrix.
  \end{thm}
  \begin{proof}
     For any face
    $f = \mathbf{F}(j_1,\cdots, j_s)$ where $1\leq |j_1| < \cdots < |j_s| \leq n$,
    by Lemma~\ref{Lem:Component-3}, the face family $\widehat{f}$ is perfect if and only if
     $\widetilde{A}^{|j_1|\cdots  |j_s|}_{|j_1|\cdots |j_s|}$ is full rank,
     or the determinant of $\widetilde{A}^{|j_1|\cdots  |j_s|}_{|j_1|\cdots |j_s|}$ over $\Z_2$ is $1$.
   Therefore, $\mathcal{F}_A$
    is perfect if and only if all the principal
   minors of $\widetilde{A}$ are $1$. Then the Lemma~\ref{Lem:MasudaPanov} below tells us
   that this is exactly equivalent to $\widetilde{A}$ being conjugate to a
   unipotent upper triangular matrix by a permutation matrix, which is also equivalent to
    $A$ being a Bott matrix.
  \end{proof}
  \nn

   \begin{lem}[Masuda and Panov~\cite{MasPan08}] \label{Lem:MasudaPanov}
     Let $R$ be a commutative integral domain with an identity
   element $1$, and let $M$ be an $n\times n$ matrix with entries in
   $R$. Suppose that every proper principal minor of $M$ is $1$. If
   $\det M =1$, then $M$ is conjugate by a permutation matrix to a
    unipotent upper triangular matrix, and otherwise to a matrix of the form:
   \[    \begin{pmatrix}
        1 &  b_1 &  0 & \cdots & 0   \\
        0 &  1   &  b_2 & \cdots & 0 \\
       \vdots & \vdots & \ddots & \ddots & \vdots \\
        0 & 0 &  \cdots & 1  & b_{n-1} \\
        b_n & 0  & \cdots & 0 & 1
   \end{pmatrix}  \ \text{where $b_i \neq 0$ for all $i=1,\cdots, n$.} \]
   \end{lem}
   \n

  For a binary square matrix $A$ with zero diagonal,
  we may also ask when the
  facets-pairing structure $\mathcal{F}_A$ on $\mathcal{C}^n$ is strong.
  Note that it is
  not so easy to check the strongness of a facets-pairing structure
   directly from the definition.
  But for the $\mathcal{F}_A$ here, we have a very simple test stated in the
  following theorem. First, we introduce some notions to be used in our argument.
  \vskip .2cm

   \begin{defi}[Extended vector, Reduced vector]
       Suppose $M$ is an $m\times n$ matrix over a field and
    $M'$ is a submatrix of $M$ formed by a set of
    column vectors of $M$. For any $1\leq i \leq m$, let
    $\alpha_i$ and $\alpha'_i$ be the $i$-th row vector of
    $M$ and $M'$ respectively. Then we call $\alpha_i$ the \textit{extended
    vector} of $\alpha'_i$ in $M$, and call $\alpha'_i$ the \textit{reduced vector} of
    $\alpha_i$ in $M'$.
   \end{defi}

  \begin{thm} \label{thm:Main-5}
     For a binary square matrix $A$ with zero diagonal,
   the facets-pairing structure $\mathcal{F}_A$ on $\mathcal{C}^n$ is strong
  if and only if  for any $1\leq j_1 < \cdots < j_s \leq n$,
   $\mathrm{rank}_{\Z_2}(\widetilde{A}^{j_1\cdots  j_s}_{j_1\cdots j_s}) =
      \mathrm{rank}_{\Z_2}(\widetilde{A}^{j_1\cdots j_s})$.
  \end{thm}
  \begin{proof}
   For any face $\mathbf{F}(j_1,\cdots, j_s)$ of $\mathcal{C}^n$,
    let $(k_1,\cdots, k_m)$ and $(k'_1,\cdots, k'_r)$
   be the derived sequences of two sequences
   $(i_1,\cdots, i_m)$ and $(i'_1,\cdots, i'_{r})$ from
   $(j_1,\cdots, j_s)$ respectively in $\mathcal{F}_A$.
    By Lemma~\ref{Lem:Component-0},
   we can assume $(i_1,\cdots, i_m)$ and $(i'_1,\cdots, i'_{r})$ are
   both irreducible sequences. Then by our discussion in
   section~\ref{Section4}, $\mathcal{F}_A$ is strong if and only if
   whenever
   $\mathbf{F}_{k_1,\cdots, k_m}(j_1,\cdots, j_s)
   = \mathbf{F}_{k'_1,\cdots, k'_r}(j_1,\cdots, j_s)$, we must have
   \n
    \begin{enumerate}
      \item[(i)] $\mathbf{F}_{k_1,\cdots, k_m}(j_1,\cdots, j_s)
      \equiv \mathbf{F}_{k'_1,\cdots, k'_r}(j_1,\cdots, j_s)$, and
      \n

      \item[(ii)] $\widetilde{\sigma}^A_{k_m} \circ\cdots\circ\widetilde{\sigma}^A_{k_1} =
        \widetilde{\sigma}^A_{k'_r} \circ\cdots\circ
        \widetilde{\sigma}^A_{k'_1}$ as a permutation on $[\pm n]$.
    \end{enumerate}
 Without loss of generality, we can assume $1\leq |j_1| < \cdots < |j_s| \leq n$
 here.\n

  By~\eqref{Equ:Face-matrix-repre},
  $\mathbf{F}_{k_1,\cdots, k_m}(j_1,\cdots, j_s)= \mathbf{F}_{k'_1,\cdots, k'_r}(j_1,\cdots, j_s)$
   if and only if  the sum of the $i_1$-th, $\cdots, i_m$-th row vectors
  and the sum of the $i'_1$-th, $\cdots, i'_r$-th row vectors of
   $\widetilde{A}^{|j_1|\cdots  |j_s|}_{|j_1|\cdots |j_s|}$ coincide. So
  in this case,
  $\mathbf{F}_{k_1,\cdots, k_m}(j_1,\cdots, j_s)= \mathbf{F}_{k'_1,\cdots, k'_r}(j_1,\cdots, j_s)$
  will imply $\mathbf{F}_{k_1,\cdots, k_m}(j_1,\cdots, j_s)
      \equiv \mathbf{F}_{k'_1,\cdots, k'_r}(j_1,\cdots, j_s)$ automatically.
      But the (ii) says more. Indeed, (ii) requires that for
      $\forall\, l \neq \pm j_1,\cdots, \pm j_s$,
    \begin{equation} \label{Equ:Main-5_1}
      \widetilde{ \tau}^A_{k_m}\circ \cdots \circ \widetilde{\tau}^A_{k_1} ( \mathbf{F}(j_1,\cdots, j_s,l) )
        \equiv  \widetilde{\tau}^A_{k'_r}\circ\cdots\circ
        \widetilde{\tau}^A_{k'_1} (\mathbf{F}(j_1,\cdots, j_s,l)),
      \end{equation}
     Notice that $(k_1,\cdots, k_m)$ is the derived sequence of $(i_1,\cdots, i_m)$
     from $(j_1,\cdots, j_s,l)$ too. Similarly, $(k'_1,\cdots,
     k'_r)$ is the derived sequence of $(i'_1,\cdots, i'_r)$
     from $(j_1,\cdots, j_s,l)$.
     So~\eqref{Equ:Main-5_1} is equivalent to:
      \begin{equation} \label{Equ:Main-5_2}
       \mathbf{F}(j_1,\cdots, j_s,l)^{\Sigma}
       \equiv  \mathbf{F}(j_1,\cdots, j_s,l)^{\Sigma'},
        \ \forall\, l \neq \pm j_1,\cdots, \pm j_s,
      \end{equation}
        where $\Sigma = \{ i_1,\cdots, i_m \}$ and $\Sigma'= \{ i'_1,\cdots,
        i'_r\}$.
   By the formula in Lemma~\ref{Lem:Component-2}, the equation in~\eqref{Equ:Main-5_2}
   implies that
    the sum of $i_1$-th, $\cdots, i_m$-th row vectors
  and sum of $i'_1$-th, $\cdots, i'_r$-th row vectors of the matrix
   $\widetilde{A}^{|j_1|\cdots  |j_s|}$
  coincide (recall that
   $\widetilde{A}^{|j_1|\cdots |j_s|}$ is the submatrix of $\widetilde{A}$
   formed by the $|j_1|$-th, $\cdots, |j_s|$-th row vectors of
   $\widetilde{A}$). \n

    So the condition that $\mathcal{F}_A$ is strong is equivalent to the
    condition on $\widetilde{A}$ that: for any
    $1\leq j_1 < \cdots < j_s \leq n$,
    a set of row vectors of $\widetilde{A}^{j_1\cdots  j_s}_{j_1\cdots j_s}$
    are linearly dependent over $\Z_2$ implies that their extended row vectors in
    $\widetilde{A}^{j_1\cdots j_s}$ are also linearly
    dependent over $\Z_2$. By Lemma~\ref{Lem:Rank_SubMatrix} below, this condition is exactly
     equivalent to
     $\mathrm{rank}_{\Z_2}(\widetilde{A}^{j_1\cdots  j_s}_{j_1\cdots j_s}) =
      \mathrm{rank}_{\Z_2}(\widetilde{A}^{j_1\cdots j_s})$.
    \end{proof}
   \n

  \begin{lem} \label{Lem:Rank_SubMatrix}
    Suppose $M$ is an $m\times n$ matrix over a field $\F$ and
    $M'$ is a submatrix of $M$ formed by a set of
    column vectors of $M$. Then the following statements are equivalent.
    \begin{enumerate}
      \item[(i)] If a set of row vectors of $M'$ are linearly dependent over $\F$,
           then their extended row vectors in $M$ are also linearly
           dependent over $\F$. \n
      \item[(ii)] $\mathrm{rank}_{\F}(M) = \mathrm{rank}_{\F}(M')$;
    \end{enumerate}
  \end{lem}

\begin{proof}
    (i) $\Rightarrow$ (ii): The condition in (i) is equivalent to
    say that a set of row vectors of $M$ are linearly
    independent will force their reduced row vectors in $M'$
    to be also linearly independent. Then we must have $\mathrm{rank}_{\F}(M) \leq
     \mathrm{rank}_{\F}(M')$.
     But since $M'$ is a submatrix of $M$, $\mathrm{rank}_{\F}(M') \leq
     \mathrm{rank}_{\F}(M)$. So we get (ii). \nn

   (ii) $\Rightarrow$ (i): Assume $ \alpha_{i_1},\cdots, \alpha_{i_r}$
   are a set of linearly independent row vectors of $M$, but
   their reduced row vectors $\alpha'_{i_1},\cdots, \alpha'_{i_r}$ in $M'$
   are not linearly independent.
   We can add some extra row vectors $\alpha_{i_{r+1}},\cdots, \alpha_{i_k}$ of $M$
    to $\alpha_{i_1},\cdots, \alpha_{i_r}$ so that
    $\alpha_{i_1},\cdots, \alpha_{i_r}, \alpha_{i_{r+1}},\cdots, \alpha_{i_k}$
    form a set of maximally linearly independent row vectors of $M$.
    Note that $k=\mathrm{rank}_{\F}(M)$.
   But by our assumption, their
    reduced row vectors
    $\alpha'_{i_1},\cdots, \alpha'_{i_r}, \alpha'_{i_{r+1}},\cdots, \alpha'_{i_k}$
    in $M'$ are not linearly independent. Since
     $\mathrm{rank}_{\F}(M) = \mathrm{rank}_{\F}(M')$, there must exist another row
     vector $\alpha'_l$ of $M'$ which is linearly independent from
     $\alpha'_{i_1},\cdots, \alpha'_{i_r}, \alpha'_{i_{r+1}},\cdots,
     \alpha'_{i_k}$. Then $\alpha_l$ --- the extended row vector of $\alpha'_l$ in $M$
    --- is linearly independent from
     $\alpha_{i_1},\cdots, \alpha_{i_r}, \alpha_{i_{r+1}},\cdots,
     \alpha_{i_k}$. But this contradicts our assumption that
     $\alpha_{i_1},\cdots, \alpha_{i_r}, \alpha_{i_{r+1}},\cdots, \alpha_{i_k}$ are
     a set of maximally linearly independent row vectors of $M$.
  \end{proof}

    \n

       For a binary square matrix $A$ with zero diagonal,
    if the facets-pairing structure $\mathcal{F}_A$ on $\mathcal{C}^n$ is
    perfect, then Theorem~\ref{thm:Perfect-Strong} says that
     $\mathcal{F}_A$ must be strong. In fact, we can derive this result
    directly from Theorem~\ref{thm:Main-4} and Theorem~\ref{thm:Main-5}.
    If $\mathcal{F}_A$ is perfect, Theorem~\ref{thm:Main-4} says that
    for any $1\leq j_1 < \cdots < j_s \leq n$, the matrix
    $\widetilde{A}^{j_1\cdots  j_s}_{j_1\cdots j_s}$
     should be non-degenerate.
    This implies $\mathrm{rank}_{\Z_2}(\widetilde{A}^{j_1\cdots  j_s}_{j_1\cdots j_s})
    = \mathrm{rank}_{\Z_2}(\widetilde{A}^{j_1\cdots j_s})= s$.  Then by
   Theorem~\ref{thm:Main-5}, $\mathcal{F}_A$ is strong.
   \nn

  \begin{exam} \label{Exam:Strang-Strong}
    The facets-pairing structure $\mathcal{F}_A$ on
      $\mathcal{C}^3$ corresponding to the following binary matrix $A$ is
      strong. But $\mathcal{F}_A$ is not perfect since $\widetilde{A} = A + I_3$ is
      a degenerate matrix (over $\Z_2$). In fact, there are two $0$-dimensional face
      families in $\mathcal{F}_A$, each of which consists of four vertices of
      $\mathcal{C}^3$. Note that all the $1$-dimensional face families in $\mathcal{F}_A$ are
      perfect though.
       \[  A =   \begin{pmatrix}
            0  &  0 & 1   \\
            1   &  0  & 0  \\
            0   &  1  & 0
        \end{pmatrix},  \quad  \widetilde{A} =   \begin{pmatrix}
             1 &  0 & 1   \\
            1   &  1  & 0  \\
            0   &  1  & 1
        \end{pmatrix}.
        \]
  \end{exam}
    \nn

 \begin{exam}
      In Example~\ref{Exam:Strange}, the facets-pairing structure defined on
      $\mathcal{C}^3$ is equivalent to $\mathcal{F}_A$ for the following binary matrix $A$.
      \[  A =   \begin{pmatrix}
            0  &  0 & 0   \\
            0   &  0  & 1  \\
            1   &  1  & 0
        \end{pmatrix}, \quad  \widetilde{A} =   \begin{pmatrix}
            1  &  0 & 0   \\
            0   &  1  & 1  \\
            1   &  1  & 1
        \end{pmatrix}.
     \]
     Observe that the two submatrices of $\widetilde{A} = I_3 +A$ below have different rank.
     \[ \widetilde{A}^{23}_{23} =
     \begin{pmatrix}
              1  & 1  \\
               1  & 1
        \end{pmatrix},\quad
    \widetilde{A}^{2 3} =
       \begin{pmatrix}
            0   &  1  & 1  \\
            1   &  1  & 1
        \end{pmatrix}. \]
     So by Theorem~\ref{thm:Main-5}, we conclude that $\mathcal{F}_A$ is not
     strong, which agrees with our analysis in Example~\ref{Exam:Strange}.\\
  \end{exam}

   \section{Seal Space of $\mathcal{F}_A$ and glue-back
   Construction} \label{Section7}

       It is well-known that any $n$-dimensional real Bott manifold
        is a small cover over the $n$-dimensional cube. Recall that
        an $n$-dimensional small cover $M^n$ over a simple polytope $P^n$
        is a closed connected $n$-manifold
        with a locally standard $(\Z_2)^n$-action whose orbit space
        is $P^n$ (see~\cite{DaJan91}). The ``locally standard'' here means
        that locally the $(\Z_2)^n$-action is
          equivariantly homeomorphic to a faithful representation of
        $(\Z_2)^n$ on $\R^n$.
        The locally standard $(\Z_2)^n$-action on $M^n$ determines a $(\Z_2)^n$-valued function
       $\lambda$ on the facets of $P^n$, i.e.
       $\lambda : \mathcal{S}_F(P^n) \rightarrow (\Z_2)^n$. The
        $\lambda$
        encodes the information of the isotropy subgroups of
      the non-free orbits of the $(\Z_2)^n$-action.
      In addition, the $\lambda$ is \textit{non-degenerate} at
      each vertex $v$ of $P^n$, which means that:
    if $F_1,\cdots, F_n$ are all the facets
      of $P^n$ meeting at $v$, $\{ \lambda (F_1) \cdots, \lambda(F_n) \} $
      form a basis of $(\Z_2)^n$.  We call $\lambda$ the \textit{characteristic function} of
     $M^n$ on $P^n$. Conversely,  it was shown
     in~\cite{DaJan91} that $M^n$ is equivariantly homeomorphic to
      a space $M(P^n,\lambda)$ with a canonical $(\Z_2)^n$-action
    defined below.\nn

      For any nice manifold with corners $W^n$ and a $(\Z_2)^m$-valued function
      $\mu$ on all the facets of $W^n$, i.e. $\mu: S_F(W^n) \rightarrow (\Z_2)^m$
  ($m$ may be different from $n$), we can define a space $M(W^n,\mu)$ as following.
       For any proper face $f$ of $W^n$, let $G_f$ be the
     subgroup of $(\Z_2)^m$ generated by the following set
    $$\{ \mu(F) \, ; \, F \
   \text{is any facet of $W^n$ with}\ F\supseteq f \} . $$
    For any $p\in W^n$, let $f(p)$ be the unique face
    of $W^n$ that contains $p$ in its relative interior.
   Then we can glue $2^m$ copies of $W^n$ according to the
   information of $\mu$ by:
       \begin{equation} \label{Equ:Glue-back}
      M(W^n, \mu):= W^n \times (\Z_2)^m \slash \sim
      \end{equation}
    where $(p,g) \sim (p',g')$ if and only if $p=p'$ and
    $g-g' \in G_{f(p)}$ (see~\cite{DaJan91} and~\cite{Yu2010}). We call
     $M(W^n, \mu)$ the \textit{glue-back construction} from $(W^n,\mu)$.
     Moreover,
    there is a natural action of $(\Z_2)^m$ on $M(W^n, \mu)$
    defined by:
     \begin{equation} \label{Equ:Natural-Action}
       g\cdot [(p, g_0)] = [(p,  g_0 + g)],\ \forall\, p\in W^n,\ \forall\,
    g,g_0 \in (\Z_2)^m.
     \end{equation}

   In this paper, we will always assume $M(W^n, \mu)$ being equipped with this
   $(\Z_2)^m$-action. The reader is
       referred to~\cite{Yu2010} for more general form of glue-back construction.
   In addition, the function $\mu$ is called \textit{non-degenerate at a face $f$} if
    $\mu(F_{i_1}), \cdots, \mu(F_{i_k})$ are linearly independent
   over $\Z_2$ where $F_{i_1},\cdots, F_{i_k}$ are all the facets of $W^n$ which contain $f$.
    And $\mu$ is called \textit{non-degenerate on $W^n$} if
    $\mu$ is non-degenerate at all faces of $W^n$.
     In particular, when $W^n$ is a simple polytope,
    the non-degeneracy of $\mu$ on $W^n$ is equivalent to the non-degeneracy of $\mu$
  at all vertices of $W^n$.\nn

  Since we mainly deal with simple polytopes in this paper, we
  make the following special definition for functions on the facets of a simple
  polytope.\nn

     \begin{defi}[Generalized Characteristic Function on a
     Simple Polytope] \label{Def:Generalized-Char-Func}
      For an $n$-dimensional simple polytope $P^n$, any $(\Z_2)^n$-valued function
       $\widetilde{\lambda}$ (may not be non-degenerate) on the facets of $P^n$ is called a
       \textit{generalized characteristic function} on $P^n$.
      \nn

       For a generalized characteristic function $\widetilde{\lambda}$ on $P^n$,
       if $\widetilde{\lambda}$ is not non-degenerate at some vertices of $P^n$,
        the $M(P^n,\widetilde{\lambda})$ may not be a closed manifold
      (see Example~\ref{Exam:Strang-Strong-2} and
          Lemma~\ref{Lem:General-Char-Func}) and the natural $(\Z_2)^n$-action on
      $M(P^n,\widetilde{\lambda})$ may not be locally standard.
      So in general, we call $M(P^n, \widetilde{\lambda})$ a \textit{real toric orbifold}.
     \end{defi}
     \n

        For any $A \in \mathfrak{B}(n)$, since $M(A)$
         is a small cover over an $n$-dimensional cube, there is a locally standard
        $(\Z_2)^n$-action on $M(A)$ whose orbit space is a cube.
        Now, let us see what this
        locally standard $(\Z_2)^n$-action on $M(A)$ looks like when we identify
       $M(A)$ with the seal space $Q^n_{\mathcal{F}_A}$, where $\mathcal{F}_A$ is
       the facets-pairing structure on $\mathcal{C}^n$ associated to $A$. \nn

        For any $1\leq i \leq n$, let $h_i : \mathcal{C}^n \rightarrow \mathcal{C}^n$
       be the homeomorphism which sends any point $(x_1,\cdots, x_n) \in \mathcal{C}^n$ to
       $(x_1,\cdots, x_{i-1}, -x_i, x_{i+1}, \cdots, x_n)$. Let
         $$H=\langle h_1, \cdots, h_n \rangle \cong (\Z_2)^n .$$
       Then $H$ is a subgroup of the symmetry group of
       $\mathcal{C}^n$. For any $n\times n$ binary matrix $A$ with zero diagonal,
       the action of $H$ on $\mathcal{C}^n$ obviously commutes with
       the $\widetilde{\tau}^A_1, \cdots, \widetilde{\tau}^A_n$ defined by $\mathcal{F}_A$.
       So we get an action of $H$ on the seal space $Q^n_{\mathcal{F}_A}$.
       It is easy to see that $Q^n_{\mathcal{F}_A} \slash H$
       is an $n$-dimensional cube, denoted by $\mathcal{C}^n_0$. We can identify
       $\mathcal{C}^n_0$ with the following subset of $\mathcal{C}^n$.
       \begin{equation} \label{Equ:Small_Cube}
        \mathcal{C}^n_0 = \{ (x_1,\cdots, x_n) \in \mathbb{R}^n
        \, | \, 0 \leq x_i\leq \frac{1}{4},\ 1\leq \forall\, i \leq n \}
        \end{equation}

      For each $1\leq j \leq n$, let $\overline{F}_j$ be the facet of
      $\mathcal{C}^n_0$ which lies in the coordinate hyperplane $\{ x_j =0 \}$.
      And let $\overline{F}^*_j$ be the opposite facet of $\overline{F}_j$ in $\mathcal{C}^n_0$.
      Recall that we use the bold symbol $\mathbf{F}(j)$ to denote the facets of $\mathcal{C}^n$,
      but $\overline{F}_j$ and $\overline{F}^*_j$ are not in bold form. \nn

       In addition, let $\{ e_1,\cdots, e_n \} $ be a linear basis of
      $(\Z_2)^n$. We define a $(\Z_2)^n$-valued function $\lambda_A$ on the set of
     facets of $\mathcal{C}^n_0$ by:
      \begin{align}
       \lambda_A(\overline{F}_j) & = e_j, \ 1\leq \forall\, j \leq n  \label{Equ:Char-Function-1}\\
       \lambda_A(\overline{F}^*_j) &= \sum^n_{j=1}\widetilde{A}^j_k \cdot e_k, \ 1\leq \forall\, j \leq n
       \label{Equ:Char-Function-2}
       \end{align}
        where $\widetilde{A} = A + I_n$.  Note that the $ \lambda_A(\overline{F}^*_j)$ can be identified with
         the $j$-th row vector of $\widetilde{A}$.
         In general, $\lambda_A$ might not be non-degenerate at all
      vertices of $\mathcal{C}^n_0$. So by our terms, $\lambda_A$ is only
    a generalized characteristic function on $\mathcal{C}^n_0$
    (see Remark~\ref{Def:Generalized-Char-Func}). \nn

     \begin{lem} \label{Lem:Equiv-Equivalence}
       For any $n\times n$ binary matrix $A$ with zero diagonal
     entries, the seal space
        $Q^n_{\mathcal{F}_A}$ is homeomorphic to $M(\mathcal{C}^n_0, \lambda_A)$,
        and the action of $H$ on
        $Q^n_{\mathcal{F}_A}$ can be identified with the natural $(\Z_2)^n$-action on
        $M(\mathcal{C}^n_0, \lambda_A)$.
     \end{lem}
     \begin{proof}
     Notice that $\mathcal{C}^n$ is divided into $2^n$ small cubes
     of the same size by the $n$ coordinate hyperplanes of $\R^n$.
     The $\mathcal{C}^n_0$ defined above is just one of them. On the other hand,
     in the definition of $M(\mathcal{C}^n_0, \lambda_A)$, if we
    only glue the facets $\overline{F}_1,\cdots, \overline{F}_n$ in each $\mathcal{C}^n_0 \times \{ g\} ,g \in (\Z_2)^n$
    first according to the rule in~\eqref{Equ:Glue-back}, we will get a
    big cube which can be identified with the $\mathcal{C}^n$.
   Then we can think of the boundary of $\mathcal{C}^n$ being tessellated by those facets
   $\{ \overline{F}^*_i \}$ of the $2^n$ copies of $\mathcal{C}^n_0$ which have not been glued.
   In fact, for each $1\leq i \leq n$, the facet $\mathbf{F}(i)$ of $\mathcal{C}^n$ is tessellated by
    the $\overline{F}^*_i$ in all copies of $\mathcal{C}^n_0$ in
     $\mathcal{C}^n_0\times G_i$, where $G_i$ is the subgroup of $(\Z_2)^n$
    generated by $\{ e_1,\cdots, \widehat{e}_i, \cdots, e_n \}$,
    and $\mathbf{F}(-i)$ of $\mathcal{C}^n$ is tessellated by
    the $\overline{F}^*_i$ in all copies of $\mathcal{C}^n_0$ in
     $\mathcal{C}^n_0\times (e_i+G_i)$ (see Figure~\ref{p:Partial-Glue} for $n=2$ case).\n

       \begin{figure}
         \includegraphics[width=0.52\textwidth]{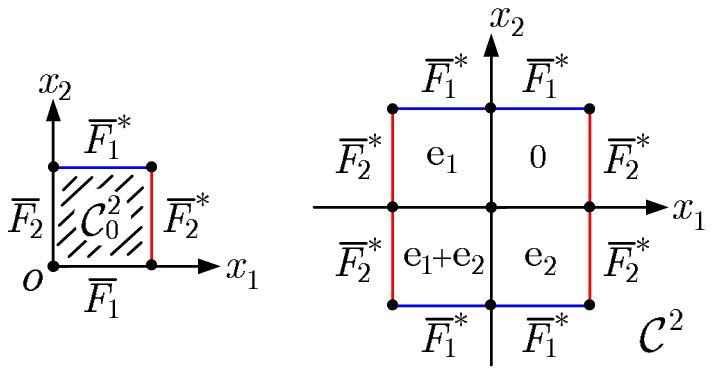}
          \caption{  }\label{p:Partial-Glue}
      \end{figure}

   To further obtain $M(\mathcal{C}^n_0, \lambda_A)$,
  we should glue any $\overline{F}^*_i \times \{ g\} \subset
  \mathbf{F}(i)$, $g\in G_i$
  to $ \overline{F}^*_i \times \{ g + \lambda_A(\overline{F}^*_i)\} \subset \mathbf{F}(-i)$
    by the map $(x_1,\cdots,x_n) \longrightarrow
      ((-1)^{\widetilde{A}^i_1}x_1, \cdots, (-1)^{\widetilde{A}^i_n}x_n )$

   \noindent   $ =
      ((-1)^{A^i_1}x_1, \cdots, (-1)^{A^i_{i-1}}x_{i-1}, -x_i,
     (-1)^{A^i_{i+1}}x_{i+1}, \cdots, (-1)^{A^i_n}x_n )$, $1\leq i \leq n$ (see
     the definition of $\lambda_A$).
   Observe that this exactly agrees with the structure map
    $\tau^A_i: \mathbf{F}(i) \rightarrow \mathbf{F}(-i)$ of $\mathcal{F}_A$
    (see~\eqref{Equ:tau_j-Cube}). So
      $M(\mathcal{C}^n_0, \lambda_A)$ and $Q^n_{\mathcal{F}_A}$ are
      the quotient space of $\mathcal{C}^n$ by the same
      quotient map, hence they are homeomorphic.
    And clearly the action of $H$ on
     $Q^n_{\mathcal{F}_A}$ can be identified with the natural $(\Z_2)^n$-action
     on $M(\mathcal{C}^n_0, \lambda_A)$ defined in~\eqref{Equ:Natural-Action}.
  \end{proof}
   \n

      When $A\in \mathfrak{B}(n)$, we can show that $\lambda_A$ is non-degenerate
      on $\mathcal{C}^n_0$, hence $M(\mathcal{C}^n_0,\lambda_A)$ is a small cover.
      Indeed,
      let $u_j$ be the vertex of $\mathcal{C}^n_0$ on
      the $x_j$-axis, $1\leq j \leq n$. For any subset
      $\{ j_1,\cdots, j_s \} \subset \{ 1,\cdots, n \}$,
      let $u_{j_1\cdots j_s}$ be the vertex of
      $\mathcal{C}^n_0$ whose projection to the $x_{j_i}$-axis is
      $u_{j_i}$ for each $1\leq i \leq s$. So the facets of $\mathcal{C}^n_0$
      which contain $u_{j_1\cdots j_s}$ are
       $$\{ \overline{F}^*_{j_1}, \cdots,
      \overline{F}^*_{j_s}, \overline{F}_{l_1} , \cdots, \overline{F}_{l_{n-s}}\}, $$
   where $\{ l_1,\cdots, l_{n-s}\}= \{ 1,\cdots, n \} \backslash \{ j_1,\cdots, j_s \}$.
     From the definition of $\lambda_A$,
     it is easy to see that the non-degeneracy of $\lambda_A$
       at a vertex $u_{j_1\cdots j_s}$
      corresponds exactly to the non-degeneracy of the
    matrix $\widetilde{A}^{j_1\cdots j_s}_{j_1\cdots j_s}$ (see~\eqref{Equ:Auxil-SubMatrix}).
      Since $A$ is a Bott matrix, Lemma~\ref{Lem:MasudaPanov} implies that
      any principal minor of $\widetilde{A}$ is $1$, so
      $\widetilde{A}^{j_1\cdots j_s}_{j_1\cdots j_s}$ is
      non-degenerate.
      Therefore when $A\in \mathfrak{B}(n)$,
       the natural $(\Z_2)^n$-action on  $M(\mathcal{C}^n_0,
      \lambda_A)$ is locally standard, so is the action of $H$ on
      $Q^n_{\mathcal{F}_A}$.
       \nn

     \begin{rem}
       In the section 4 of ~\cite{SuyMasDong10}, another form of
        locally standard $(\Z_2)^n$-action on a real Bott manifold is
       given, which is equivalent to the one constructed above.
        But we want to warn the reader that the $M(A)$ defined in this paper
       actually corresponds to
       the space $M(A+I_n)$ defined in~\cite{SuyMasDong10}.
      In addition, the formulae in the section 4 of~\cite{SuyMasDong10} are written
        for quasitoric manifolds, but the small cover case is
        parallel.  \\
      \end{rem}

 \section{Singularities in Glue-back Construction} \label{Section8}

    Suppose $A$ is an arbitrary $n\times n$ binary matrix with zero diagonal
    entries. By Lemma~\ref{Lem:Equiv-Equivalence}, we have an equivariant
     homeomorphism from the seal space $Q^n_{\mathcal{F}_A}$ of $\mathcal{F}_A$ to
     $M(\mathcal{C}^n_0, \lambda_A)$.
    So to understand the singularities that might occur in
    $Q^n_{\mathcal{F}_A}$, it amounts to understand the
    singularities in the glue-back construction.
     Notice that when $\lambda_A$ is not non-degenerate at all
      vertices of $\mathcal{C}^n_0$, the
       natural $(\Z_2)^n$-action on $M(\mathcal{C}^n_0, \lambda_A)$ may not be
        locally standard at some places in $M(\mathcal{C}^n_0,\lambda_A)$
    and, $M(\mathcal{C}^n_0,\lambda_A)$ may not even be a manifold.
    Let us see such an example first.\nn

  \begin{exam} \label{Exam:Strang-Strong-2}
    For the matrix $A$ defined in Example~\ref{Exam:Strang-Strong},
    the function $\lambda_A$ on the facets of $\mathcal{C}^3_0$ is:
     $$\lambda_A (\overline{F}^*_1) = e_1+ e_3,\ \lambda_A (\overline{F}^*_2) = e_1+
      e_2, \ \ \lambda_A (\overline{F}^*_3) = e_2+ e_3. $$
    So the $\lambda_A$ is degenerate at the vertex
   $u_{123}$ of $\mathcal{C}^3_0$ (see Figure~\ref{p:Strang-Strong}).
    And the natural $(\Z_2)^3$-action on
   $M(\mathcal{C}^3_0, \lambda_A)$ is not locally
   standard in a neighborhood of $u_{123}$.
   In fact,
   the neighborhood of $u_{123}$ in $M(\mathcal{C}^3_0, \lambda_A)$
   is homeomorphic to a cone of $\R P^2$. This is because for any section
   $\bigtriangledown$ of the cube
    near $u_{123}$, $\lambda_A$ induces a characteristic function $\lambda_A^{\bigtriangledown}$
    on the three edges of $\bigtriangledown$ (see the right picture of Figure~\ref{p:Strang-Strong}).
    Obviously, $M(\bigtriangledown, \lambda^{\bigtriangledown}_A) \cong \R P^2$.
  So $M(\mathcal{C}^3_0, \lambda_A)$
    is not a manifold. This example tells us that the seal space of
   a (non-trivial) strong regular facets-pairing structure on a cube might
   not be a manifold.

    \begin{figure}
         \includegraphics[width=0.56\textwidth]{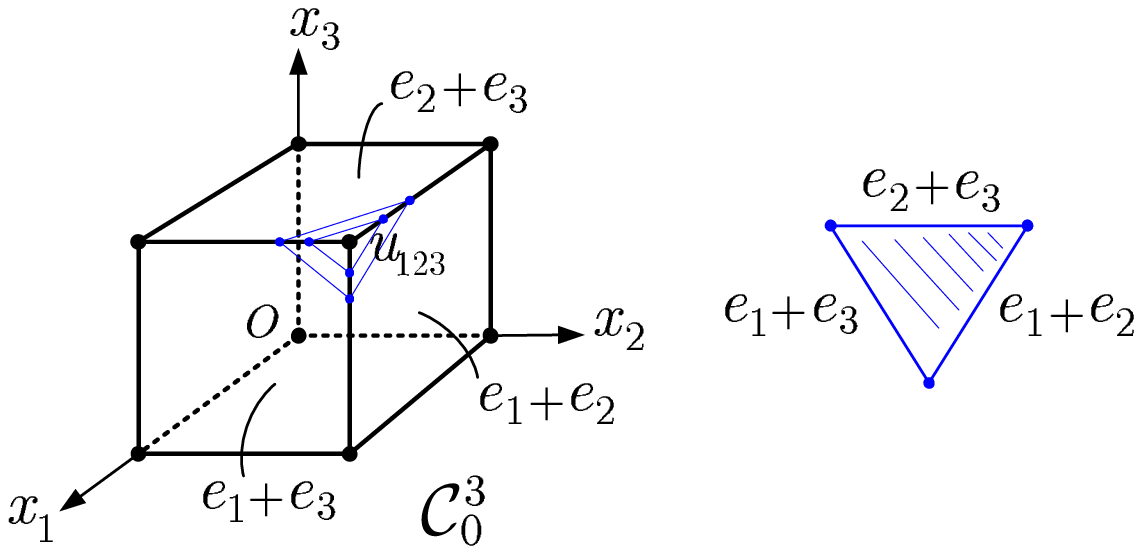}
          \caption{  }\label{p:Strang-Strong}
      \end{figure}
   \end{exam}

    \n

    The cause of the singularity in the above example can be
     formulated into a general condition on a generalized characteristic
     function $\widetilde{\lambda}$ on a simple polytope $P^n$
     which would make $M(P^n, \widetilde{\lambda})$ have some singular point.\nn

    \begin{lem} \label{Lem:General-Char-Func}
         Suppose $P^n$ is an $n$-dimensional simple polytope and
  $\widetilde{\lambda}$ is a generalized characteristic function
  on $P^n$.
     If there exists a vertex $v_0$ of $P^n$
       and a set of facets $F_{i_1},\cdots, F_{i_r}$ ($3\leq r \leq n$)
        meeting $v_0$
       such that: $\lambda(F_{i_1}),  \cdots, \lambda(F_{i_{r-1}}) \in (\Z_2)^n$
       are linearly independent over $\Z_2$ and
        $\lambda(F_{i_r}) = \lambda(F_{i_1}) + \cdots + \lambda(F_{i_{r-1}})$,
      then the space $M(P^n, \widetilde{\lambda})$ defined by~\eqref{Equ:Glue-back}
      is not a manifold.
    \end{lem}
  \begin{proof}
     The type of singularity in $M(P^n, \widetilde{\lambda})$
      is similar to that in the Example~\ref{Exam:Strang-Strong-2}.
      Indeed, for any point $q$ in the relative interior of
      the face $F_{j_1} \cap \cdots \cap F_{j_r}$,
        an open neighborhood of $q$ in $M(P^n, \widetilde{\lambda})$
        is homeomorphic to $(-\varepsilon, \varepsilon)^{n-r} \times \mathrm{Cone}(\R P^{r-1})$.
          Since $r\geq 3$, the cone
         on $\R P^{r-1}$ is not homeomorphic to a ball. So the space
        $M(P^n, \widetilde{\lambda})$ is not a manifold at $q$.
     \end{proof}
     \n

    An equivalent way to state Lemma~\ref{Lem:General-Char-Func} is the
   following.\nn

  \begin{lem} \label{Lem:General-Char-Func-2}
     If $M(P^n, \widetilde{\lambda})$ is a  manifold,
     it is necessary that: at any vertex $v = F_1 \cap \cdots \cap F_n$ of $P^n$,
    if
     $\widetilde{\lambda}(F_{i_1}),\cdots, \widetilde{\lambda}(F_{i_s})$
     are maximally linearly independent among $\widetilde{\lambda}(F_1),\cdots,
     \widetilde{\lambda}(F_n)$, then each
     $\widetilde{\lambda}(F_i)$ ($1\leq i \leq n$) must
     coincide with one of the $\widetilde{\lambda}(F_{i_1}),\cdots, \widetilde{\lambda}(F_{i_s})$.
 \end{lem}
   \n

  From the above lemma, we can easily derive the following.\nn

  \begin{cor} \label{Cor:Mfd-General-Char-Func}
  Suppose $P^n$ is an $n$-dimensional simple polytope and
  $\widetilde{\lambda}$ is a generalized characteristic function
  on $P^n$ so that $M(P^n, \widetilde{\lambda})$ is a
  manifold. Then at any vertex $v=F_1\cap \cdots \cap F_n$ of $P^n$, if
  $\widetilde{\lambda}(F_1), \cdots, \widetilde{\lambda}(F_n) \in (\Z_2)^n$ are pairwise
  distinct, then $\widetilde{\lambda}$ must be non-degenerate at $v$.
   \end{cor}

   \n

          However,
      it is possible that a
      generalized characteristic function $\widetilde{\lambda}$
     on $P^n$, even not non-degenerate at some vertices,
      can still make $M(P^n, \widetilde{\lambda})$ a closed manifold.
      Let us see such an example below.
      \nn

  \begin{exam} \label{Exam:PR2}
   For the binary matrix $A = \begin{pmatrix}
       0  & 1   \\
        1 &  0  \\
   \end{pmatrix}$,  the generalized
   characteristic function $\lambda_A$ on $\mathcal{C}^2_0$ defined by ~\eqref{Equ:Char-Function-1}
    and~\eqref{Equ:Char-Function-2} is:
    \[ \lambda_A(\overline{F}_1) = e_1, \ \lambda_A(\overline{F}_2) = e_2 ; \
      \lambda_A(\overline{F}^*_1) = \lambda_A (\overline{F}^*_2) = e_1 + e_2.
    \]
     So $\lambda_A$ is not non-degenerate at the vertex
     $u_{12} = \overline{F}^*_1 \cap \overline{F}^*_2$. But it is easy to check that
    $M(\mathcal{C}^2_0, \lambda_A)$ is homeomorphic to $\R P^2$ and
     the natural $(\Z_2)^2$-action on $M(\mathcal{C}^2_0, \lambda_A)$
     defined by~\eqref{Equ:Natural-Action} is locally standard.
   Indeed, by Lemma~\ref{Lem:Equiv-Equivalence},  $M(\mathcal{C}^2_0, \lambda_A)$ is homeomorphic to
   the seal space of the facets-pairing structure $\mathcal{F}_A$ on $\mathcal{C}^2$
   (see the picture in Figure~\ref{p:RP2}). Obviously, the seal space of $\mathcal{F}_A$ is $\R P^2$.
     \n

   Another way to understand this example is:
    since $\lambda_A(\overline{F}^*_1) = \lambda_A (\overline{F}^*_2)$,
     we let the edge $\overline{F}^*_1$ merge with
      $\overline{F}^*_2$ to form a long edge. And then we get a non-degenerated
    characteristic function $\lambda^{red}_A$ on a $2$-simple $\Delta^2$ (see the right picture
    in Figure~\ref{p:RP2}). The corresponding small
    cover $M(\Delta^2, \lambda^{red}_A)$ is homeomorphic to $\R P^2$.
    Moreover, we have an equivariant homeomorphism from $M(\mathcal{C}^2_0, \lambda_A)$ to
    $M(\Delta^2,\lambda^{red}_A)$
    which is induced by the merging of $\overline{F}^*_1$
    with $\overline{F}^*_2$ on $\mathcal{C}^2_0$.
  \begin{figure}
         \includegraphics[width=0.75\textwidth]{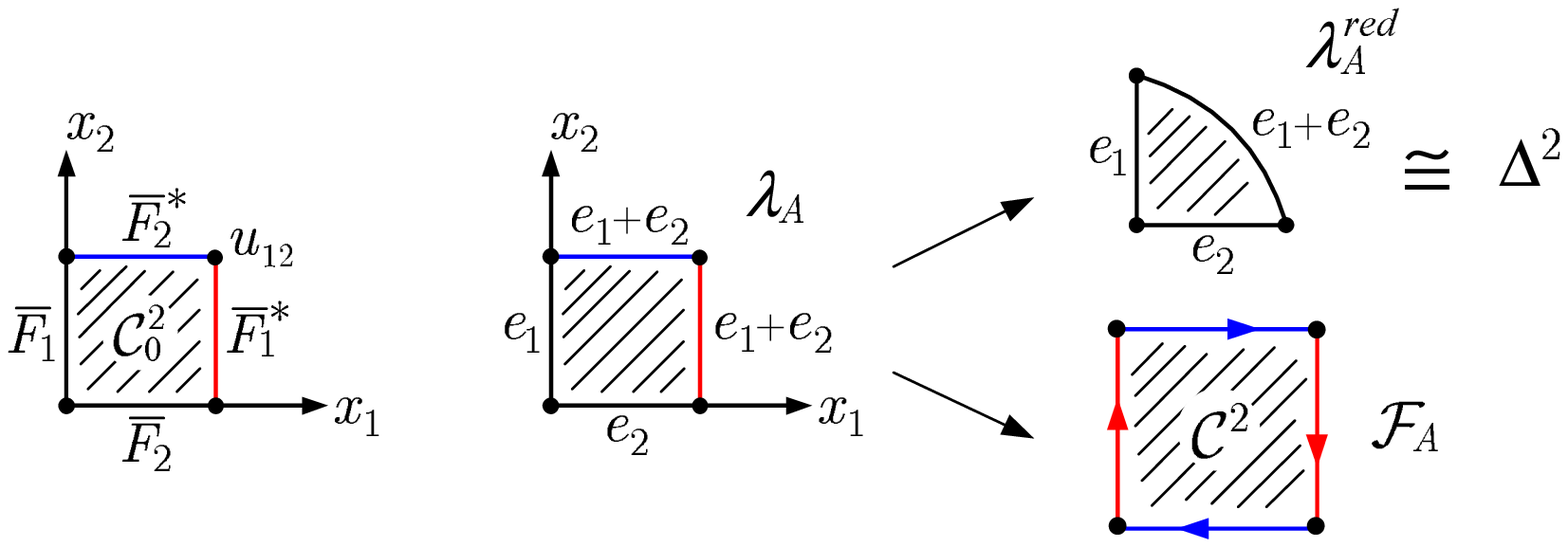}
          \caption{  }\label{p:RP2}
      \end{figure}

   \end{exam}
     \n

   The idea of merging two neighboring edges into one edge in Example~\ref{Exam:PR2} can
   be generalized to the following setting.\nn

  \begin{defi}[Smoothing a nice manifold with corners along codimension-two faces]
  \label{Defi:Smoothing}
    Suppose $W^n$ is a nice manifold with corners and $\mathbf{f}=\{ f_1,\cdots, f_k \}$ is a
    set of codimension-two faces of
     $W^n$. When we say \textit{smoothing $W^n$ along $\mathbf{f}$}, we mean that
    we forget $f_1,\cdots, f_k$ as well as all their faces
    from the manifold with corners structure of $W^n$. The stratified space we get
   is denoted by $W^n[\mathbf{f}]$. In other words, we think of
   $f_1,\cdots, f_k$ as well as all their faces as empty
   faces in $W^n[\mathbf{f}]$.\n

      Geometrically, we can think of the smoothing of $W^n$ along
      $\mathbf{f}=\{ f_1,\cdots, f_k \}$
       as a local deformation
      of $W^n$ around $f_1,\cdots, f_k$ to make $W^n$ ``smooth'' at those
      places, then removing $f_1,\cdots, f_k$ as well as all their faces from
      the stratification of $\partial W^n$. This process is similar
      to the \textit{straightening of angles} introduced
      in the first chapter of~\cite{Conner79}.\nn

    In Figure~\ref{p:Smoothing}, we can see the local picture of smoothing
    an $n$-dimensional nice manifold with corners ($n=3$ and $4$) along a codimension-two face.
    \end{defi}

    \begin{rem}
     Generally speaking, $W^n[\mathbf{f}]$ may not be a nice manifold
    with corners any more, though $W^n$ is.
    \end{rem}

   Suppose $f_i = F^{(1)}_{i} \cap F^{(2)}_{i}$ where
   $ F^{(1)}_{i}, F^{(2)}_{i}$ are facets of $W^n$. Then
   $F^{(1)}_{i}$ and $F^{(2)}_{i}$ will merge into a big facet or
  part of a big facet in $W^n[\mathbf{f}]$. More generally, two facets
   $F, F'$ of $W^n$ will become part of a big facet in $W^n[\mathbf{f}]$
  if and only if there exists a sequence $F=F_1, F_2,\cdots, F_r= F'$
   so that for each $1\leq j \leq r-1$,
     $F_j \cap F_{j+1} \in \mathbf{f}$.
   Let $\mathcal{S}_F(W^n)$ and $\mathcal{S}_F(W^n[\mathbf{f}])$ denote the set of facets of
  $W^n$ and $W^n[\mathbf{f}]$ respectively, then we have a natural map
    $$\psi_{[\mathbf{f}]}: \ \mathcal{S}_F(W^n)\ \longrightarrow \ \mathcal{S}_F(W^n[\mathbf{f}])
      $$
   where $\psi_{[\mathbf{f}]}$ sends any
   facet $F$ of $W^n$ to the facet of $W^n[\mathbf{f}]$ which contains
   $F$ as a set. Obviously, $\psi_{[\mathbf{f}]}$ is surjective.\nn

   \begin{figure}
         \includegraphics[width=0.9\textwidth]{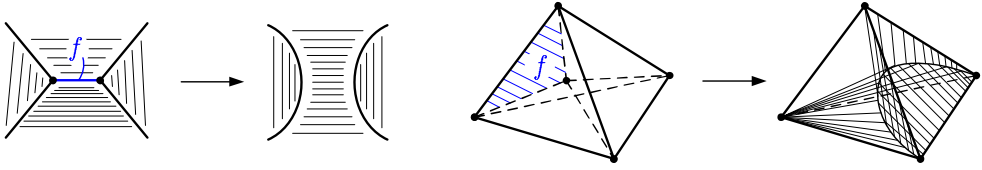}
          \caption{ Smoothing an $n$-dimensional nice manifold with corners
        along a codimension-two face $f$ ($n=3$ and $4$) }\label{p:Smoothing}
      \end{figure}

  Suppose $\mu$ is a $(\Z_2)^m$-valued function on the set of facets of $W^n$
  which satisfies:
  \[ \text{$\mu(F) = \mu(F')$
    whenever $\psi_{[\mathbf{f}]}(F)=\psi_{[\mathbf{f}]}(F')$ for any facets $F,F'$ of
   $W^n$,}
   \]
  we say that $\mu$ is \textit{compatible with} $\psi_{[\mathbf{f}]}$. In this case,
   $\mu$ induces a $(\Z_2)^m$-valued function $\mu[\mathbf{f}]$
   on the set of facets of $W^n[\mathbf{f}]$ by:
    \begin{equation} \label{Equ:Induced-Function}
      \mu[\mathbf{f}] (\psi_{[\mathbf{f}]}(F))
   := \mu(F)\ \ \text{for any facet $F$ of $W^n$}.
   \end{equation}
 We call $\mu[\mathbf{f}]$ the \textit{induced function from} $\mu$ with respect to
  the smoothing. If we assume
   $W^n[\mathbf{f}]$ is still a nice manifold with corners, then
  the glue-back construction $M(W^n[\mathbf{f}], \mu[\mathbf{f}])$ can be defined.
  It is easy to see that the natural $(\Z_2)^m$-action on $M(W^n, \mu)$ and
  $M(W^n[\mathbf{f}], \mu[\mathbf{f}])$ can be identified through the smoothing of $W^n$.
  So we have the following.
 \nn

   \begin{lem} \label{Lem:Induced-Char-Func}
  Suppose $\mu$ is a $(\Z_2)^m$-valued function on the set of facets of $W^n$
   which is compatible with $\psi_{[\mathbf{f}]}$. If
   $W^n[\mathbf{f}]$ is still a nice manifold with corners, then
  there is an equivariant homeomorphism from $M(W^n, \mu)$
   to $M(W^n[\mathbf{f}], \mu[\mathbf{f}])$.
  \end{lem}
  \n

    Next, let us discuss a special kind of smoothings of an $n$-dimensional cube. Suppose
    $\{ I_1,\cdots, I_m \}$ is a \textit{partition} of the set $[n] := \{ 1,\cdots, n \}$, i.e.
    $I_1,\cdots, I_m$ are pairwise disjoint nonempty subsets of $[n]$ with
     $I_1 \cup \cdots \cup I_m = [n]$. Let $\mathbf{f}_{I_1\cdots I_m}$ be
    a set of codimension-two faces of $\mathcal{C}^n_0$ defined by:
     \begin{equation} \label{Equ:f-I}
        \mathbf{f}_{I_1\cdots I_m} :=
        \{ \overline{F}^*_l \cap \overline{F}^*_{l'}  \, ; \,
        l \ \text{and}\ l'\ \text{belong to the same}\ I_j \ \text{for some}
       \ 1\leq j \leq m  \}
     \end{equation}
    Notice that if $I_i$ has only one element, it
   has no contribution to $\mathbf{f}_{I_1\cdots I_m}$.
    Let $\mathcal{C}^n_{I_1\cdots I_m} := \mathcal{C}^n_0[\mathbf{f}_{I_1\cdots I_m}]$ be
    the smoothing of $\mathcal{C}^n_0$ along $\mathbf{f}_{I_1\cdots I_m}$.
    So we have a map
     $$\psi_{[\mathbf{f}_{I_1\cdots I_m}]} : \mathcal{S}_F(\mathcal{C}^n_0)
      \rightarrow \mathcal{S}_F(\mathcal{C}^n_{I_1\cdots I_m}).$$

   It is easy to see that for any $1\leq j \leq n$,
    $\overline{F}_j$ does not merge
    with any other facets in $\mathcal{C}^n_0$, while
    all the facets in $\{ \overline{F}^*_l \, ; \, l \in I_i \}$
    will merge into a big facet in $\mathcal{C}^n_{I_1\cdots I_m}$.
    We denote all the facets of $\mathcal{C}^n_{I_1\cdots I_m}$ by
     $\{ \widetilde{F}_1,\cdots, \widetilde{F}_n, \widetilde{F}^*_{I_1},
     \cdots, \widetilde{F}^*_{I_m} \}$
     where:\n

\begin{itemize}
  \item $\widetilde{F}_j = \psi_{[\mathbf{f}_{I_1\cdots I_m}]}(\overline{F}_j)$,
    $1\leq j \leq n$.
   \n

  \item $\widetilde{F}^*_{I_i} = \psi_{[\mathbf{f}_{I_1\cdots I_m}]} (\overline{F}^*_l)$ for any
     $l \in I_i$, $1\leq i \leq m$. In other words,
     $\widetilde{F}^*_{I_i}$  is the merging
     of all the facets $\{ \overline{F}^*_l \, ; \, l \in I_i \}$.
\end{itemize}

   It is easy to see that any face of
   $\mathcal{C}^n_{I_1\cdots I_m}$ is homeomorphic to a
   ball. So  $\mathcal{C}^n_{I_1\cdots I_m}$ is a nice manifold with corners
   with all faces contractible. \nn

     \begin{exam}
         In Figure~\ref{p:Smoothing-Faces}, we have two
         smoothings of $\mathcal{C}^3_0$.
          By our notation,
         the upper one is $\mathcal{C}^3_{\{ 1 \}\{2,3 \}} \cong \Delta^1 \times
         \Delta^2$, and the lower one is $\mathcal{C}^3_{\{1, 2,3 \}} \cong
         \Delta^3$ where $\Delta^i$ denotes the standard
         $i$-dimensional simplex in $\R^{i}$.

          \begin{figure}
         \includegraphics[width=0.67\textwidth]{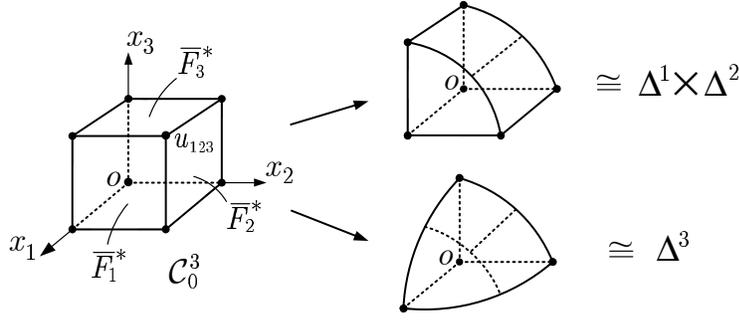}
          \caption{Two different smoothings of a cube}\label{p:Smoothing-Faces}
      \end{figure}
      \end{exam}
        \n

    \begin{thm} \label{thm:Smoothing-Cube}
       For any partition $\{ I_1,\cdots, I_m \}$
    of the set $[n] := \{ 1,\cdots, n \}$, the
    $\mathcal{C}^n_{I_1\cdots I_m}$ is
    homeomorphic to $\Delta^{n_1}\times\cdots \times \Delta^{n_m}$
    as a manifold with corners, where $n_i = |I_i|$, $1\leq i \leq m$
    and $n_1+\cdots + n_m =n$.
    \end{thm}
   \begin{proof}
      We will borrow some notations in~\cite{SuyMasDong10}. Let $\{ v^i_0,\cdots, v^i_{n_i} \}$ be
     the set of all vertices of $\Delta^{n_i}$.
     Then each vertex of
     $\Delta^{n_1}\times\cdots \times \Delta^{n_m}$ can be uniquely
     written as a product of vertices from
     $\Delta^{n_i}$'s. Hence all the vertices
     of $\Delta^{n_1}\times\cdots \times \Delta^{n_m}$ are:
      \[  \{   \widetilde{v}_{j_1\ldots j_m}= v^1_{j_1}\times \cdots \times v^m_{j_m} \,
      | \ 0\leq j_i \leq n_i,\ i=1,\cdots, m \}. \]
    Each facet of $\Delta^{n_1}\times\cdots \times \Delta^{n_m}$ is
    the product of a codimension-one
    face of some $\Delta^{n_i}$ and the remaining simplices. So all
    the facets of $\Delta^{n_1}\times\cdots \times \Delta^{n_m}$ are:
      \[ \{ F^i_{k_i} \, | \ 0 \leq k_i \leq n_i,
      \ i=1,\cdots, m  \},  \]
   where $F^i_{k_i} = \Delta^{n_1}\times \cdots \times \Delta^{n_{i-1}} \times f^{i}_{k_i}
   \times \Delta^{n_{i+1}} \times \cdots \times \Delta^{n_m}$, and
   $f^{i}_{k_i}$ is the codimension-one face of the simplex $\Delta^{n_i}$
   which is opposite to the vertex $v^i_{k_i}$. So there are
   total of $m+n$ facets in $\Delta^{n_1}\times\cdots \times \Delta^{n_m}$.
   In addition, the $n$ facets which meet at the vertex
   $\widetilde{v}_{j_1\ldots j_m}$ are:
   $$\mathcal{S}_{F}(\Delta^{n_1}\times\cdots \times \Delta^{n_m})
    - \{ F^i_{j_i}\, | \ i=1,\cdots, m \}.$$

    In particular, the $n$ facets that meet at the vertex $\widetilde{v}_{0\ldots 0}$ are:
    $$\mathcal{S}_{F}(\Delta^{n_1}\times\cdots \times \Delta^{n_m}) - \{ F^i_0 \, | \ i=1,\cdots, m
   \} = \{ F^1_1, \cdots, F^1_{n_1},\cdots, F^m_1, \cdots, F^m_{n_m} \}. $$

   Next, we define a map $\Theta$ from the
   set of all facets of $\mathcal{C}^n_{I_1\cdots I_m}$ to the set of all
    facets of $\Delta^{n_1}\times\cdots \times \Delta^{n_m}$.
   Without loss of generality, we can assume that:
      \begin{align} \label{Equ:Partition-n}
        I_1 & = \{ 1, \cdots, n_1 \}, \ \ I_2 = \{ n_1 +1,\cdots, n_1 + n_2
           \}, \ \cdots    \notag \\
      & \cdots, I_m = \{ n_1+\cdots + n_{m-1}+1,\cdots, n \}.
      \end{align}
   First, we define $\Theta$ to map the facets of $\mathcal{C}^n_{I_1\cdots I_m}$
   meeting at the origin
   to the facets of $\Delta^{n_1}\times\cdots \times \Delta^{n_m}$
    meeting at $\widetilde{v}_{0\ldots 0}$ by:
   \begin{align*}
      \Theta (\widetilde{F}_1 ) &= F^1_1 ,\ \cdots\ , \
     \Theta(\widetilde{F}_{n_1})= F^1_{n_1} \\
          \Theta (\widetilde{F}_{n_1+1} ) & = F^2_1 , \ \cdots\ , \
     \Theta(\widetilde{F}_{n_1+n_2})= F^2_{n_2} \\
              & \cdots \qquad \cdots \qquad \cdots \\
         \Theta (\widetilde{F}_{n_1+\cdots + n_{m-1} +1} ) & = F^{m}_1 , \
         \cdots\ , \ \Theta(\widetilde{F}_{n_1+\cdots + n_{m-1} + n_m})= F^{m}_{n_m}
     \end{align*}
     where $n_1+\cdots + n_{m-1} + n_m =n$. For the remaining facets of $\mathcal{C}^n_{I_1\cdots I_m}$, we
     define:
        \[  \Theta(\widetilde{F}^*_{I_i}) = F^i_0, \ 1\leq i \leq m .\]
      By the definition of $\mathcal{C}^n_{I_1\cdots I_m}$,
     it is easy to check that $\Theta$ induces an
      isomorphism between the face lattices of $\mathcal{C}^n_{I_1\cdots I_m}$
      and $\Delta^{n_1}\times\cdots \times \Delta^{n_m}$. In addition, since any
      face of $\mathcal{C}^n_{I_1\cdots I_m}$ is homeomorphic to a ball, so
      $\mathcal{C}^n_{I_1\cdots I_m}$ is homeomorphic to
       $\Delta^{n_1}\times\cdots \times\Delta^{n_m}$ as a manifold
       with corners.
   \end{proof}
      \n

      By Lemma~\ref{Lem:Induced-Char-Func}, if a generalized characteristic
    function $\widetilde{\lambda}$ on $\mathcal{C}^n_0$ is
    compatible with the map $\psi_{[\mathbf{f}_{I_1\cdots I_m}]} :
    \mathcal{S}_F(\mathcal{C}^n_0) \rightarrow \mathcal{S}_F(\mathcal{C}^n_{I_1\cdots I_m})$
    (see~\eqref{Equ:f-I}),
     then there exists an equivariant homeomorphism from $M(\mathcal{C}^n_0,\widetilde{\lambda})$
    to
     $M(\mathcal{C}^n_{I_1\cdots I_m}, \widetilde{\lambda}[\mathbf{f}_{I_1\cdots I_m}])$.
      \\

     \section{Getting generalized real Bott manifolds from cubes} \label{Section9}

      An $n$-manifold $M^n$ is called a \textit{generalized real Bott manifold}
      (see~\cite{SuyMasDong10}) if
    there is a finite sequence of fiber bundles
   \begin{equation} \label{Equ:Gene-Real-Bott}
     M^n=B_m \overset{\pi_m}{\longrightarrow} B_{m-1} \overset{\pi_{m-1}}{\longrightarrow}
    \cdots \overset{\pi_2}{\longrightarrow} B_1 \overset{\pi_1}{\longrightarrow} B_0
   =\ \{ \text{a point} \},
   \end{equation}
   where each $B_i$ ($1\leq i \leq m$) is the projectivization of
   the Whitney sum of a finite collection of real line bundles over $B_{i-1}$.
  So the fiber of each $\pi_i : B_i \rightarrow B_{i-1}$ is a finite dimensional real
  projective space. It is known that any generalized real Bott manifold
   is a small cover over some product of simplices (see the Remark $6.5$ in~\cite{SuyMasDong10}).
      Indeed,  suppose the fiber of the bundle $\pi_i : B_i \rightarrow B_{i-1}$
      in~\eqref{Equ:Gene-Real-Bott} is
      homeomorphic to $\R P^{n_i}$ ($n_i \geq 1$). Then $M^n$ is a small cover
      over $\Delta^{n_1}\times \cdots \Delta^{n_m}$ where $\Delta^{n_i}$ is the standard
      $n_i$-dimensional simplex and $n_1+\cdots + n_m =n$.
      \nn

      In section~\ref{Section6}, we have seen that any real Bott
      manifolds can be obtained from a facets-pairing structure $\mathcal{F}_A$ on
      a cube where $A$ is a binary square matrix with zero diagonal.
   It is natural to ask if we can obtain
      any generalized real Bott manifold from some $\mathcal{F}_A$ too.
      In this section, we will see that the answer is also yes. In fact,
        the set of closed manifolds that we can
     obtain from the seal spaces of such $\mathcal{F}_A$'s is exactly
   the set of all generalized real Bott manifolds
    (see Theorem~\ref{thm:Main-6} and Theorem~\ref{thm:Main-7}). Moreover, we will give
     the necessary and sufficient condition on a binary square matrix $A$ so that
     the seal space of $\mathcal{F}_A$ is a closed manifold (see Theorem~\ref{thm:Main-8}).
     \nn

     In the following, we will think of a generalized real Bott manifold $M^n$ as
      a small cover over $\Delta^{n_1}\times \cdots \times \Delta^{n_m}$
      where $n_1 + \cdots + n_m =n$, and denote the corresponding characteristic function on
      $\Delta^{n_1}\times \cdots \times \Delta^{n_m}$ by $\lambda_{M^n}$.
      By Theorem~\ref{thm:Smoothing-Cube}, we can identify
      $\Delta^{n_1}\times \cdots \times \Delta^{n_m}$
      with $\mathcal{C}^n_{I_1\cdots I_m}$, where
      $I_1,\cdots, I_m$ are given by~\eqref{Equ:Partition-n}. So we
       can think of $\lambda_{M^n}$ as a characteristic function on
     $\mathcal{C}^n_{I_1\cdots I_m}$, and we have
         $$M^n \cong M(\mathcal{C}^n_{I_1\cdots I_m},
        \lambda_{M^n}). $$
   By our discussion in section~\ref{Section8}, the facets of $\mathcal{C}^n_{I_1\cdots I_m}$ are
     $\{ \widetilde{F}_1,\cdots, \widetilde{F}_n, \widetilde{F}^*_{I_1},
     \cdots, \widetilde{F}^*_{I_m} \}$.
      Since $\lambda_{M^n}$ is non-degenerate, we can assume
      $\lambda_{M^n} (\widetilde{F}_j) =
      e_j$ for each $1\leq  j \leq n$, where $\{ e_1,\cdots, e_n \} $
        is a linear basis of $(\Z_2)^n$. And we suppose
        \[ \lambda_{M^n}(\widetilde{F}^*_{I_i}) = \mathbf{a}_i \in
            (\Z_2)^n, \ 1\leq i \leq m.
          \]
      Then we have an $m\times n$ binary matrix $\mathbf{\Lambda}$.
      \[ \mathbf{\Lambda} =  \begin{pmatrix}
        \mathbf{a}_1     \\
         \vdots  \\
         \mathbf{a}_m
   \end{pmatrix},\quad \text{where each}\ \mathbf{a}_i \in (\Z_2)^n.
   \]
   \begin{align*}
    \text{We write}\ \, \mathbf{a}_i & =(\mathbf{a}^1_i,\cdots, \mathbf{a}^j_i, \cdots,
      \mathbf{a}^m_i) \\
        & = ([a^1_{i 1},\cdots, a^1_{i n_1}], \cdots, [a^j_{i 1},\cdots, a^j_{i n_j}],\cdots,
         [a^m_{i 1}, \cdots, a^m_{i n_m}]),
   \end{align*}
    where $\mathbf{a}^j_i =  [a^j_{i 1}, \cdots, a^j_{i n_j}]  \in (\Z_2)^{n_j}$
    for each $j=1,\cdots, m$. Then we have:
   \begin{align}
      \mathbf{\Lambda} &= \begin{pmatrix}
        \mathbf{a}_1     \\
         \vdots  \\
         \mathbf{a}_m
   \end{pmatrix} = \begin{pmatrix}
        \mathbf{a}^1_1 & \cdots & \mathbf{a}^m_1     \\
         \vdots & \cdots  & \vdots \\
         \mathbf{a}^1_m & \cdots & \mathbf{a}^m_m
   \end{pmatrix}  \label{Equ:Vector-Matrix}  \notag \\
      &=  \begin{pmatrix}
        a^1_{11} & \cdots &  a^1_{1 n_1} & \cdots   &  a^m_{11} & \cdots & a^m_{1 n_m}   \\
        \vdots & \cdots  & \vdots  & \cdots  & \vdots & \cdots  & \vdots \\
        a^1_{m 1} & \cdots & a^1_{m n_1} & \cdots & a^m_{m 1} & \cdots &  a^m_{m n_m}
   \end{pmatrix}.
   \end{align}
   So the matrix $\mathbf{\Lambda}$ can be viewed as an $m\times m$ matrix
   whose entries in the $j$-th column are vectors in $(\Z_2)^{n_j}$.
   Such a matrix $\mathbf{\Lambda}$ is called a \textit{vector matrix} (see~\cite{SuyMasDong10}).
   In addition, for given $1\leq k_j \leq n_j$, $j=1,\cdots, m$, let
    $\mathbf{\Lambda}_{k_1\cdots k_m}$ be the $m\times m$ submatrix of $\mathbf{\Lambda}$
    whose $j$-th column is the $k_j$-th column of the following $m\times n_j$
    matrix.
    \[   \begin{pmatrix}
         \mathbf{a}^j_1  \\
         \vdots  \\
         \mathbf{a}^j_m
   \end{pmatrix} =
      \begin{pmatrix}
        a^j_{11} & \cdots & a^j_{1 k_j} & \cdots & a^j_{1 n_j}   \\
        \vdots &  \  & \vdots  & \ & \vdots  \\
        a^j_{m 1} & \cdots & a^j_{m k_j} & \cdots & a^j_{m n_j}
   \end{pmatrix}
   \]
   \begin{equation} \label{Equ:Principal-minor}
     \text{So we have:} \ \;
       \mathbf{\Lambda}_{k_1\cdots k_m} =  \begin{pmatrix}
           a^1_{1 k_1} & \cdots & a^m_{1 k_m}    \\
           \vdots & \  & \vdots  \\
         a^1_{m k_1} & \cdots & a^m_{m k_m}
    \end{pmatrix} .  \qquad\qquad \qquad \qquad
    \end{equation}
  \nn

   A \textit{principal minor} of the $m\times m$ vector matrix $\mathbf{\Lambda}$
    in~\eqref{Equ:Vector-Matrix} means a principal minor of an
   $m\times m$ matrix $\mathbf{\Lambda}_{k_1\cdots k_m}$ for some
   $1\leq k_1 \leq n_1$, $\cdots, 1\leq k_m \leq n_m$. And the
   determinant of $\mathbf{\Lambda}_{k_1\cdots k_m}$ itself is also considered as
    a principal minor of $\mathbf{\Lambda}$. \nn

    The lemma 3.2
    in~\cite{SuyMasDong10} says that $\lambda_{M^n}$ is
    non-degenerate at all vertices of $\Delta^{n_1}\times\cdots \times \Delta^{n_m}$
    is exactly equivalent to all principal minors of $\mathbf{\Lambda}$ being $ 1$.
    This implies:\n

    \begin{itemize}

      \item[(c1)] $\mathbf{a}_1,\cdots , \mathbf{a}_m$ are pairwise different.\n

      \item[(c2)] in the vector $\mathbf{a}_i = (\mathbf{a}^1_i,\cdots, \mathbf{a}^m_i)$,
    we must have $\mathbf{a}^i_i = (1,1,\cdots , 1)$ for any $1\leq i \leq m$.
    \end{itemize}

    Now, from the $\mathbf{\Lambda}$ in~\eqref{Equ:Vector-Matrix},
     we define an $n\times n$ binary matrix $\widetilde{A}$
     by: the first row to $n_1$-th row vectors of $\widetilde{A}$ are all
     $\mathbf{a}_1$, the $(n_1 +1)$-th row to $(n_1 + n_2)$-th row vectors
     of $\widetilde{A}$ are all $\mathbf{a}_2$, $\cdots$, the $(n_1+\cdots + n_{m-1}
     +1)$-th row to the $n$-th row vectors of $\widetilde{A}$ are all $\mathbf{a}_m$.
     Then the condition (c2) above implies that all the diagonal entries of $\widetilde{A}$ are
     $1$. For this $\widetilde{A}$, define
      \begin{equation} \label{Equ:Def-A}
        A = \widetilde{A} - I_n .
       \end{equation}
   So $A$ is an $n\times n$ binary matrix with zero diagonal.
     \nn

     \begin{thm} \label{thm:Main-6}
      For a given generalized real Bott manifold $M^n$,
      the matrix $A$ defined in~\eqref{Equ:Def-A} satisfies:
      \begin{itemize}
        \item[(i)] $\mathcal{F}_A$ is a strong regular facets-pairing structure on
       $\mathcal{C}^n$, and
        \item[(ii)] the seal space $Q^n_{\mathcal{F}_A}$ is
       homeomorphic to $M^n$.
      \end{itemize}

   \end{thm}
     \begin{proof}
        To show $\mathcal{F}_A$ is strong,
        it suffices by Theorem~\ref{thm:Main-5} to show that:
        for any $1\leq j_1 < \cdots < j_s \leq n$,
         $\mathrm{rank}_{\Z_2}(\widetilde{A}^{j_1\cdots  j_s}_{j_1\cdots j_s})
    = \mathrm{rank}_{\Z_2}(\widetilde{A}^{j_1\cdots j_s})$.
   First, we assume the $j_1$-th,$\cdots, j_s$-th row vectors of
   $\widetilde{A}$ are pairwise different. Then by our construction of $\widetilde{A}$,
    the integers $j_1 ,\cdots, j_s$ must lie in different intervals below:
    $$ [1, n_1],\ [n_1+1 ,n_2 ],\ \cdots\ ,[n_1+\cdots+ n_{m-1} + 1 , n ].$$
   Then it is not hard to see that
   $\widetilde{A}^{j_1\cdots  j_s}_{j_1\cdots j_s}$ can be realized as a principal
   minor matrix of $\mathbf{\Lambda}$, hence $\det(\widetilde{A}^{j_1\cdots  j_s}_{j_1\cdots j_s})=1$.
   So $\mathrm{rank}_{\Z_2}(\widetilde{A}^{j_1\cdots  j_s}_{j_1\cdots j_s})
    = \mathrm{rank}_{\Z_2}(\widetilde{A}^{j_1\cdots j_s})= s$.
   If the $j_1$-th,$\cdots, j_s$-th row vectors of
   $\widetilde{A}$ are not pairwise different, we
    assume that the $j_{i_1}$-th, $\cdots , j_{i_r}$-th row vectors
    of $\widetilde{A}$ are all the different vectors among them.
   Then by repeating the above argument, we can easily see that
   $\mathrm{rank}_{\Z_2}(\widetilde{A}^{j_1\cdots  j_s}_{j_1\cdots j_s})
    = \mathrm{rank}_{\Z_2}(\widetilde{A}^{j_1\cdots j_s})=r$. So
   $\mathcal{F}_A$ is a strong facets-pairing structure on $\mathcal{C}^n$. \n

   Next, we show that the seal space $Q^n_{\mathcal{F}_A}$
   is homeomorphic to $M^n$. In fact, by the definition of $A$, the generalized
   characteristic function $\lambda_A$ on $\mathcal{C}^n_0$ defined
   by~\eqref{Equ:Char-Function-1} and~\eqref{Equ:Char-Function-2}
   satisfies: $\lambda_A(\overline{F}^*_l) = \mathbf{a}_i = \lambda_{M^n}(\widetilde{F}^*_{I_i})$
    for $\forall\, l\in I_i$. So $\lambda_A$ is compatible with the
    map $\psi_{[\mathbf{f}_{I_1\cdots I_m}]} :
   \mathcal{S}_F(\mathcal{C}^n_0) \rightarrow \mathcal{S}_F(\mathcal{C}^n_{I_1\cdots
   I_m})$ where $\mathbf{f}_{I_1\cdots I_m}$ is defined
   by~\eqref{Equ:f-I}. Obviously, the induced function
   $\lambda_A[\mathbf{f}_{I_1\cdots I_m}]$ on $\mathcal{C}^n_{I_1\cdots I_m}$ from $\lambda_A$
   coincides with $\lambda_{M^n}$.
   So we have
    $$Q^n_{\mathcal{F}_A} \overset{\mathrm{Lem\,}\ref{Lem:Equiv-Equivalence}}{\cong} M(\mathcal{C}^n_0,
    \lambda_A) \overset{\mathrm{Lem}\ref{Lem:Induced-Char-Func}}{\cong}
     M(\mathcal{C}^n_{I_1\cdots I_m}, \lambda_A[\mathbf{f}_{I_1\cdots I_m}]) =
  M(\mathcal{C}^n_{I_1\cdots I_m}, \lambda_{M^n}) \cong M^n . $$
   Moreover, since the above homeomorphisms are all equivariant, so $Q^n_{\mathcal{F}_A}$ with the
  action of $H$ is equivariantly homeomorphic to $M^n$.
  \end{proof}
      \n

  For an arbitrary binary square matrix $A$ with zero diagonal,
   assuming $\mathcal{F}_A$ is strong can not guarantee the seal space
   $Q^n_{\mathcal{F}_A}$ is a closed manifold (see
   Example~\ref{Exam:Strang-Strong-2}). But if $Q^n_{\mathcal{F}_A}$ is
   a closed manifold, the following theorem asserts that
   $\mathcal{F}_A$ must be a strong regular facets-pairing structure and
    $Q^n_{\mathcal{F}_A}$ must be a generalized real Bott manifold. \nn

   \begin{thm} \label{thm:Main-7}
   For an $n\times n$ binary matrix $A$ with zero diagonal,
  if the seal space $Q^n_{\mathcal{F}_A}$ is a closed manifold,
   we must have:
  \begin{itemize}
    \item[(i)] $\mathcal{F}_A$ is a strong facets-pairing structure on $\mathcal{C}^n$,
    and\n
    \item[(ii)] $Q^n_{\mathcal{F}_A}$ is homeomorphic to a generalized real Bott manifold.
  \end{itemize}
   \end{thm}
   \begin{proof}
    By our discussion at the beginning of this section, we can identify
       $Q^n_{\mathcal{F}_A}$ with $M(\mathcal{C}^n_0,\lambda_A)$.
      If $Q^n_{\mathcal{F}_A}$ is a manifold,
       then at any vertex $u_{j_1\cdots j_s}$ of $\mathcal{C}^n_0$,
        $\lambda_A$
       must satisfy the condition in Lemma~\ref{Lem:General-Char-Func-2}.
      In particular, at the vertex $u_{12\cdots n}$, all the facets of $\mathcal{C}^n_0$
      meeting $u_{12\cdots n}$ are $\overline{F}^*_1,\cdots, \overline{F}^*_n$. So
     Lemma~\ref{Lem:General-Char-Func-2}
       implies that by re-indexing the coordinates of $\R^n$,
    we can assume:
    \begin{align*}
      \lambda_A(\overline{F}^*_1) = \cdots & = \lambda_A(\overline{F}^*_{n_1}) = \mathbf{a}_1  , \\
     \lambda_A(\overline{F}^*_{n_1+1}) = \cdots & = \lambda_A(\overline{F}^*_{n_1+n_2}) = \mathbf{a}_2 , \\
      \cdots  \qquad & \cdots \qquad \cdots \\
     \lambda_A(\overline{F}^*_{n_1+\cdots + n_{m-1}}) = \cdots
      &= \lambda_A(\overline{F}^*_{n_1+\cdots + n_{m-1}+ n_m})
     = \mathbf{a}_m,
     \end{align*}
   where $ n_1 + \cdots + n_m = n$ and $\mathbf{a}_1,\cdots, \mathbf{a}_m$
   are linearly independent elements of $(\Z_2)^n$. Then let
   $I_1,\cdots, I_m$ be the partition of $\{ 1,\cdots, n\}$
   defined by~\eqref{Equ:Partition-n} and $\mathbf{f}_{I_1\cdots I_m}$ be the set of
   codimension-two faces of $\mathcal{C}^n$ defined
   by~\eqref{Equ:f-I}. Obviously, $\lambda_A$ is compatible with
   the smooth of $\mathcal{C}^n_0$ along $\mathbf{f}_{I_1\cdots I_m}$. So by
    Lemma~\ref{Lem:Induced-Char-Func},
   $M(\mathcal{C}^n_0,\lambda_A)$ is homeomorphic to
   $M(\mathcal{C}^n_{I_1\cdots I_m}, \lambda_A[\mathbf{f}_{I_1\cdots I_m}])$,
    where $\lambda_A[\mathbf{f}_{I_1\cdots I_m}]$
   is the induced function by $\lambda_A$ on the facets of $\mathcal{C}^n_{I_1\cdots I_m}$.
   We have:
          $$  \lambda_A[\mathbf{f}_{I_1\cdots I_m}](\widetilde{F}^*_{I_i}) =
           \mathbf{a}_i, \ 1\leq i \leq m \ \,
           \mathrm{(see~\eqref{Equ:Induced-Function})}.
          $$
    \[ \mathrm{Let}\ \ \mathbf{\Lambda}_A =  \begin{pmatrix}
        \mathbf{a}_1     \\
         \vdots  \\
         \mathbf{a}_m
     \end{pmatrix},\quad \text{where each}\ \mathbf{a}_i \in (\Z_2)^n.
   \]

   By Theorem~\ref{thm:Smoothing-Cube},
   $\mathcal{C}^n_{I_1\cdots I_m} \cong \Delta^{n_1}\times \cdots \times \Delta^{n_m}$.
   So to prove $Q^n_{\mathcal{F}_A}$ is homeomorphic to a generalized real Bott manifolds, it suffices to
   show that the function $\lambda_A[\mathbf{f}_{I_1\cdots I_m}]$ is non-degenerate at all vertices of
   $\mathcal{C}^n_{I_1\cdots I_m}$. Recall that all the facets
   of $\mathcal{C}^n_0$ meeting at a vertex $u_{j_1\cdots j_s}$ are:
    $$ \overline{F}^*_{j_1}, \cdots,
      \overline{F}^*_{j_s}, \overline{F}_{l_1} , \cdots, \overline{F}_{l_{n-s}},
      \  \text{where}\ \{ l_1,\cdots, l_{n-s}\}
       = \{ 1,\cdots, n \} \backslash \{ j_1,\cdots, j_s \}
      $$
   A critical observation here is that: $\lambda_A (\overline{F}^*_j) \neq
   \lambda_A(\overline{F}_l)$ for $\forall\, j\in \{ j_1,\cdots, j_s \}$ and $\forall\, l\in \{ l_1,\cdots,
   l_{n-s}\}$ (see the definition of $\lambda_A$ in~\eqref{Equ:Char-Function-1}
    and~\eqref{Equ:Char-Function-2}).  So by the definition of
    $\lambda_A[\mathbf{f}_{I_1\cdots I_m}]$,
   at any vertex $\widetilde{v}$ of
    $\mathcal{C}^n_{I_1\cdots I_m}$
    the value of $\lambda_A[\mathbf{f}_{I_1\cdots I_m}]$ on all the
    facets meeting at $\widetilde{v}$ are pairwise
    distinct.
    $$ \text{Then since}\ M(\mathcal{C}^n_{I_1\cdots I_m}, \lambda_A[\mathbf{f}_{I_1\cdots I_m}])
    \cong M(\mathcal{C}^n_0,\lambda_A) \cong Q^n_{\mathcal{F}_A} \ \text{is a closed manifold},
     $$
    Corollary~\ref{Cor:Mfd-General-Char-Func} asserts that
    $\lambda_A[\mathbf{f}_{I_1\cdots I_m}]$ must be non-degenerate at any vertex $\widetilde{v}$
    of $\mathcal{C}^n_{I_1\cdots I_m}$. So
     $Q^n_{\mathcal{F}_A}$ is homeomorphic to a generalized real Bott manifolds.\n

   Moreover, the non-degeneracy of  $\lambda_A[\mathbf{f}_{I_1\cdots I_m}]$ at all vertices
   of $\mathcal{C}^n_{I_1\cdots I_m}$ implies that
   all the principal minors of $\mathbf{\Lambda}_A$
  are $1$, where $\mathbf{\Lambda}_A$ is considered as an $m\times m$ vector matrix.
   Then for the
   matrix $\widetilde{A}= A+I_n$, we can repeat the
   argument in the proof of Theorem~\ref{thm:Main-6} to show that
  for any $1\leq j_1 < \cdots < j_s \leq n$,
   $\mathrm{rank}_{\Z_2}(\widetilde{A}^{j_1\cdots  j_s}_{j_1\cdots j_s})
    = \mathrm{rank}_{\Z_2}(\widetilde{A}^{j_1\cdots j_s})$.
  So by Theorem~\ref{thm:Main-5}, we can conclude that $\mathcal{F}_A$
  is a strong facets-pairing structure on $\mathcal{C}^n$.
   \end{proof}

  By Theorem~\ref{thm:Main-6} and
   Theorem~\ref{thm:Main-7}, we see that the set of closed manifolds that
  occur as the seal spaces of
   $\mathcal{F}_A$'s are exactly all the generalized real Bott
  manifolds. This gives another reason why generalized real Bott manifolds
  are naturally the ``extension'' of real Bott manifolds.
   In addition, from the proof of
  Theorem~\ref{thm:Main-6} and
   Theorem~\ref{thm:Main-7}, we can easily see the following.\nn

 \begin{thm}\label{thm:Main-8}
    For any $n\times n$ binary matrix $A$ with zero diagonal, the space $Q^n_{\mathcal{F}_A}$
    is a closed manifold if and only if the matrix $\widetilde{A} =  A+ I_n$
    satisfies:\n
       \begin{enumerate}
         \item for any $1\leq j_1 < \cdots < j_s \leq n$,
           $\mathrm{rank}_{\Z_2}(\widetilde{A}^{j_1\cdots  j_s}_{j_1\cdots j_s})
    = \mathrm{rank}_{\Z_2}(\widetilde{A}^{j_1\cdots j_s})$.\n

         \item For any set of row vectors $\alpha_1,\cdots, \alpha_s$ of $\widetilde{A}$, if
         $\alpha_1,\cdots, \alpha_s$
          are pairwise different, then they are linearly independent over
            $\Z_2$.
       \end{enumerate}
 \end{thm}
 \n

  In addition, if we view a generalized real Bott manifold as the seal space of a
  facets-pairing structure $\mathcal{F}_A$ on a cube, we can easily tell when
  it is orientable from the information of $A$. The following
  theorem generalizes the Lemma 2.2 in~\cite{MasKam09-1}. \nn

  \begin{thm} \label{thm:orientability}
     For an $n$-dimensional generalized real Bott manifold
      $Q^n_{\mathcal{F}_A}$ where $A$ is an $n\times n$ binary matrix with zero diagonal
      entries,  $Q^n_{\mathcal{F}_A}$ is orientable if and only if the
      sum of the entries in each row vector of $A$ is zero over $\Z_2$.
  \end{thm}
  \begin{proof}
   Let $\{ \tau^A_j : \mathbf{F}(j) \rightarrow \mathbf{F}(-j) \}_{j\in [\pm n]}$ be the structure
   maps of the facets-pairing structure $\mathcal{F}_A$. Given an orientation of $\mathcal{C}^n$,
   we can orient each facet of $\mathcal{C}^n$ by the outward
   normal. So $Q^n_{\mathcal{F}_A}$ is orientable
   if and only if each $\tau^A_j$ is orientation-reversing with
   respect to the induced orientation on $\mathbf{F}(j)$ and
   $\mathbf{F}(-j)$. It is easy to see that this is equivalent to
   requiring the sum of all the entries in each row vector of $A$ to be
   zero.
  \end{proof}
  \n

\begin{rem}
   It was shown in~\cite{MasKam09-1} that two real Bott manifolds
   are homeomorphic if and only if their $\Z_2$-cohomology rings are
   isomorphic. This is called \textit{cohomological rigidity} of real
   Bott manifolds. But for generalized real Bott manifolds,
   cohomological rigidity does not hold (see~\cite{Masuda2010}). In fact, it was shown in
   ~\cite{Masuda2010} that there exist two
   generalized real Bott manifolds with the same $\Z_2$-cohomology rings,
   Stiefel-Whitney classes and homotopy groups, but they are not homeomorphic.  So
   we need to use extra topological invariants
   to distinguish the homeomorphism types of generalized real Bott
    manifolds. A possible candidate is the integral homology groups.
    But in general, the calculation of integral homology groups for small
    covers is quite difficult.
 \end{rem}

 \n

 \noindent \textbf{Summary:} Let $A$ denote any $n\times n$ binary matrix with zero diagonal.
  According to our discussion in this paper,
   we have the following correspondence between the facets-pairing
  structures
  $\mathcal{F}_A$ on $\mathcal{C}^n$ and the associated seal spaces
  $Q^n_{\mathcal{F}_A}$.

  \[ \ \ \,\text{Facets pairing structure $\mathcal{F}_A$ on $\mathcal{C}^n$} \qquad \qquad \ \qquad
  \qquad  \text{Seal space $Q^n_{\mathcal{F}_A}$} \qquad \qquad  \qquad \,\ \  \ \
  \]
  \begin{align*}
    \text{$\mathcal{F}_A$ is strong }\ \qquad\qquad\quad \
   & \longleftrightarrow \qquad  \quad\ \text{Real toric orbifolds} \\
          \bigcup \ \ \qquad \qquad\qquad\quad\ & \qquad
         \qquad\qquad \quad \ \ \ \ \ \,\,\ \bigcup \\
    \text{$\mathcal{F}_A$ is strong and $Q^n_{\mathcal{F}_A}$ is a manifold} \  \ \
    & \longleftrightarrow \ \ \ \text{Generalized real Bott manifolds} \\
       \bigcup \ \ \qquad \qquad\qquad\quad\ & \qquad
         \qquad\qquad \quad \ \ \ \ \ \,\,\ \bigcup \\
       \text{$A$ is a Bott matrix $\Longleftrightarrow \mathcal{F}_A$ is perfect} \
    & \longleftrightarrow
    \ \text{Real Bott manifolds (Rieman. flat)}
  \end{align*}
 \nn

  \n

 \begin{rem}
 For two
  non-equivalent facets-pairing structures $\mathcal{F}_{A_1}$ and
  $\mathcal{F}_{A_2}$, it is possible that their seal spaces  $Q^n_{\mathcal{F}_{A_1}}$ and  $Q^n_{\mathcal{F}_{A_2}}$
  are homeomorphic.
  For example when $A_1, A_2 \in \mathfrak{B}(n)$, Theorem~\ref{thm:F_A-Classify} says that
   $\mathcal{F}_{A_1}$ and $\mathcal{F}_{A_2}$ are equivalent if and only if
   $A_1$ is conjugate to $A_2$ by a permutation matrix. But it was shown in~\cite{SuMasOum10}
  that the seal space $Q^n_{\mathcal{F}_{A_1}}$
   and $Q^n_{\mathcal{F}_{A_2}}$ (two real Bott manifolds) are homeomorphic if and only
  if $A_1$ can be transformed to $A_2$ via three
  types of matrix operations. The conjugation by
  a permutation matrix is just one of the three types of matrix
  operations.\nn

   Moreover, it is interesting to know what kind of closed manifolds (e.g.
  closed flat Riemannian manifolds)
   we can obtain from
   facets-pairing structures on cubes other than the $\mathcal{F}_A$'s.
  Of course, the answer should be much harder than the $\mathcal{F}_A$ case.
  We will study more about this problem in a subsequent paper.
  \end{rem}

 \noindent  \textbf{Acknowledgement:} The author wants to thank professor
      Mikiya Masuda for some helpful discussions.

\end{document}